\documentclass[a4paper,12pt]{amsart}
\usepackage{hyperref}
\usepackage{bm}
\usepackage{mathtools}
\usepackage{graphicx}
\usepackage{amsmath,amssymb,amsthm,amsfonts}
\usepackage{amssymb}
\usepackage[active]{srcltx}

\usepackage{color}
\newtheorem{thm}{Theorem}[section]
\newtheorem{cor}[thm]{Corollary}
\newtheorem{lem}[thm]{Lemma}
\newtheorem{prop}[thm]{Proposition}
\newtheorem{rem}[thm]{Remark}
\newtheorem{Defi}{Definition}[section]
\newtheorem{Exam}[thm]{Example}

\usepackage{tikz}
\usepackage{standalone}
\usepackage{graphicx}
\usepackage{subcaption}
\usepackage{multicol,float}

\newcommand{\disp}{\displaystyle}

\numberwithin{equation}{section}

\newcommand{\R}{\mathbb R}
\newcommand{\Z}{\mathbb Z}

\newcommand{\cic}{C^{\infty}_{\text{c}}}

\newcommand{\dr}{\partial}

\begin{document}

\parindent 0pc
\parskip 6pt
\overfullrule=0pt

\title{\quad Monotonicity for solutions to semilinear problems in epigraphs}

\author{Nicolas Beuvin$^{\dagger}$, Alberto Farina$^{\dagger}$,  Berardino Sciunzi$^{*}$}
\address{$^{\dagger}$  Universit\'e de Picardie Jules Verne, LAMFA, CNRS UMR 7352, 33, rue Saint-Leu 80039 Amiens, France}\email{nicolas.beuvin@u-picardie.fr, alberto.farina@u-picardie.fr}
\address{$^{*}$Dipartimento di Matematica e Informatica, UNICAL, Ponte Pietro  Bucci 31B, 87036 Arcavacata di Rende, Cosenza, Italy}
\email{sciunzi@mat.unical.it}




\maketitle


\begin{abstract}
We consider positive solutions, possibly unbounded, to the semilinear equation $-\Delta u=f(u)$ on continuous epigraphs bounded from below. Under the homogeneous Dirichlet boundary condition, we prove new monotonicity results for $u$, when $f$ is a (locally or globally) Lipschitz-continuous function satisfying $ f(0) \geq 0$. As an application of our new monotonicity theorems, we prove some classification and/or non-existence results. 
To prove our results, we first establish some new comparison principles for semilinear problems on general unbounded open sets of $\R^N$, and then we use them to start and to complete a modified version of the moving plane method adapted to the geometry of the epigraph $\Omega$. \\
As a by-product of our analysis, we also prove some new results of uniqueness and symmetry for solutions (possibly 
unbounded and sign-changing) to the homogeneous Dirichlet BVP for the semilinear Poisson equation in fairly general unbounded domains. 

\end{abstract}

\section{Introduction and main results}

We consider solutions, possibly unbounded, to the problem
\begin{equation}\label{NonLin-PoissonEq}
\begin{cases}
-\Delta u=f(u) & \text{ in } \Omega,\\
\quad u >0 & \text{ in } \Omega,\\
\quad u=0\,\, &\text{ on } \partial\Omega,
\end{cases}
\end{equation}
where $f$ is a (locally or globally) Lipschitz-continuous function satisfying $$f(0)\geq 0,$$ and $\Omega$ is an epigraph of 
$\mathbb{R}^N$, with $ N \geq 2$, i.e.  
\[
\Omega\,:=\,\{x=(x',x_N) \in\mathbb{R}^{N-1}\times\mathbb{R} \,:\, x_N> g(x')\}\,,
\]
where $g\,:\, \mathbb{R}^{N-1}\rightarrow \mathbb{R}$ is a continuous function bounded from below. 

\smallskip

The main results of the present paper prove that the solution $u$ is strictly increasing in the $x_N$-direction, i.e., $\frac{\partial u}{\partial x_N}>0$ in $\Omega$. Our monotonicity results are new. They cover the case of uniformly continuous epigraphs (not necessarily locally Lipschitz-continuous), coercive epigraphs, as well as that of a large family of merely continuous epigraphs (with possibly arbitrary growth at infinity and not necessarily coercive).  Furthermore, we do not assume that $u$ is bounded.
\\ 
Half-spaces and coercive epigraphs are a very special case of the geometries covered by our analysis.
Thus, our results recover, extend and complete existing ones, which mainly deal with the case of a half-space (\cite{BCNMagenes}-\cite{BCN3},\cite{DS3},\cite{Dancer},\cite{Far-Sandro},\cite{FaSciunzi},\cite{FaSciunzi2},\cite{fval},\cite{QS}) or a locally Lipschitz-continuous coercive epigraph (\cite{BCN3},\cite{Esteban-Lions},\cite{FaSciunzi}).

Our results apply to classical solutions of \eqref{NonLin-PoissonEq} as well as to those in the sense of distributions. 
To deal with the latter case and settle the general framework in which we work, we introduce the following functional space : 
\begin{equation}\label{Def-H1loc-Adherence-Omega}
H^1_{loc}(\overline{\Omega}) = \{ u : \Omega \mapsto \R,  \,\,  \text{$u$ Lebesgue-mesurable} \, : \, u \in H^1(\Omega \cap B(0,R)) \quad \forall \, R>0 \}.\footnote{ \, i.e., $u$ is Lebesgue-measurable on $\Omega$ and $u \in H^1(U)$ for any open bounded set $ U \subset \Omega$.}
\end{equation}

The choice of the distributional framework is quite natural and justified by the fact that, in many cases, the epigraphs we consider are merely continuous (i.e., the functions $g$ that define them are assumed to be continuous and nothing more). 

\smallskip

Let us first state the main results for uniformly continuous epigraphs. 

\begin{thm}\label{TH1}
Let $\Omega$ be a uniformly continuous epigraph bounded from below. \\ 
Assume $f \in {Lip}_{loc}([0,+\infty))$ with 
\begin{equation}\label{cond-in-zero}
          \liminf_{t \to 0^+} \frac{f(t)}{t} > 0
\end{equation}
and let $u \in C^0(\overline{\Omega}) \cap H^1_{loc}(\overline{\Omega})$ 
be a distributional solution to \eqref{NonLin-PoissonEq} which is uniformly continuous on finite strips.\footnote{\, i.e., for any $R>0$, $u$ is uniformly continuous on the strip $\Omega \cap \left \lbrace x_N <R \right\rbrace$. }\\
Then $u$ is strictly increasing in the $x_N$-direction, i.e., $\frac{\partial u}{\partial x_N}>0$ in $\Omega$.\footnote{ \, Continuous distributional solutions of $-\Delta u=f(u) $ in $\Omega$ belong to $C^2(\Omega),$ by standard elliptic theory.} 
\end{thm}

\smallskip

The epigraphs considered in the previous result could be very wild, as shown by the following two-dimensional examples constructed by using the Weierstrass-type functions $ g_{b,\alpha}(x) = \sum_{n=1}^{\infty} b^{-n \alpha} \cos(b^n \pi x)$, where $b>1$ is an integer and $ \alpha \in (0,1)$. The function $g_{b,\alpha}$ is uniformly continuous, bounded and nowhere differentiable (\cite{Hardy}). 

\smallskip

If we further assume that the epigraph satisfies a weak regularity assumption, we can prove the monotonicity result for any (locally or globally) Lipschitz-continuous function $f$ satisfying $ f(0) \geq 0$ (and this by also weakening the assumptions on $u$). 

\smallskip

\begin{thm}\label{TH2bis}
Let $\Omega$ be a uniformly continuous epigraph bounded from below and satisfying a uniform exterior cone condition.
\\ 
Assume $ f \in {Lip}_{loc}([0,+\infty))$ with $f(0)\geq0$ and let $u \in C^0(\overline{\Omega}) \cap H^1_{loc}(\overline{\Omega})$ be a distributional solution to \eqref{NonLin-PoissonEq} which is bounded on finite strips.\footnote{\, i.e., for any $R>0$, 
\begin{equation}\label{Hyp-TH2bis}
\sup_{\Omega \cap \left \lbrace x_N <R \right\rbrace } u \,\, < + \infty.
\end{equation}
}\\
Then $u$ is strictly increasing in the $x_N$-direction, i.e., $\frac{\partial u}{\partial x_N}>0$ in $\Omega$.
\end{thm}

\begin{thm}\label{TH2}
Let $\Omega$ be a uniformly continuous epigraph bounded from below and satisfying a uniform exterior cone condition.
\\ 
Assume $ f \in {Lip}([0,+\infty))$ with $f(0) \geq 0$ and let $u \in C^0(\overline{\Omega}) \cap H^1_{loc}(\overline{\Omega})$ be a distributional solution to \eqref{NonLin-PoissonEq} with at most exponential growth on finite strips.\footnote{ \, i.e., for any $R>0$, there are positive numbers $A=A(R),B=B(R)$ such that 
$$ u(x) \leq Ae^{B\mid x \mid} \qquad  \forall \, x \in \Omega \cap \left \lbrace x_N <R \right\rbrace .$$} \\ 
Then $u$ is strictly increasing in the $x_N$-direction, i.e., $\frac{\partial u}{\partial x_N}>0$ in $\Omega$.
\end{thm}

\smallskip

Note that Theorem \ref{TH2bis} and Theorem \ref{TH2} cover the case of uniformly continuous epigraphs bounded from below, which are not necessarily locally Lipschitz-continuous. See Example \ref{Ex1} in Section \ref{SectEXAMPLES}.

\smallskip

Although the following results are special cases of the preceding theorems, even in this weaker form, they are new.

\begin{cor}\label{Cor1}
Let $\Omega$ be a uniformly continuous epigraph bounded from below. 
Assume $ f \in {Lip}_{loc}([0,+\infty))$ with 
\begin{equation}\label{cond-in-zero}
          \liminf_{t \to 0^+} \frac{f(t)}{t} > 0
\end{equation}
 and let $u \in C^2(\Omega) \cap C^0(\overline{\Omega})$ be a classical solution of \eqref{NonLin-PoissonEq} such 
 that $\nabla u \in L^{\infty}(\Omega)$. 
Then $u$ is strictly increasing in the $x_N$-direction, i.e., $\frac{\partial u}{\partial x_N}>0$ in $\Omega$.
\end{cor}

\begin{cor}\label{Cor2}
Let $\Omega$ be a globally Lipschitz-continuous epigraph bounded from below. 
Assume $ f \in {Lip}_{loc}([0,+\infty))$ with $f(0)\geq 0$ and let $u \in C^2(\Omega) \cap C^0(\overline{\Omega})$ be a classical solution of \eqref{NonLin-PoissonEq} such that $\nabla u \in L^{\infty}(\Omega)$. 
Then $u$ is strictly increasing in the $x_N$-direction, i.e., $\frac{\partial u}{\partial x_N}>0$ in $\Omega$.
\end{cor}

\begin{cor}\label{Cor3}
Assume $ \alpha \in (0,1)$ and let $\Omega$ be a globally Lipschitz-continuous epigraph bounded from below with 
$g \in C^{1, \alpha}_{loc}(\R^{N-1}).$ 
Assume $ f \in {Lip}_{loc}([0,+\infty))$ with $f(0)\geq0$ and let $u \in C^2(\Omega) \cap C^0(\overline{\Omega})$ be a classical solution of \eqref{NonLin-PoissonEq} which is bounded on finite strips.
Then $u$ is strictly increasing in the $x_N$-direction, i.e., $\frac{\partial u}{\partial x_N}>0$ in $\Omega$.
\end{cor}
 
When the epigraph is coercive, the monotonicity result holds under the sole assumption of continuity of $g$, i.e., we do not need to require either uniform continuity or the uniform exterior cone condition for the epigraph. Moreover, this result holds regardless of the value of $f(0)$. More precisely, we have
 
\begin{thm}\label{TH2coerc}
Let $\Omega$ be a coercive continuous epigraph. Assume $ f \in {Lip}_{loc}([0,+\infty))$ and let $u \in C^0(\overline{\Omega}) \cap H^1_{loc}(\overline{\Omega})$ be a distributional solution to \eqref{NonLin-PoissonEq}. 
Then $u$ is strictly increasing in the $x_N$-direction, i.e., $\frac{\partial u}{\partial x_N}>0$ in $\Omega$.
\end{thm}

To prove our main results, we first establish some new comparison principles for semilinear problems on general unbounded open sets of $\R^N$, 
and then we use them to start and to complete a modified version of the moving plane method adapted to the geometry of the epigraph $\Omega$. Then, by a delicate analysis based on the translation invariance of the equation and on some fine compactness and regularity arguments, we obtain the desired results of monotonicity.

The flexibility of our methods also allows to prove various extensions of our main results to a large class of merely continuous epigraphs bounded from below (see the class $\mathcal{G}$ introduced in definition \ref{Def-classG} and the list of examples that follows it). 
Those general results (see Theorems \ref{TH_fonc_G}, \ref{TH_fonc_G-bis} and \ref{TH3-disc}) are stated, discussed and proven in Section \ref{SectEXT}. 
Below we illustrate them  with some particularly evocative examples. 

\begin{thm}\label{TH_fonc_G-spec}
Let $N \geq2$ and let $ g : \R^{N-1} \mapsto \R$ be a continuous function such that : 

\begin{enumerate}
\item $N=2 \,\,  $ and $ \quad \lim_{x \mapsto -\infty}g(x)\in (-\infty , + \infty], \quad \lim_{x \mapsto +\infty} g(x) \in (-\infty , + \infty];$
\item $N=2 \,\,  $ and $g$ is quasiconvex (resp. quasiconcave) and bounded from below (in particular monotone or convex functions bounded from below qualify);
\item $N \geq2 \,\,  $ and $g$ of class $C^2$, bounded from below and with bounded second derivatives; 
\item $N \geq2 \,\,  $ and $g$ bounded from below, $g= \varphi \circ \theta$, where $ \theta$ is uniformly continuous on 
$\R^{N-1}$ and $ \varphi  : \R \mapsto \R$ is a continuous bijection;
\item $g(x_1, \ldots, x_{N-1}) = g(x_1,\ldots, x_{n-1})$, where $2 \leq n < N$ and $ g : \R^{n-1} \to \R$  is one of the functions defined in one of the items $(1)-(4). $
\end{enumerate}
Let $\Omega $ be the epigraph defined by $g$ and assume $f \in {Lip}_{loc}([0,+\infty))$ with 
\begin{equation}\label{cond-in-zero-bisbis}
          \liminf_{t \to 0^+} \frac{f(t)}{t} > 0.
\end{equation}

(i) If $u \in C^0(\overline{\Omega}) \cap H^1_{loc}(\overline{\Omega})$ is a  distributional solution to \eqref{NonLin-PoissonEq} which is uniformly continuous on finite strips. \\ 
Then $u$ is strictly increasing in the $x_N$-direction, i.e., $\frac{\partial u}{\partial x_N}>0$ in $\Omega$.

(ii) If $u \in C^2(\Omega) \cap C^0(\overline{\Omega})$ is a classical solution of \eqref{NonLin-PoissonEq} such that $\nabla u \in L^{\infty}(\Omega)$. 
Then $u$ is strictly increasing in the $x_N$-direction, i.e., $\frac{\partial u}{\partial x_N}>0$ in $\Omega$.
\end{thm}

Typical examples of functions $g$, for which Theorem \ref{TH_fonc_G-spec} applies but not Theorem \ref{TH1} (resp. Corollary \ref{Cor1}), are provided by : 
$g_1 (x_1) = e^{x_1} - 4 \arctan(x_1) - 2 $, $g_2(x_1) = e^{e^{x_1}}$ if $ N=2$ and 
$ g(x_1,\ldots, x_{N-1}) = {(x_1)}^2 + \prod_{j=2}^{N-1} \sin(j x_j), \,$ $ g(x_1,\ldots, x_{N-1})=e^{x_1 + \sum_{j=2}^{N-1} \cos^j(x_j)}$ if $ N \geq 3$. 

\medskip

\begin{thm}\label{TH_fonc_G-bis-spec}
Let $N \geq 2$ and let $\Omega $ be as in the statement of Theorem \ref{TH_fonc_G-spec}.
Also suppose that $\Omega$ satisfies a uniform exterior cone condition.

(i) Assume $ f \in {Lip}([0,+\infty))$ with $f(0) \geq 0$ and let $u \in C^0(\overline{\Omega}) \cap H^1_{loc}(\overline{\Omega})$ be a distributional solution to \eqref{NonLin-PoissonEq} with at most exponential growth on finite strips.
Then $u$ is strictly increasing in the $x_N$-direction, i.e., $\frac{\partial u}{\partial x_N}>0$ in $\Omega$.

(ii) Assume $ f \in {Lip}_{loc}([0,+\infty))$ with $f(0)\geq0$ and let $u \in C^0(\overline{\Omega}) \cap H^1_{loc}(\overline{\Omega})$ be a distributional solution to \eqref{NonLin-PoissonEq} which is bounded on finite strips.
Then $u$ is strictly increasing in the $x_N$-direction, i.e., $\frac{\partial u}{\partial x_N}>0$ in $\Omega$.
\end{thm}

Notice that Theorem \ref{TH_fonc_G-bis-spec} applies (for instance) to epigraphs defined by 
$ g(x_1,\ldots, x_{N}) = e^{e^{x_1}}$ or $ g(x_1,\ldots, x_{N}) =  e^{x_1} - 4 \arctan(x_1) - 2 $ if $ N \geq 1$ and to $g(x_1,\ldots, x_{N}) = x_1^4 + e^{x_2}$ if $ N \geq 2$. 

\medskip

The next result applies to solutions of \eqref{NonLin-PoissonEq} where $\Omega$ is merely a continuous epigraph and 
$f$ is a non-increasing function, possibly discontinuous, and with no restriction on the sign of $f(0)$ (see Theorem \ref{TH3-disc} in Section 5). 

\begin{thm}\label{TH3-disc-spec}
Let $\Omega$ be a continuous epigraph bounded from below and let $f: [0,\infty)  \mapsto \R$ be any non-increasing function.
Let $u \in C^0(\overline{\Omega}) \cap H^1_{loc}(\overline{\Omega})$ be a distributional solution to \eqref{NonLin-PoissonEq} 
with subexponential growth on finite strips 
\footnote{i.e., for any $R>0$, 
$$ \limsup_{\overset {\vert x \vert \to \infty,} {x \in \Omega \cap \left \lbrace x_N <R \right\rbrace}} \frac{\ln u(x)}{\vert x \vert} \leq 0. $$}.  \\   
Then $u$ is non-decreasing, i.e., $\frac{\partial u}{\partial x_N} \geq 0$ in $\Omega$.\footnote{\, Note that $ u \in C^1(\Omega),$ since $f(u) \in L^{\infty}_{loc}(\Omega)$.}  \\ 
Moreover, if $f \in {Lip}_{loc},$ then $u$ is strictly increasing, i.e., $\frac{\partial u}{\partial x_N} > 0$ in $\Omega$.
\end{thm}

We conclude the introduction by observing that

1) the new comparisons principles we have demonstrated  have also allowed us to prove some new results of uniqueness and symmetry for solutions (possibly unbounded and sign-changing) to the homogeneous Dirichlet BVP for the semilinear Poisson equation in fairly general unbounded domains (see Corollary \ref{Unic-Dirichlet} and Corollary \ref{Sym-Dirichlet} in Section \ref{SectA});

2) as an application of our new monotonicity theorems, we also prove some new classification and/or non-existence results for solutions to the nonlinear problem \eqref{NonLin-PoissonEq} (see Section \ref{Sect-Applic}). 

\smallskip

For the reader's convenience, the notations used in the paper are collected in Section \ref{Notations}. 

\smallskip

The paper is organized as follows :

\begin{enumerate}

\small {

\item[1.]  \textit{Introduction and main results} 

\smallskip

\item[2.]  \textit{Some new comparison principles on unbounded domains and their applications to the uniqueness and symmetry of solutions to the semilinear Poisson equation} 

\smallskip

\item[3.] \textit{Some uniform estimates in unbounded domains}

\smallskip

\item[4.] \textit{Proofs}

\smallskip

\item[5.] \textit{Extensions to merely continuous epigraphs bounded from below and further observations}

\smallskip

\item[6.] \textit{Some applications to classification and non-existence results}

\smallskip

\item[7.] \textit{Some examples}

\smallskip

\item[8.] \textit{Notations}

}

\end{enumerate}

\section{Some new comparison principles on unbounded domains and their applications to the uniqueness and symmetry of solutions to the semilinear Poisson equation} \label{SectA}
 
In this section we prove some new comparison principles for solutions of semilinear problems on unbounded open sets, whose section has some ''good properties". Those results recover and considerably improve the comparison principle on strips (or on open subsets included in strips) proved by the second author in \cite{Far-Sandro}. Their role in this paper is twofold:

1) they are crucial for proving our monotonicity results on epigraphs,

2) they  allow us to prove some new results of uniqueness and symmetry for solutions (possibly unbounded and sign- changing) to the homogeneous Dirichlet BVP for the semilinear Poisson equation in fairly general unbounded domains.  

To this end we need the following

\begin{Defi} \label{Def-goodsection} 
Assume $N \geq 2.$ Let $\Omega$ be an open subset of  $\R^N$ and let $\nu$ be a unit vector of $\R^N$. \\ We denote by $\R \nu$ the vector space spanned by the unit vector $\nu$ and by $\{\nu\}^{\bot}$ the orthogonal complement of $\nu$ in $\R^N$.

(i) We shall say that $\Omega $ is \texttt{locally bounded in the direction $\nu$} if
\begin{equation}\label{cylindre}
\begin{array}{ll@{\;}l@{\;}}
\forall R >0 \qquad & C_{\nu}(R)=(B'(0',R) \times \R \nu) \cap \Omega & \quad \text{is a bounded subset of } \R^N.
\end{array}
\end{equation}
Here $B'(0',R)$ denotes the $N-1$-dimensional open ball of radius $R$ centered at the origin $0'$ of $\{\nu\}^{\bot}$.

(ii) For every $x' \in \{\nu\}^{\bot}$ let us define the set $S^{\nu}_{x'}:= (\{x'\} \times\R \nu)\cap \Omega.$

We shall say that $\Omega $ has \texttt{bounded  section in the direction $\nu$} if 
\begin{equation}\label{sup}
\mathtt{S}_{\nu}(\Omega):=\displaystyle\sup_{x' \in \{\nu\}^{\bot}}\mathcal{L}^{1} (S^{\nu}_{x'})<+\infty , 
\end{equation}
where $\mathcal{L}^{1}$ denotes the 1-dimensional Lebesgue measure. 

The positive number $\mathtt{S}_{\nu}(\Omega)$  will be called  the  \texttt{section of $\Omega$ in the direction $\nu$}.

(iii) We shall say that $\Omega $ has \texttt{good section in the direction $\nu$} if it is locally bounded
in the direction $\nu$ and if it also has bounded  section in the direction $\nu$ $($that is, if it satisfies both \eqref{cylindre} and  \eqref{sup}$)$. 
\end{Defi} 
 
\medskip	
	
The following remark shows how large the families of sets defined above are.	
	
\medskip		
	
\begin{rem}\label{ex-good-sets}

(i) Open sets (possibly unbounded and not necessarily connected) that are bounded in a fixed direction $\nu$ $($i.e., open sets such that, up to a rotation, are included in a finite strip $\{x=(x',x_N) \in \R^N  \,  : \, \alpha <x_N < \beta \}$, with $\alpha, \beta \in \R)$ have good section in the direction $\nu$. Indeed, they clearly satisfy \eqref{cylindre} as well as \eqref{sup} with 
$\mathtt{S}_{\nu}(\Omega) \leq \beta - \alpha$. 

\smallskip

(ii) The family of open sets with good section in a fixed direction $\nu$ is much larger than the one of sets that are bounded in the direction $\nu$. Indeed, the unbounded open connected set $\Omega_{1}=\{(x,y) \in \R^{2}   \,  : \,  |x|-h(x)<y<|x|+h(x) \} \cup \{(x,y) \in \R^{2}   \,  : \, - |x| -h(x) <y<- |x| + h(x)\} \cup \{(x,y) \in \R^{2}, |y|<1 \}$, where 
$h(x)=sinh^{-1}(e^{-|x|}),$$($see Figure \ref{fig:ex1}$)$ has good section in the direction $e_{2} = (0,1)$, with $\mathtt{S}_{e_2}(\Omega_1)\leq 4$, but it is unbounded in any direction. Actually, it is not contained in any affine half-plane. Also note that the Lebesgue measure of $\Omega_{1}$ is infinite.

\begin{figure}[!h]
	\centering
	\includegraphics[width=0.6\textwidth]{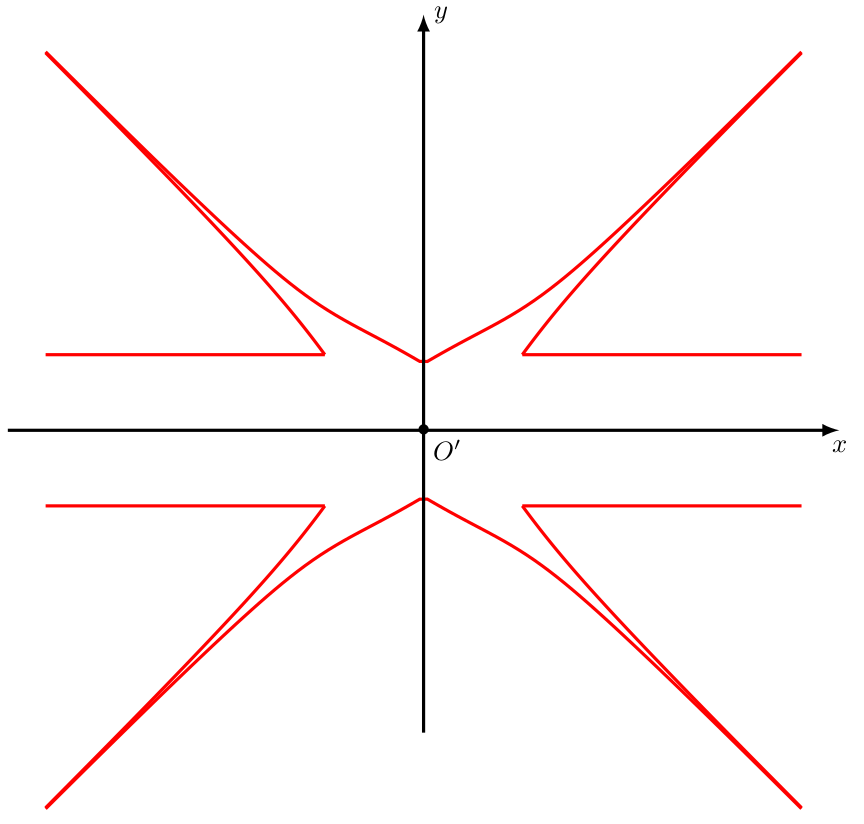}
	\caption{$\Omega_{1}$}
	\label{fig:ex1}
\end{figure}

\smallskip

(iii) We also note that  open sets with bounded section in a fixed direction $\nu$  are not necessarily locally bounded in the direction $\nu$ and vice versa.  Consider the open sets $\Omega_2 = \bigcup_{n \geq 1} \left\lbrace (x,y) \in \R^{2}   \,  : \,
n  < y  <  n + \frac{1}{2^n} \right\rbrace $ and $\Omega_3 =  \left\lbrace (x,y) \in \R^{2}   \,  : \,  0 < y  <  x^2 \right\rbrace$, then $\Omega_2$ has bounded section in the direction $e_2$, but  it is not locally bounded in that direction, while $\Omega_3$ is locally bounded in the direction $e_2$, but it has not bounded section in the direction $e_2$. 
\end{rem}

\medskip 
 
Our first comparison principle is the content of the next result. 
 
\begin{thm}\label{thComp1}
Assume $N \geq 2.$ Let $\Omega$ be an open subset of $\R^N$ and let $\nu$ be a unit vector of $\R^N$. 

\smallskip

(i) Assume that $\Omega$ has good section in the direction $\nu$. 
Let $f \in {Lip}_{loc}(\R)$, $M>0$ and $u,v \in H^{1}_{\text{loc}}(\overline{\Omega}) \cap C^{0}(\overline{\Omega})$  satisfying  
		\begin{equation}
			\left\{
			\begin{array}{cll}
				- \Delta u-f(u)\leq -\Delta v -f(v) & \text{in}& \mathcal{D}'(\Omega),\\
				|u|,|v|\leq M & \text{in} & \Omega,\\
				u \leq v & \text{on} & \partial \Omega.
			\end{array}
			\right.\label{pb1}
		\end{equation} 
		Then, there exists $\varepsilon=\varepsilon(f,M)>0$ such that 
		\begin{equation}
			\mathtt{S}_{\nu}(\Omega)< \varepsilon \implies u \leq v \,\, \text{in} \,\, \Omega.\label{inégalitée1}
		\end{equation}
		
\smallskip		
		
(ii) Assume that $\Omega$ has good section in the direction $\nu$. 
Let $f \in C^{0}(\R)$ be a non-increasing function, $M>0$ and $u,v \in H^{1}_{\text{loc}}(\overline{\Omega}) \cap C^{0}(\overline{\Omega})$ satisfying  
		\begin{equation}
			\left\{
			\begin{array}{cll}
				- \Delta u-f(u)\leq -\Delta v -f(v) & \text{in}&  \mathcal{D}'(\Omega),\\
				|u|,|v|\leq M & \text{in} & \Omega,\\
				u \leq v & \text{on} & \partial \Omega.
			\end{array}
			\right.\label{pb2}
		\end{equation} 
Then, $ u \leq v $ in $ \Omega$. 	
		
\smallskip		
		
(iii) Assume that $\Omega$ has bounded section in the direction $\nu$. 
Let $f \in {Lip}_{loc}(\R)$, $M>0$ and $u,v \in Lip_{\text{loc}}(\overline{\Omega})$ satisfying 
		\begin{equation}
			\left\{
			\begin{array}{cll}
				- \Delta u-f(u)\leq  -\Delta v -f(v) & \text{in}&  \mathcal{D}'(\Omega),\\
				|u|,|v| \leq  M & \text{in} & \Omega,\\
				| \nabla u|,| \nabla v | \leq M & \text{a.e.  in} & \Omega,\\
				u \leq v & \text{on} & \partial \Omega.
			\end{array}
			\right.\label{pb3}
		\end{equation} 
		Then, there exists $\varepsilon=\varepsilon(f,M)>0$ such that 
		\begin{equation}
			\mathtt{S}_{\nu}(\Omega)< \varepsilon \implies u \leq  v \,\, \text{in} \,\, \Omega.\label{inégalitée3}
		\end{equation}		
\end{thm} 
	 
More generally, we have the following results\footnote{\, For sake of clarity, Theorem \ref{thComp2} and Theorem \ref{thComp3} are stated with respect to the direction $e_N$, the last vector of the canonical base of $\R^N$. Since the considered problems are invariant by rotation, it is clear that they also hold true if the unit vector $e_N$ is replaced by any unit vector $\nu$ (and with the corresponding natural modification  of condition \eqref{eq3}  in the case of Theorem \ref{thComp2}).}
	
\begin{thm}\label{thComp2}
Assume $ \gamma \geq 0$, $\delta \geq 0$, $N \geq 2$ and let $\Omega $ be an open subset of $ \R^N$ with good section in the direction $e_N$, the last vector of the canonical base of $\R^N$, such that   
		\begin{equation}\label{eq3}
				\sup_{x' \in\R^{N-1}} \Big(\disp\int_{S^{e_N}_{x'}} |x_N|^{2 \delta} e^{2 \gamma|x_N|}d x_N\Big) < + \infty . 
		\end{equation}	
		Let $f = f_1 + f_2,$ with $f_1 \in {Lip}(\R)$ and $f_2 : \R \mapsto \R$ be a non-increasing function.
		
		Let $a>0$ and $u,v \in H^{1}_{\text{loc}}(\overline{\Omega})\cap C^{0}(\overline{\Omega})$ such that 
		\begin{equation*}
			\left\{
			\begin{array}{cll}
				- \Delta u-f(u)\leq  -\Delta v -f(v) & \text{in}&  \mathcal{D}'(\Omega),\\
				|u|,|v|\leq a |x|^{\delta} e^{\gamma |x|} & \text{in} & \Omega,\\
				u \leq v & \text{on} & \partial \Omega.
			\end{array}
			\right.
		\end{equation*} 
Then, there exists $\varepsilon=\varepsilon(L_{f_1},\gamma)>0$ such that
	\begin{equation*}
		\mathtt{S}_{e_N}(\Omega)< \varepsilon \implies u \leq v \,\, \text{in} \,\, \Omega.
	\end{equation*}
Here, $L_{f_1}$ denotes the Lipschitz constant of $f_1$.  	
	\end{thm} 
	
\medskip	

Some remarks are in order 

\medskip

\begin{rem}\label{rem-striscia-exp}

(i) Any open set $\Omega$ included in a strip $\{x=(x',x_N) \in \R^N \,  : \, \alpha <x_N < \beta \}$, with $\alpha, \beta \in \R$, clearly satisfies the assumption \eqref{eq3} for every $\gamma, \delta \geq 0$. 

\smallskip

(ii) Notice that the preceding theorem applies to the set $\Omega_1$ described in Remark \ref{ex-good-sets}. Indeed, 
$\Omega_1$ satisfies the assumption \eqref{eq3} for any $\delta \geq 0$ and any $\gamma \in [0,1/2]$.  

\smallskip

(iii) The previous comparison result applies, in particular, to the case where $f$ is globally Lipschitz-continuous. We shall use it in this form to prove Theorem \ref{TH2}.
\end{rem}

\medskip

When $f$ is non-increasing on $\R$, the comparison principle holds even without the smallness assumption on the section of $\Omega$.  More precisely we have the following

\medskip

\begin{thm}\label{thComp3}
Assume $N \geq 2$ and let $\Omega $ be an open subset of $ \R^N$ bounded in the direction $e_N$, the last vector of the canonical base of $\R^N$. Let $f : \R \mapsto \R$ be a non-increasing function and 
$u,v \in H^{1}_{\text{loc}}(\overline{\Omega})\cap C^{0}(\overline{\Omega})$ such that 
		\begin{equation*}
			\left\{
			\begin{array}{cll}
				- \Delta u-f(u)\leq  -\Delta v -f(v) & \text{in}&  \mathcal{D}'(\Omega),\\
				|u|,|v|\leq a|x|^{\delta} e^{\gamma |x|} & \text{in} & \Omega,\\
				u \leq v & \text{on} & \partial \Omega,
			\end{array}
			\right.
		\end{equation*} 
for some $a >0$, $\delta \geq 0$ and $ \gamma \in \left[ 0,  \frac{\pi}{4\mathtt{S}_{e_N}(\Omega) \sqrt{e-1}} \right) $. \\	
Then, $ u \leq v $ in $ \Omega$. 	
\end{thm} 

\smallskip

The smallness assumption on the section of $\Omega$ is necessary for the validity of both Theorem \ref{thComp1} (item $(i)$ and item $(iii)$) and Theorem \ref{thComp2}. Indeed, the functions $u= \sin(y)$ and $ v \equiv 0$ satisfy $- \Delta u -u = 0 = - \Delta v -v$ on  the two-dimensional strip $ \Omega = \{(x,y) \in \R^2 \,  : \, 0 <y < \pi \}$ and $ u \leq v $ on $ \partial \Omega$, but the conclusion of the comparison principles fails.

That the growth assumption on $u$ in Theorem \ref{thComp3} is necessary to ensure the comparison principle is seen from the following classical example involving harmonic functions on finite strips of $\R^2$, namely, when $f\equiv 0$. It is  well-known that, for any integer $ k \geq 1$, the  function $u_k(x,y) = \sum_{m=1}^k (e^{mx} + e^{-mx}) \sin(mx)$ is harmonic on the strip $ \Omega = \{(x,y) \in \R^2 \,  : \, 0 <y < \pi \}$ and vanishes on $\partial \Omega$. Therefore, a restriction on the exponential growth must be prescribed in order to get the desired results. 

\smallskip

As an immediate consequence of our comparison principles, we have the following result about the uniqueness of the homogeneous Dirichlet problem.

\begin{cor}\label{Unic-Dirichlet}
Let $\Omega$ and $f$ be as in the statement of Theorem \ref{thComp1} (resp. Theorem \ref{thComp2} or Theorem \ref{thComp3}). Then, the Dirichlet problem
\begin{equation}\label{Dirichlet-pbl}
\begin{cases}
-\Delta u=f(u) &\text{ in }  \mathcal{D}'(\Omega),\\
\quad u=0\,\, &\text{ on } \partial\Omega,
\end{cases}
\end{equation}
has, at most, one solution $u$ satisfying the regularity and the growth conditions stated in Theorem \ref{thComp1} (resp. Theorem \ref{thComp2} or Theorem \ref{thComp3}). \\
In particular, if $f(0) =0$, the function $u \equiv 0$ is the only solution to the Dirichlet problem \eqref{Dirichlet-pbl}. 
\end{cor}

The examples considered after Theorem \ref{thComp3},  prove that the preceding result is no longer true if either the assumptions on the size of $\Omega$ or those on the growth of $u$ are not met. 

\medskip

An immediate consequence of  Corollary \ref{Unic-Dirichlet} is the following general symmetry result for solutions to the homogeneous Dirichlet problem \eqref{Dirichlet-pbl}. 

\begin{cor}\label{Sym-Dirichlet}
Let $\Omega$, $f$ be as in the statement of Corollary \ref{Unic-Dirichlet} and assume that the Dirichlet problem
\begin{equation}\label{Sym-Dir-pbl}
\begin{cases}
-\Delta u=f(u) &\text{ in }  \mathcal{D}'(\Omega),\\
\quad u=0\,\, &\text{ on } \partial\Omega,
\end{cases}
\end{equation}
has a  solution $u$ satisfying the regularity and the growth conditions stated in Corollary \ref{Unic-Dirichlet}.

Let $ \rho$ be an isometry of $\R^N$, $N\geq2$. If $\Omega$ is invariant with respect to $\rho$, i.e., $\Omega$ satisfies 
$ \rho(\Omega) = \Omega$, then the solution $u$ inherits the same symmetry, namely, $ u(x) = u(\rho(x))$ for any $x \in \Omega$.   
\end{cor}

Let us illustrate the above result on the \textit{torsion problem} for an infinite solid bar. Let us start with the case of an infinite solid straight bar with spherical cross section. That is, the study of the solutions to the problem 
\begin{equation}\label{Sym-Dir-bar}
\begin{cases}
-\Delta u= K &\text{ in }  \mathcal{D}'(\Omega),\\
\quad u=0\,\, &\text{ on } \partial\Omega,
\end{cases}
\end{equation}
where $ \Omega = \Omega_{K,R} := \R^{N-K} \times B^K_R$, where $1 \leq K< N$ is an integer,  and 
$B^K_R$ denotes the open ball of $\R^K$ centered at the origin and of radius $R>0$. 

Since $\Omega_{K,R}$ is invariant with respect to any translation in the variables $x_1,\ldots, x_{N-K}$ and also to any rotation in the variables $x_{N-K+1},\ldots, x_{N}$, Corollary \ref{Sym-Dirichlet}\footnote{ \, applied with the non-increasing function $f \equiv K$.} tell us that the unique solution $u$ of   \eqref{Sym-Dir-bar}, with subexponential growth at infinity, must be independent of $x_1,\ldots, x_{N-K}$ and radially symmetric in the variables $x_{N-K+1},\ldots, x_{N}$. \\
It is therefore easy to check that the function 
$$u(x) =\frac{R^2 - (x_{N-K+1}^2 + \ldots + x_{N}^2)}{2},  \quad \qquad x \in \Omega_{K,R} $$
is the only solution to the problem \eqref{Sym-Dir-bar} with subexponential growth at infinity. 

The previous analysis extends to general infinite solid bars (not necessarily  straight). For instance, it applies to the following examples : 

1) $ \Omega = \R^{N-K} \times \omega$, where $1 \leq K< N$ is an integer  and $\omega$ is an open bounded subset of $\R^K$. In this case, the unique solution $u$ of \eqref{Sym-Dir-bar} with subexponential growth at infinity, depends only on the variables $x_{N-K+1},\ldots, x_{N}$. 

2) $ \Omega = \{(x_1,x_2) \in \R^2 \,  : \, -\varphi (x_1) <x_2 < \varphi(x_1) \}$, where $\varphi : \R \mapsto (0, +\infty)$ is a bounded continuous function.  In this case the solution $u$ must be symmetric with respect to the line $ x_2 =0$. 

3) $\Omega $ is a domain of revolution of $\R^N$,$ N \geq 3$, i.e.,  $ \Omega = \left\lbrace x \in \R^N \, : \,  \sqrt{x_{2}^2 + \ldots + x_{N}^2} < \varphi (x_1) \right\rbrace $, where $\varphi : \R \mapsto (0, +\infty)$ is a bounded continuous function. 
Here, $u$ must have the form $u = u(x_1, \sqrt{x_{2}^2 + \ldots + x_{N}^2})$.

Also observe that, if $ \varphi $ is $T$-periodic, with $ T>0$, then the solution $u$ is $T$-periodic in the $x_1$ variable.  

\bigskip

It is clear that the above discussions and results also hold true for general non linear function $f$ satisfying the assumptions of Corollary \ref{Sym-Dirichlet}. 

\bigskip

Now we are ready to prove our comparison principles.

\bigskip

\textit{Proof of Theorem \ref{thComp1}.}  Since our problem is invariant under rotation, we may and do suppose that 
$\nu=e_{N},$ the last vector of the canonical base of $\R^N$.
By assumption, for every $\phi \in \cic(\Omega)$, $ \phi \geq 0$, we have 
		\begin{equation}\label{eq4}
			\disp\int_{\Omega} \nabla (u-v) \nabla \phi \leqslant \disp\int_{\Omega} (f(u)-f(v)) \phi .
		\end{equation} 
Let us first consider the cases $(i)$ and $(ii)$, where we suppose that $\Omega$ has good section in the direction $e_N$.

Observe that, for any $\Psi \in C^{0,1}_{c}(\R^{N-1})$ and for any open ball $B' \subset \R^{N-1}$ with $\text{supp}(\Psi)\subset B'$, we have $g:= (u-v)^{+} \Psi^{2} \in C^{0}(\overline{\Omega})$. Let us prove that $g$ also belongs to $H^{1}_{0}(\Omega\cap (B'\times \R))$.  \\	
The assumption \eqref{cylindre} ensures that $\Omega \cap (B' \times \R)$ is a bounded open subset of $\R^N$ thus, from the assumption  $u,v \in H^{1}_{\text{loc}}(\overline{\Omega})\cap C^{0}(\overline{\Omega})$, it follows that $(u-v)^{+}$ is a bounded, continuous function on the set $ \overline{\Omega \cap (B' \times \R)}$ such that $(u-v)^{+} \in H^{1}(\Omega \cap (B' \times \R))$. Also, since $\Psi^2 \in C^{0,1}_{c}(\R^{N-1})$, we have $g \in H^{1}(\Omega \cap (B' \times \R))$ and 
		\begin{equation}\label{deriv-g}
			\begin{array}{lll}
			\nabla g =\nabla ((u-v)^{+})\Psi^{2}+2\Psi (u-v)^{+} \nabla \Psi & \text{in} & \mathcal{D}'(\Omega \cap (B' \times \R)). 
			\end{array}
		\end{equation}
Also note that 
\begin{itemize}
			\item if $x \in \dr \Omega \cap (B'\times \R),$ then $g(x)=0, $ since $u(x) \leqslant v(x)$ on $ \dr \Omega $,		
			\item if $ x=(x',x_N) \in \overline{\Omega} \cap \dr (B'\times \R),$ then $\Psi(x')=0$, and so $g(x)=0$,
		\end{itemize}
hence $ g=0 $ on $\dr (\Omega\cap (B' \times \R)) = (\dr \Omega \cap (B'\times \R)) \cup (\overline{\Omega} \cap \dr (B'\times \R)).$ Then $g \in H^{1}_{0}(\Omega\cap (B' \times \R)),$ $ g \geq 0,$ and so we can find a sequence of functions 
$\phi_{n}\in \cic(\Omega \cap (B' \times \R))$, $\phi_{n} \geq 0,$ such that 
		\begin{equation*}
			\lim\limits_{n \to +\infty} \| \phi_{n}-g \|_{H^{1}(\Omega\cap (B' \times \R))}=0.
		\end{equation*} 
By \eqref{eq4}, for any $n \geq 1$ we have 
		\begin{equation*}
\disp\int_{\Omega \cap (B' \times \R)} \nabla (u-v) \nabla \phi_{n} = \int_{\Omega } \nabla (u-v) \nabla \phi_{n} \leq \disp\int_{\Omega} (f(u)-f(v)) \phi_{n} =\disp\int_{\Omega\cap (B' \times \R)} (f(u)-f(v)) \phi_{n}
		\end{equation*}     
and passing to the limit, we obtain 
		\begin{equation}\label{eq5}
			\disp\int_{\Omega} \nabla (u-v) \nabla g \leq \disp\int_{\Omega} (f(u)-f(v)) g, 
		\end{equation}
since the support of $g$ is contained in  $\Omega\cap (B' \times \R)$. 
Plugging  \eqref{deriv-g} into \eqref{eq5}  yields 
	\begin{equation}\label{eq5-nuova}
	\disp\int_{\Omega} \nabla (u-v) (\nabla (u-v)^{+} \Psi^{2} +(u-v)^{+} 2\Psi \nabla \Psi) \leq \disp\int_{\Omega} 
	(f(u)-f(v))(u-v)^{+} \Psi ^{2} 
	\end{equation}		
and so 	
	\begin{equation}\label{eq5bis}
\disp\int_{\Omega} |\nabla (u-v)^{+}|^{2}  \Psi^{2} +\disp\int_{\Omega} (u-v)^{+} 2\Psi \nabla (u-v)^{+} \nabla \Psi  \leq L_{f,M}\disp\int_{\Omega} ((u-v)^{+})^{2} \Psi ^{2}, 
	\end{equation}
where $L_{f,M}$ is any positive number greater or equal to the Lipschitz constant of $f$ on the compact set $[-M,M]$ if we are in the case $(i),$ while $L_{f,M}=0$ if the case $(ii)$ is in force. Hence
	\begin{align*}			
			& \disp\int_{\Omega} |\nabla (u-v)^{+}|^{2}  \Psi^{2} \leq -2\disp\int_{\Omega} (u-v)^{+} \Psi \nabla (u-v)^{+} \nabla \Psi + L_{f,M}\disp\int_{\Omega} ((u-v)^{+})^{2}  \Psi^{2} \\  
			& \disp\int_{\Omega} |\nabla (u-v)^{+}|^{2}  \Psi^{2} \leq -2\disp\int_{\Omega} (\sqrt{2}(u-v)^{+} \nabla \Psi ) 
			\left( \dfrac{\Psi\nabla (u-v)^{+}}{\sqrt{2}}\right)   + L_{f,M}\disp\int_{\Omega} ((u-v)^{+})^{2} \Psi^{2} \\  
			& \disp\int_{\Omega} |\nabla (u-v)^{+}|^{2}  \Psi^{2} \leq 2\disp\int_{\Omega} (\sqrt{2}(u-v)^{+} |\nabla \Psi| ) 
			\left( \dfrac{|\Psi||\nabla (u-v)^{+}|}{\sqrt{2}}\right)   + L_{f,M}\disp\int_{\Omega} ((u-v)^{+})^{2} \Psi^{2}
		\intertext{and by Young's inequality}  
			& \disp\int_{\Omega} |\nabla (u-v)^{+}|^{2}  \Psi^{2} \leq \disp\int_{\Omega} 2((u-v)^{+})^{2} |\nabla \Psi|^{2} +
			\disp\int_{\Omega} \dfrac{\Psi^{2}|\nabla (u-v)^{+}|^{2}}{2} + L_{f,M} \disp\int_{\Omega} ((u-v)^{+})^{2} \Psi^{2}. 
\end{align*}
Therefore we have
		\begin{equation}\label{eq6}
		\disp\int_{\Omega} |\nabla (u-v)^{+}|^{2}  \Psi^{2} \leq 4\disp\int_{\Omega} ((u-v)^{+})^{2} |\nabla \Psi|^{2} + 2L_{f,M}\disp\int_{\Omega} ((u-v)^{+})^{2} \Psi^{2}. 
		\end{equation}
		Since
		\begin{equation*}
			\begin{array}{lll}
				|\nabla (u-v)^{+}|^{2} \geq |\partial_{N} (u-v)^{+}|^{2}& \text{a.e. on } & \Omega,
			\end{array}	
		\end{equation*}
		we get
		\begin{align*}
			\disp\int_{\Omega} |\nabla (u-v)^{+}|^{2}  \Psi^{2}  &\geq \disp\int_{\Omega} |\partial_{N} (u-v)^{+}|^{2}\Psi^{2}=
			\disp\int_{B'} \disp\int_{S^{e_{N}}_{x'}}|\partial_{N} (u-v)^{+}(x',x_N)|^{2}  \Psi^{2}(x') dx_N dx'=\\
			&= \disp \int_{B'} \Psi^{2}(x') \Big(\disp\int_{S^{e_{N}}_{x'}}|\partial_{N} (u-v)^{+}(x',x_N)|^{2} dx_N \Big)dx'. 
		\end{align*}
Now we observe that, for every $x' \in B'$, the one-dimensional set $ S^{e_{N}}_{x'}$ can be written as the disjoint union of an at most countable family of open bounded intervals (its connected components), i.e., $ S^{e_{N}}_{x'}=\sqcup_{j \in J(x')}I_{x',j}$, with $J(x') \subset \mathbb{N}$ and $I_{x',j}$ an open bounded interval. Therefore, by Poincaré's inequality, we have 
		\begin{align*}
			\disp\int_{S^{e_{N}}_{x'}}|\partial_{N} (u-v)^{+}(x',x_N)|^{2} d x_N
			& =\disp\sum_{j \in J(x')} \disp\int_{I_{x',j}}|\partial_{N} (u-v)^{+}(x',x_N)|^{2} d x_N \geq \\
			&\geq \disp\sum_{j \in J(x')} \dfrac{\pi^{2}}{(\mathcal{L}^{1}(I_{x',j}))^{2}} \disp\int_{I_{x',j}}((u-v)^{+}(x',x_N))^{2}
			d x_N \geq \\
&\geq \dfrac{\pi^{2}}{(\mathtt{S}_{e_{N}}(\Omega))^{2}}\disp\sum_{j \in J(x')} \disp\int_{I_{x',j}}((u-v)^{+}(x',x_N))^{2}d x_N=\\
			&= \dfrac{\pi^{2}}{(\mathtt{S}_{e_{N}}(\Omega))^{2}}\disp\int_{S^{e_{N}}_{x'}}((u-v)^{+}(x',x_N))^{2}d x_N
		\end{align*} 
and so
		\begin{equation}
		\disp\int_{\Omega} |\nabla (u-v)^{+}|^{2}  \Psi^{2} \geq \dfrac{\pi^{2}}{(\mathtt{S}_{e_{N}}(\Omega))^{2}} \disp\int_{\Omega} 
		((u-v)^{+})^{2} \Psi^{2}. \label{eq7}
		\end{equation} 
Now, thanks to \eqref{eq6} and \eqref{eq7}, we deduce that
		\begin{equation}\label{eq8}
			\Big(\frac{\pi^{2}}{(\mathtt{S}_{e_{N}}(\Omega))^{2}}-2L_{f,M} \Big) \disp\int_{\Omega} ((u-v)^{+})^{2} \Psi^{2} \leq  4\disp\int_{\Omega} ((u-v)^{+})^{2} |\nabla \Psi|^{2} . 
		\end{equation}
If the case $(i)$ is in force, we set $\varepsilon(f,M) = \frac{\pi}{\sqrt{2L_{f,M}}}$.  Then, if $\mathtt{S}_{e_{N}}(\Omega)<\varepsilon(f,M)$ we have 
       \begin{equation}\label{eq9}
\disp\int_{\Omega} ((u-v)^{+})^{2} \Psi^{2} \leq C_{1}(f,M,\Omega) \disp\int_{\Omega} ((u-v)^{+})^{2} |\nabla \Psi|^{2}, 
       \end{equation}
with $C_{1}(f,M,\Omega)=\dfrac{4}{\frac{\pi^{2}}{(\mathtt{S}_{e_{N}}(\Omega))^{2}}-2L_{f,M}}>0$.\\ 
When the case $(ii)$ is in force, equation \eqref{eq8} yields 
		\begin{equation}\label{eq10}
		\disp\int_{\Omega} ((u-v)^{+})^{2} \Psi^{2} \leq C_{2}(\Omega) \disp\int_{\Omega} ((u-v)^{+})^{2} |\nabla \Psi|^{2} 
		\end{equation}
where $C_2 (\Omega) = \frac{4\mathtt{S}^{2}_{e_N}(\Omega)}{\pi^2}$ (and without any restriction on the size of 
$\mathtt{S}_{e_N}(\Omega)$).\\
		
For $0< a<b$, set $h=b-a$ and consider the Lipschitz-continuous function 
		\begin{align*}
			\eta_{h} : & \, \R^{+} \to \R\\
			& \, \, t \mapsto \left\{
			\begin{array}{cll}
				1 & \text{if} & t \in [0,a],\\
				\frac{b-t}{h} & \text{if} & t \in [a,b],\\
				0 & \text{if} & t \in [b,+\infty).
			\end{array}
			\right.
		\end{align*} 
For any $x' \in \R^{N-1},$ set $\Psi_{h}(x'):=\eta_{h}(\vert x' \vert)$, then $\Psi_{h} \in C^{0,1}_{c}(\R^{N-1})$ and 
$| \nabla' \Psi_{h}(x') |\leq  \frac{1}{h}$ for almost every $x' \in  \R^{N-1}$.\\
For $r>0,$ recall that $C_{e_{N}}(r)= \Omega \cap (B'(0',r) \times \R)$ and consider the function defined by 
		\begin{equation}\label{def-w(r)}
w(r):=\disp\int_{C_{e_{N}}(r)} ((u-v)^{+})^{2}.		
		\end{equation}

If the case $(i)$ is in force, from \eqref{eq9} we  have 
		\begin{align*}
		w(a)=\disp\int_{C_{e_{N}}(a)} ((u-v)^{+}(x))^{2} &\leq \disp\int_{\Omega}((u-v)^{+}(x))^{2}\Psi^{2}_{h}(x')\\ &\leq C_{1}(f,M,\Omega) \disp\int_{\Omega} ((u-v)^{+}(x))^{2}|\nabla \Psi_{h}(x')|^{2}\\
			& \leq  \frac{C_{1}(f,M,\Omega)}{h^{2}}\disp\int_{C_{e_{N}}(b)\backslash C_{e_{N}}(a)} ((u-v)^{+}(x))^{2}	
		\end{align*} 
		
and by adding $ \frac{C_{1}(f,M,\Omega)}{h^{2}}w(a)$ to both sides of the latter, we have

		\begin{equation*}
			w(a) \leq \frac{1}{\frac{h^{2}}{C_{1}(f,M,\Omega)}+1} w(a+h), \qquad \forall \, a, h >0. 
		\end{equation*}
To conclude it is enough to prove that $ w \equiv 0$. If not, we can find $A>0$ such that 
		\begin{equation}\label{eq11}
			w(A)>0,
		\end{equation}
and so, we can apply Lemma $3.1$ in \cite{bisfapi} with 
$\alpha=\frac{1}{C_{1}(f,M,\Omega)}>0$,  $ \delta =0, \gamma =2$ and $\beta=1$ to infer that, for any $h>0$, the function $w$ satisfies 
\begin{equation}\label{eq12}
			w(R)\geq w(A)(\alpha h^{2} +1)^{\frac{R-A}{h}-1}, \qquad \forall \, R \geq A+h .
\end{equation}
		
If we take $h=\sqrt{\frac{e-1}{\alpha}}>0,$ then 
		\begin{equation}\label{eq12bis}
		w(R)\geq w(A)(\alpha h^{2} +1)^{\frac{R-A}{h}-1}=w(A)e^{-(\frac{A}{h}+1)}e^{\frac{R}{h}}, \qquad \forall \, R >> 1 .
		\end{equation}
On the other hand, the boundedness of $u$ and $v$ leads to 
\begin{equation}\label{eq13}
w(R) \leq 4M^{2} \mathcal{L}^{N}(C_{e_{N}}(R)) \leq 4M^{2}\mathtt{S}_{e_{N}}(\Omega)\omega_{N-1}R^{N-1}, \qquad \forall \, R >0, 
\end{equation}
where $\mathcal{L}^{N}$ denotes the $N$-dimensional Lebesgue measure and $\omega_{N-1}$ is the Lebesgue measure of unit ball of $\R^{N-1}$.
Inequality \eqref{eq13} is in contradiction with \eqref{eq12bis}, so $ w \equiv 0$ and the desired conclusion follows. 

The case $(ii)$ is proven in exactly the same way. Simply replace the constant $C_{1}(f,M,\Omega)$ by the 
constant $C_{2}(\Omega)$ in the previous argument.  Indeed, since $f$ is non-increasing, inequality \eqref{eq5-nuova}	becomes 
\begin{equation*}
	\disp\int_{\Omega} \nabla (u-v) (\nabla (u-v)^{+} \Psi^{2} +(u-v)^{+} 2\Psi \nabla \Psi) \leq \disp\int_{\Omega} 
	(f(u)-f(v))(u-v)^{+} \Psi ^{2} \leq 0
\end{equation*}	
and so 	\eqref{eq5bis} holds true with $ L_{f,M} =0$. 
		
The proof in the case $(iii)$ is similar to the one of the case $(i)$. However, we must pay attention to the fact that in this case we do not assume that $\Omega$ is locally bounded in the direction $e_N,$ and therefore we have to use a different path to prove that $g \in H^{1}_{0}(\Omega\cap (B'\times \R))$. \\
To this end, we set $\omega = \Omega\cap (B'\times \R)$ and we observe that $\mathcal{L}^{N}(\omega)<+\infty.$ Indeed, 
\begin{equation*}
\mathcal{L}^{N}(\omega)=\disp\int_{\omega}dx=\disp\int_{B'}\Big(\disp\int_{S_{x'}^{e_{N}}}d x_N \Big)dx' \leq  
\mathtt{S}_{e_{N}}(\Omega) \mathcal{L}^{N-1}(B')<+\infty,
\end{equation*}
where $\mathcal{L}^{N-1}$ denotes the ($N-1$)-dimensional Lebesgue-measure of the open ball $B' \subset \R^{N-1}$.

We already know that  $g \in C^{0}(\overline{\omega})$ and $g=0$ on $\dr \omega$. Moreover $g \in L^2(\omega)$ since 
\begin{equation*}
\disp\int_{\omega}g^{2}(x)dx=\disp\int_{\omega}((u-v)^{+}(x))^{2} \Psi^{4}(x')dx\leq 4M^{2}\|\Psi\|_{L^{\infty}(\R^{N-1})}^{4}
\mathcal{L}^{N}(\omega)<+\infty.
\end{equation*}
Also, $g$ is locally Lipschitz-continuous in $\Omega$, since $u,v \in Lip_{\text{loc}}(\overline{\Omega})$ by assumption. 
Hence, formula \eqref {deriv-g} holds true and so $g \in H^{1}(\Omega\cap (B'\times \R))$, since 
\begin{align*}
\disp\int_{\omega}(\dr_{j}((u-v)^{+})^{2} \Psi^{4} & \leq 4M^{2}\|\Psi\|_{L^{\infty}(\R^{N-1})}^{4}\mathcal{L}^{N}(\omega),
\end{align*}
		\begin{equation*}
			\disp\int_{\omega}\Psi^{2}((u-v)^{+})^{2}(\dr_{j}\Psi)^{2} \leq 4 M^{2} \|\Psi\|_{L^{\infty}(\R^{N-1})}^{2} 
			\| \partial_j \Psi\|_{L^{\infty}(\R^{N-1})}^{2}\mathcal{L}^{N}(\omega).
		\end{equation*}
Then, $g \in H^{1}_{0}(\omega)$ and so, we can proceed  as in the proof of the case $(i)$ to obtain the desired conclusion. \qed

\medskip

\textit{Proof of Theorem \ref{thComp2}.} We proceed as in the proof of the case $(i)$ of Theorem \ref{thComp1} until 
\eqref{eq5-nuova}. (This is possible, since in this case \eqref{eq5} holds even under the assumption $f= f_1 + f_2$, where $f_2$ is a non-increasing function, possibly discontinuous. Indeed, in this case $f(u)$ and $f(v)$ are bounded measurable functions on the bounded set $\Omega \cap (B' \times \R)$). Then we obtain \eqref{eq5bis}, where now $L_{f,M}$ is any positive number greater or equal to $L_{f_1}$, the Lipschitz constant of $f_1$.  Here we have used that $(f(u)-f(v))(u-v)^{+} \leq (f_1(u)-f_1(v))(u-v)^{+}$, since $f_2$ is non-increasing. After that, the same proof leads to inequality \eqref{eq8}. 

To complete the proof we set $\varepsilon(L_{f_1},\gamma) = \dfrac{\pi}{\sqrt{16(e-1)\gamma^{2}+2 L_{f,M}}}$ and we observe that \eqref{eq9} is still satisfied if 
$ \mathtt{S}_{e_{N}}(\Omega)<\varepsilon(L_{f_1}, \gamma)$. Therefore, we can consider once again the function $w$ defined by \eqref{def-w(r)} and then follow the proof until \eqref{eq12bis}, i.e., 
\begin{equation}\label{eq12tris}
		w(R) \geq w(A)e^{-(\frac{A}{h}+1)}e^{\frac{R}{h}}, \qquad \forall \, R >> 1, 
		\end{equation}
where $h=\sqrt{\frac{e-1}{\alpha}}$, $A>0$ and $w(A)>0$.
On the other hand, for $R >1$, we also have  
	\begin{align}\label{eq12tris-contrad}
	w(R)& =\disp\int_{C_{e_{N}}(R)}((u-v)^{+})^{2} dx  \leq 4a^{2} \disp\int_{C_{e_{N}}(R)} \vert x \vert^{2 \delta} e^{2 \gamma \vert x \vert} dx  \\
	&\leq 4a^{2} \disp\int_{C_{e_{N}}(R)} (R+\vert x_N \vert)^{2 \delta} e^{2 \gamma (R+\vert x_N \vert)} dx \\
	&\leq 4a^{2} (2R)^{2\delta} e^{2\gamma R} \disp\int_{B'(0',R)} \left(\disp\int_{S^{e_{N}}_{x'}}e^{2\gamma \vert x_N \vert} dx_N\right)dx'  + \\
	&4a^{2} 2^{2\delta} e^{2\gamma R} \disp\int_{B'(0',R)} \left(\disp\int_{S^{e_{N}}_{x'}}
	\vert x_N \vert^{2 \delta} e^{2\gamma \vert x_N \vert} dx_N\right)dx' \\
	&\leq 4a^{2} (2R)^{2\delta} e^{2\gamma R} \disp\int_{B'(0',R)} \left(\disp\int_{S^{e_{N}}_{x'}}
	 (1 + \vert x_N \vert^{2 \delta})e^{2\gamma \vert x_N \vert} dx_N\right)dx'
	\end{align}
Now we obeserve that 	
\begin{align}\label{calcolo-int}
& \disp\int_{S^{e_{N}}_{x'}} (1 + \vert x_N \vert^{2 \delta})e^{2\gamma \vert x_N \vert} dx_N = \\
& \disp\int_{S^{e_{N}}_{x'} \cap \left\lbrace \vert x_N \vert \leq 1\right\rbrace } (1 + \vert x_N \vert^{2 \delta})e^{2\gamma \vert x_N \vert} dx_N  + \disp\int_{S^{e_{N}}_{x'} \cap \left\lbrace \vert x_N \vert >1\right\rbrace } (1 + \vert x_N \vert^{2 \delta})e^{2\gamma \vert x_N \vert} dx_N \\
	&\leq 2 e^{2\gamma} \mathtt{S}_{e_{N}}(\Omega) +  \disp\int_{S^{e_{N}}_{x'}} 2 \vert x_N \vert^{2 \delta} e^{2\gamma \vert x_N \vert} dx_N \leq 2 e^{2\gamma} \mathtt{S}_{e_N}(\Omega) +  2 C_3 := C_4 < + \infty,  
	\end{align}
where $C_3 := \sup_{x' \in\R^{N-1}} \Big(\disp\int_{S^{e_N}_{x'}} \vert x_N \vert^{2 \delta}  e^{2 \gamma |x_N|}d x_N\Big) < + \infty $, by \eqref{eq3}.

From the previous inequalities we infer that, for $R>1$,  
\begin{align}\label{calcolo-int-finale}
w(R) \leq 4a^{2} (2R)^{2\delta} e^{2\gamma R} C_4\mathcal{L}^{N-1}(B'(0',R)) = 4a^{2} 2^{2\delta} C_4 \mathcal{L}^{N-1}(B'(0',1)) R^{2\delta + N-1} e^{2\gamma R} 
\end{align}
The latter inequality contradicts \eqref{eq12tris}, since 
$$
2 \gamma < {\frac{1}{h}} \Longleftrightarrow \mathtt{S}_{e_{N}}(\Omega)< \dfrac{\pi}{\sqrt{16(e-1)\gamma^{2}+
2 L_{f,M}}},
$$
therefore $w \equiv 0$ and so $ u \leq v $ on $\Omega$. \qed

\bigskip

\textit{Proof of Theorem \ref{thComp3}.} We proceed as in the proof of the case $(ii)$ of Theorem \ref{thComp2} until 
\eqref{eq10}. (This is possible, since \eqref{eq5} holds even when $f$ is a non-increasing function, possibly discontinuous. Indeed, in this case $f(u)$ and $f(v)$ are bounded measurable functions on the bounded set $\Omega \cap (B' \times \R)$). Then we consider the function $w$ defined by \eqref{def-w(r)} and then follow the proof until \eqref{eq12bis}, i.e., 
\begin{equation}\label{eq12quater}
		w(R) \geq w(A)e^{-(\frac{A}{h}+1)}e^{\frac{R}{h}}, \qquad \forall \, R >> 1, 
		\end{equation}
where $h=\sqrt{\frac{e-1}{\alpha}}$, $A>0$ and $w(A)>0$.
On the other hand, we also have  
\begin{align}\label{eq12quater-contrad}
	w(R)& =\disp\int_{C_{e_{N}}(R)}((u-v)^{+})^{2} dx  \leq 4a^{2} \disp\int_{C_{e_{N}}(R)} |x|^{2\delta} e^{2 \gamma \vert x \vert} dx \\
	& \leq 4a^{2} \disp\int_{C_{e_{N}}(R)} (R+|x_{N}|)^{2\delta} e^{2 \gamma (R+\vert x_N \vert)} dx \\
	&\leq 4a^{2} e^{2\gamma R} \disp\int_{B'(0',R)} \left(\disp\int_{S^{e_{N}}_{x'}} (R+|x_{N}|)^{2\delta} e^{2\gamma \vert x_N \vert} dx_N\right)dx' \\
	& \leq 4a^{2} \mathcal{L}^{N-1}(B'(0',1)) \mathtt{S}_{e_{N}}(\Omega)(1+\beta)^{2\delta} e^{2\gamma \beta}  R^{N+2\delta-1} e^{2\gamma R}, \qquad \forall \, R >1, 
	\end{align}
where in the latter we have used that $\Omega \subseteq \{x=(x',x_N) \in \R^N  \,  : \, -\beta <x_N < \beta \}$, with 
$\beta >0$, since $ \Omega$ is bounded in the direction $e_N$.

The latter inequality contradicts \eqref{eq12quater}, since $ \gamma \in \left[ 0,  \frac{\pi}{4\mathtt{S}_{e_N}(\Omega) \sqrt{e-1}} \right)$ implies $2 \gamma < {\frac{1}{h}}$. \qed

\section{Some uniform estimates in unbounded domains} \label{SectB}
In order to prove some of our results  we need to  establish some uniform estimates for solutions to semilinear problems on epigraphs satisfying a uniform exterior cone condition.  Let us recall that an open set $\omega \subset \R^N$ (not necessarily an epigraph) satisfies a \textit{uniform exterior cone condition} if for any $x_{0} \in \partial \omega$ there exists a finite right circular cone $V_{x_{0}},$ with vertex $x_{0},$ such that 
$\overline{\omega} \cap V_{x_{0}}=\{x_0 \}$ and the cones $V_{x_{0}}$ are all congruent to some fixed cone $V.$
The cone $V$ is called the \textit{reference cone}.\\
Any globally Lipschitz-continuous epigraph satisfies a uniform exterior cone condition, but the converse is not true. 
Indeed, the epigraph defined by the function $ x \mapsto e^x$ is bounded from below, it satisfies a uniform exterior cone condition (by convexity) without being uniformly continuous.

In the following we consider a (merely) continuous function $h\,:\, \mathbb{R}^{N-1}\rightarrow \mathbb{R}$, its epigraph 
\[
\omega\,:=\, \left\lbrace x=(x',x_N) \in\mathbb{R}^{N-1}\times\mathbb{R} \,:\, x_N> h(x') \right\rbrace 
\]
and, for any $R>0$ and any $ \tau >\displaystyle\sup_{B'(0',R)} \, \vert h \vert$, we set 
		\begin{equation}\label{epi-cilindro}
			\mathfrak{C}^{h}(0',R,\tau)= \left\lbrace x=(x',x_N)\in \R^N, \, x' \in B'(0',R) \, \text{ and } \,\, h(x')< x_N<\tau 
			\right\rbrace , 
		\end{equation}
the intersection of the epigraph $\omega$ with the truncated cylinder $B'(0',R) \times \left( -\tau,\tau \right)$ and 
        \begin{equation*}
		\widehat{\mathfrak{C}^{h}}(0',R,\tau) = \mathfrak{C}^{h}(0',R,\tau) \cup 
         \left\lbrace x=(x',x_{N}) \in B'(0',R) \times \R, \, x_{N}=h(x') \right\rbrace \, .
		\end{equation*}

\medskip

Then, the next uniform estime holds true.

\medskip

\begin{prop}\label{prop2.2}
Let $\omega$ be the epigraph of a continuous function $h : \mathbb{R}^{N-1}\rightarrow \mathbb{R}$ and suppose that 
$\omega$ satisfies a uniform exterior cone condition (with reference cone $V$).\\
Let $\widetilde{R}>0$, $ \widetilde{\tau} > \max \left\lbrace \widetilde{R}, \, \displaystyle 4\sup_{\overline{B'(0',\widetilde{R})}} \vert h \vert\right\rbrace $ and 
$u \in H^{1}(\mathfrak{C}^{h}(0',\widetilde{R},\widetilde{\tau})) \cap C^0(\overline{\mathfrak{C}^{h}(0',\widetilde{R},\widetilde{\tau})})$ satisfy 
\begin{equation}\label{equationestimé1}
	\left\{
	\begin{array}{ccc}
		-\Delta u=f(u)  & \text{in} & \mathcal{D}'(\mathfrak{C}^{h}(0',\widetilde{R},\widetilde{\tau})),\\
		u=0 & \text{on} & \dr \omega \cap \widehat{\mathfrak{C}^{h}}(0',\widetilde{R},\widetilde{\tau}),
	\end{array}
	\right.
\end{equation}     
where $ f \in {Lip}_{loc}(\R)$. Then, there exists $\alpha=\alpha(N,V) \in (0,1)$ such that $u \in C^{0,\alpha} (\overline{\mathfrak{C}^{h}(0',r,t)})$ for any $0<r<\frac{\widetilde{R}}{2}$ and any $\disp\sup_{B'(0',r)}|h|<t<\frac{3}{4}\widetilde{\tau}$
and 
\begin{equation}\label{estimeéelliptique1}
     \|u\|_{C^{0,\alpha}(\overline{\mathfrak{C}^{h}(0',r,t)})} \leq  C \left( L_f + \vert f(0) \vert +1  \right)  
     \left( \|u\|_{L^\infty      (\mathfrak{C}^{h}(0',\widetilde{R},\widetilde{\tau}))} +1 \right),
\end{equation}
where $C=C(r,\widetilde{R},\widetilde{\tau},V,N) >0 $ and $L_f$ is the Lipschitz constant of $f$ on the interval 
$ \left[ - \|u\|_{L^\infty (\mathfrak{C}^{h}(0',\widetilde{R},\widetilde{\tau}))}, \|u\|_{L^\infty (\mathfrak{C}^{h}(0',\widetilde{R},\widetilde{\tau}))} \right] $.
\end{prop}

In order to prove the previous proposition we need two  preliminary results.

\begin{lem}\label{lem2.4}
	Let $N \geq 1$, $U \subset \R^{N}$ be a domain and $u \in C^{2}(U)$ be a solution of
	\begin{equation}\label{equation_de_poisson}
		\begin{array}{ccc}
			-\Delta u=f & \text{in} & U,
		\end{array}
	\end{equation}
	with $f \in C^{0}(U)$. \\
	Let $\delta >0$ and $y \in U$ such that $\overline{B(y,\delta)} \subset U$.  Then, for any 
	$x \in B(y,\frac{\delta}{2})$ we have
	\begin{equation}\label{inégalité_de_brandt}
		|u(x)-u(y)| \leq \dfrac{\sqrt{N}}{2^{1-\gamma}}\Big(2N\|u\|_{L^{\infty}(\overline{B(y,\delta)})} \delta^{-\gamma}+
		\|f\|_{L^{\infty}(\overline{B(y,\delta)})}\delta^{2-\gamma}\Big)|x-y|^{\gamma},
	\end{equation}
	with $\gamma \in ]0,1]$.
\end{lem}
\begin{proof}
	By the mean value theorem, for any $\gamma \in ]0,1]$,
	\begin{equation}\label{Mean_value_theorem}
		\begin{split}
			|u(x)-u(y)| \leq \disp\sup_{z \in \overline{B(y,\frac{\delta}{2})}} |\nabla u(z)| |x-y| \leq 
			\frac{\delta^{1-\gamma}}{2^{1-\gamma}}\disp\sup_{z \in \overline{B(y,\frac{\delta}{2})}}|\nabla u(z)| |x-y|^{\gamma}.
		\end{split}
	\end{equation}
	Moreover,  for any $z \in \overline{B(y,\frac{\delta}{2})}$ we have $\overline{B(z,\frac{\delta}{2})}\subset \overline{B(y,\delta)}$, and so by Brandt's inequality (see \cite{br} or Theorem $3.9$ in \cite{gt}) 
	\begin{equation*}\label{vraie_inégalité_de_brandt}
		|\dr_{i}u(z)|\leq \frac{2 N}{\delta}\|u\|_{L^{\infty}(\dr B(z,\frac{\delta}{2}))}+\frac{\delta}{4} \|f\|_{L^{\infty}(\overline{B(y,\delta)})} \qquad \forall \, i \in \{1,\cdots,N\}.
	\end{equation*}
	Thus, for any $z \in \overline{B(y,\frac{\delta}{2})}$
	\begin{equation}\label{majoration_du_gradient}
		|\nabla u(z)|\leq \sqrt{N}\Big(\frac{2 N}{\delta}\|u\|_{L^{\infty}(\overline{ B(y,\delta)})}+\frac{\delta}{4} \|f\|_{L^{\infty}(\overline{B(y,\delta)})}\Big).
	\end{equation}
	The claim then follows by combining \eqref{Mean_value_theorem} and \eqref{majoration_du_gradient}.
\end{proof}

\medskip

To prove the next result, let us recall that a domain $\omega$ satisfies a \textit{uniform exterior regularity condition}, if 
\begin{equation}\label{Condition(A)}
\exists \, \gamma>0, \, \rho_{0}>0 \quad : \quad \forall \, \rho \in (0,\rho_{0}], \quad  \forall \, x_{0}\in \partial \omega, \qquad \frac{\mathcal{L}^{N}(\omega^{c} \cap B(x_{0},\rho))}{\mathcal{L}^{N}(B(x_{0},\rho))} \geq \gamma.
\end{equation}
Condition \eqref{Condition(A)} will be denoted by $C_{\gamma,\rho_0}$. 
It is well-known that any domain that satisfies a \textit{uniform exterior cone condition}, also satisfies a \textit{uniform exterior regularity condition} $C_{\gamma,\rho_0}$, with suitable parameters $\gamma$ and $\rho_0$ depending only on the reference cone. Consequently, the following result applies to any epigraph that satisfies a \textit{uniform exterior cone condition}.

\begin{prop}\label{prop0.1}
	Let $\omega$ be the epigraph of a continuous function $h : \mathbb{R}^{N-1}\rightarrow \mathbb{R}$ and suppose that 
	$\omega$ satisfies the condition $C_{\gamma,\rho_0}$.\\
	Let $\widetilde{R}>0$, $ \widetilde{\tau} > \max \left\lbrace \widetilde{R} , \displaystyle\sup_{\overline{B'(0',\widetilde{R})}} \vert h \vert\right\rbrace $, $ f \in L^{N}(\mathfrak{C}^{h}(0',\widetilde{R},\widetilde{\tau}))$ and 
	$u \in H^{1}(\mathfrak{C}^{h}(0',\widetilde{R},\widetilde{\tau})) \cap C^0(\overline{\mathfrak{C}^{h}(0',\widetilde{R},\widetilde{\tau})})$  satisfy
	\begin{equation}\label{equationestimé1}
		\left\{
		\begin{array}{ccc}
			-\Delta u=f  & \text{in} & \mathcal{D}'(\mathfrak{C}^{h}(0',\widetilde{R},\widetilde{\tau})),\\
			u=0 & \text{on} & \dr \omega \cap \widehat{\mathfrak{C}^{h}}(0',\widetilde{R},\widetilde{\tau}).
		\end{array}
		\right.
	\end{equation}     
Then, there are constants $C=C(N,\rho_{0},\gamma,\widetilde{R})>0$ and $\alpha=\alpha(N,\gamma) \in (0,1)$ such that for any $r>0$ and for any $x_0 \in \partial \omega \cap  \left( \overline{B'(0',\widetilde{R}/2)} \times \R \right) $,
\begin{equation*}
	\underset{\widehat{\mathfrak{C}^{h}}(0',\widetilde{R},\widetilde{\tau}) \cap B(x_{0},r)}{\text{osc}} u \leq C (\|u\|_{L^{\infty}(\mathfrak{C}^{h}(0',\widetilde{R},\widetilde{\tau}))}+\|f\|_{L^{N}(\mathfrak{C}^{h}(0',\widetilde{R},\widetilde{\tau}))}) r^{\alpha}.
\end{equation*} 
\end{prop}

\begin{proof}
To simplify the exposition we set $\Omega= \mathfrak{C}^{h}(0',\widetilde{R},\widetilde{\tau})$, $T=\partial \omega \cap  \left( \overline{B'(0',\widetilde{R}/2)} \times \R \right) $, $\overline{r}=\min(\frac{\widetilde{R}}{8},\frac{\rho_0}{2},1)$ and, for $r>0$ and $x_{0} \in T$, we also let $\Omega_{r}=\Omega\cap B(x_{0},r).$ \\
For any $ r \in (0, \overline{r}]$ we set $M_{4}=\displaystyle\sup_{\Omega_{4r}} \, u $, $m_{4}=\displaystyle\inf_{\Omega_{4r}} \, u $, $M_{1}=\displaystyle\sup_{\Omega_{r}} \, u$ and $m_{1}=\displaystyle\inf_{\Omega_{r}} \, u$ and we apply Theorem 8.26 in \cite{gt} to the function $W_4 =M_4-u$, with $q=2N$ and $p=1,$ to obtain
\begin{equation}\label{pre-premiere_inegaliteW4}
		\begin{split}
			r^{-N}\|W_{4,m}^{-}\|_{L^{1}(B_{2r}(x_{0}))}&\leq C_{1}(N) \left[ \displaystyle \inf_{B_{r}(x_{0})} W_{4,m}^{-} +r\|f\|_{L^{N}(\Omega)}\right],\\
		\end{split}
	\end{equation} 
where $m = \inf_{\partial \Omega \cap  B_{4r}(x_{0})} W_4$ (here we have used the notations of Theorem 8.26 in \cite{gt}).
To proceed further, we observe that $m= M_4$, since  $\widetilde{\tau} > \widetilde{R} \geq 4 \overline{r} \geq 4r$ implies 
$\partial \Omega \cap  B_{4r}(x_{0}) \subset \partial \omega \cap \widehat{\mathfrak{C}^{h}}(0',\widetilde{R},\widetilde{\tau})$. Hence 
   \begin{equation}\label{premiere_inegaliteW4}
		\begin{split}
			r^{-N}\|W_{4,m}^{-}\|_{L^{1}(B_{2r}(x_{0}))}&\leq C_{1}(N) \left[ \displaystyle\inf_{B_{r}(x_{0})} W_{4,m}^{-} + 
			r\|f\|_{L^{N}(\Omega)} \right] \\
			&\leq C_{1}(N)[M_{4}-M_{1}+r\|f\|_{L^{N}(\Omega)}]. 
		\end{split}
	\end{equation} 
Moreover, since $2r \leq \rho_{0}$, we also have  
\begin{equation}\label{deuxieme_inegaliteW4}
	r^{-N} \|W_{4,m}^{-}\|_{L^{1}(B_{2r}(x_{0}))} \geq r^{-N} M_4  \mathcal{L}^{N}(\Omega^{c} \cap B(x_0,2r)) 
	\geq M_4   \gamma {\mathcal{L}^N}(B(0,2)) .
\end{equation}
In the latter we have used that the epigraph $\omega$ satisfies the uniform condition $C_{\gamma,\rho_0}$ at $x_0$.

Therefore, from \eqref{premiere_inegaliteW4} and \eqref{deuxieme_inegaliteW4} we deduce that 
\begin{equation}\label{troisieme_inegaliteW4}
	{\mathcal{L}^N}(B(0,2)) \gamma M_{4} \leq C_{1}[M_{4}-M_{1}+r\|f\|_{L^{N}(\Omega)}].
\end{equation}
The same argument applied to the function $w_{4} =u - m_{4}$ yields 
\begin{equation}\label{inegalitew4}
	- {\mathcal{L}^N}(B(0,2)) \gamma m_4 \leq C_{1}[m_{1}-m_{4}+r\|f\|_{L^{N}(\Omega)}],
\end{equation}
hence, by adding \eqref{troisieme_inegaliteW4} and \eqref{inegalitew4} we obtain 
\begin{equation*}
	  {\mathcal{L}^N}(B(0,2)) \gamma (M_{4}-m_{4}) \leq  C_1  (M_{4}-m_{4}) -C_{1} (M_{1}-m_{1}) +2rC_{1} \|f\|_{L^{N}(\Omega)}.
\end{equation*}
Thus
\begin{equation}
	\underset{\Omega_{r}}{\text{osc}} \, u  \leq \Big(1 - \frac{{\mathcal{L}^N}(B(0,2)) \gamma}{C_{1}} \Big) \underset{\Omega_{4r}}{\text{osc}} \, u  +2r \|f\|_{L^{N}(\Omega)}.
\end{equation}
Now, we observe that the function $w(R):=\underset{\Omega_R}{\text{osc}} \, u  + 2R \|f\|_{L^{N}(\Omega)}$ is well-defined and non-decreasing on $(0, \overline{r}]$ and satisfies 
\begin{equation*}
	\begin{split}\
	     w(R/4) & = \underset{\Omega_{R/4}}{\text{osc}} \, u  + \frac{R}{2} \|f\|_{L^{N}(\Omega)} 
		\leq \Big(1 - \frac{{\mathcal{L}^N}(B(0,2)) \gamma}{C_{1}} \Big) \underset{\Omega_{R}}{\text{osc}} \, u  + \frac{R}{2} 
		\|f\|_{L^{N}(\Omega)} +\frac{R}{2} \|f\|_{L^{N}(\Omega)}\\
		& = \Big(1 - \frac{{\mathcal{L}^N}(B(0,2)) \gamma}{C_{1}} \Big) \underset{\Omega_{R}}{\text{osc}} \, u  +\frac{1}{2} \times 2R \|f\|_{L^{N}(\Omega)} \\
		& \leq \max \left[ 1 - \frac{{\mathcal{L}^N}(B(0,2)) \gamma}{C_{1}},\frac{1}{2}\right]  \left[ \underset{\Omega_{R}}{\text{osc}} \, u  + 2R \|f\|_{L^{N}(\Omega)}\right] = C_2 w(R) \\
	\end{split}
\end{equation*}
where $ C_2 = C_2(N,\gamma) := \max \left[ 1 - \frac{{\mathcal{L}^N}(B(0,2)) \gamma}{C_{1}},\frac{1}{2}\right] \in (0,1)$. Hence, we can apply Lemma $8.23$ in \cite{gt} to get  
\begin{equation}
w(R) \leq C_3 \left( \frac{R}{\overline{r}}\right) ^{\alpha} w(\overline{r}) \qquad \forall \, R \in (0, \overline{r}],
\end{equation}
where $ C_3=C_3(N,\gamma)>0$ and $\alpha=\alpha(N,\gamma) \in (0,1)$. From the latter we immediately infer that 
\begin{equation}\label{rpetit}
	\underset{\Omega_{r}}{\text{osc}}\, u \leq w(r)\leq C_{3} \Big(\frac{r}{\overline{r}}\Big)^{\alpha}w(\overline{r})\leq 2 C_{3} \Big(\frac{r}{\overline{r}}\Big)^{\alpha} ( \|u\|_{L^{\infty}(\Omega)} + \|f\|_{L^{N}(\Omega)} ) \qquad \forall \, r \in (0, \overline{r}].
\end{equation}

\noindent Now, if $r>\overline{r}$ then for any $x,y \in \Omega_{r}$
\begin{equation*}
	u(x)-u(y)\leq |u(x)-u(y)| \leq 2\|u\|_{L^{\infty}(\Omega)}\times 1 \leq 2\|u\|_{L^{\infty}(\Omega)} \Big(\frac{r}{\overline{r}}\Big)^{\alpha},
\end{equation*} 
thus 
\begin{equation}\label{rgrand}
	\underset{\Omega_{r}}{\text{osc}} \, u \leq 2\|u\|_{L^{\infty}(\Omega)} \Big(\frac{r}{\overline{r}}\Big)^{\alpha} \qquad \forall \, r > \overline{r}. 
\end{equation}
Finally, \eqref{rpetit} and \eqref{rgrand} imply the desired conclusion. 
\end{proof}

We are now ready to prove  Proposition \ref{prop2.2}.

\textit{Proof of Proposition \ref{prop2.2}.} \\  Set $\Omega=\mathfrak{C}^{h}(0',\widetilde{R},\widetilde{\tau})$, $R_{1}=\min(\frac{\widetilde{R}}{2}-r,\frac{\widetilde{\tau}}{2})$ and pick $x,y  \in \mathfrak{C}^{h}(0',r,t)$ with $x\neq y $. \\ 

\noindent \textit{1.} If $|x-y|\geq \dfrac{R_{1}}{4}$ then, for any $\gamma \in (0,1)$, we have
			\begin{equation}\label{estimée_1}
				\begin{split}
				|u(x)-u(y)| & \leq 2\|u\|_{L^{\infty}(\Omega)} =  2\|u\|_{L^{\infty}(\Omega)} \Big(\frac{4}{R_{1}}\Big)^{\gamma}\Big(\frac{R_{1}}{4}\Big)^{\gamma} \\
				& \leq 2  \Big(\frac{4}{R_{1}}\Big)^{\gamma}  \|u\|_{L^{\infty}(\Omega)} |x-y|^{\gamma} = M_1  \|u\|_{L^{\infty}(\Omega)}  |x-y|^{\gamma}.
			\end{split}
			\end{equation}
\textit{2.} If $|x-y|< \dfrac{R_{1}}{4}$, then we consider the set  
$$T: = \left\lbrace x=(x',x_{N})\in \R^{N} \, : \, x'\in \overline{B'(0',\widetilde{R}/2)} \quad  \text{and} \quad  x_{N}=h(x') \right\rbrace $$
and, by observing that $x$ and $y$ play a symmetric role,  we distinguish (only) three cases. 

\medskip
			
\textit{Case 2.1(see Figure \ref{fig:cas1})} : $d(y,T)>\frac{R_{1}}{2}$.

\begin{figure}[!h]
	\centering
	\includegraphics[width=0.6\textwidth]{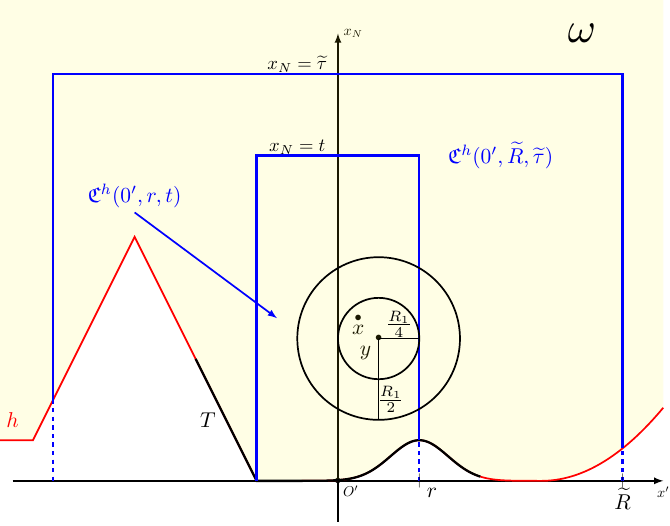}
	\caption{Case 2.1}
	\label{fig:cas1}
\end{figure}

\noindent In this case we have $\overline{B(y,\frac{R_{1}}{2})}\subset \mathfrak{C}^{h}(0',\widetilde{R},\widetilde{\tau}).$ To see this, we first observe that $\overline{B(y,\frac{R_{1}}{2})}\subset \omega$, thanks to the assumption $d(y,T)>\frac{R_{1}}{2}$ and by the definition of $R_1$, and thus, for any  $z \in \overline{B(y,\frac{R_{1}}{2})}$, we have 
			\begin{equation*}
				h(z') < z_{N}=z_{N}-y_{N}+y_{N}< z_{N}-y_{N}+t \leq \frac{R_{1}}{2}+\frac{3\widetilde{\tau}}{4} < \frac{\widetilde{\tau}}{4}+ \frac{3\widetilde{\tau}}{4} = \widetilde{\tau},
			\end{equation*}
			and 
			\begin{equation*}
				|z'|\leq |z'-y'|+|y'| \leq \frac{R_{1}}{2}+r< \frac{\widetilde{R}}{2}-r+r<\widetilde{R}. 
			\end{equation*}
Now by applying  Lemma \ref{lem2.4} \footnote{\, Note that $u \in C^2$ by standard elliptic regularity results, since 
$u$ is continuous and $f$ is locally Lipschitz-continuous.}  with $\delta=\frac{R_{1}}{2}$ and for any $\gamma \in (0,1)$, we deduce
			\begin{equation}
				\begin{split}
					|u(x)-u(y)| &\leq \dfrac{\sqrt{N}}{2^{1-\gamma}}\Big(2N\|u\|_{L^{\infty}(\Omega)} 
					\Big(\frac{R_{1}}{2}\Big)^{-\gamma}+\|f(u)\|_{L^{\infty}(\Omega)}\Big(\frac{R_{1}}{2}\Big)^{2-\gamma}\Big)
					|x-y|^{\gamma}\\
					&\leq M_{2}(L_{f}+|f(0)|+1)(\|u\|_{L^{\infty}(\Omega)}+1)|x-y|^{\gamma}
				\end{split}\label{estimée_2}
			\end{equation}
		where $M_{2}=\dfrac{\sqrt{N}}{2^{1-\gamma}}\Big(2N \Big(\dfrac{R_{1}}{2}\Big)^{-\gamma}+\Big(\dfrac{R_{1}}{2}\Big)^{2-\gamma}\Big)$.\\

\textit{Case 2.2 (see Figure \ref{fig:cas2})} : $d(x,T)\leq d(y,T) \leq \dfrac{R_{1}}{2}$ and $|x-y| \geq \frac{1}{4}d(y,T).$\\

\begin{figure}[!h]
		\centering
		\includegraphics[width=0.6\textwidth]{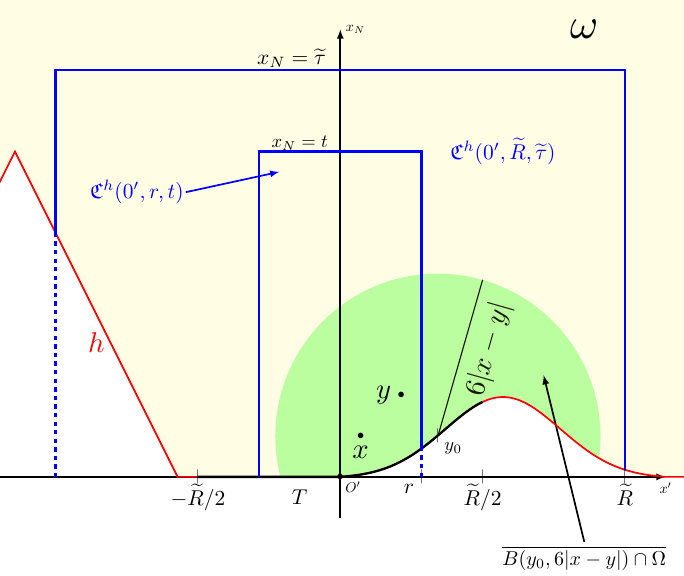}
		\caption{Case 2.2}
		\label{fig:cas2}
\end{figure}

 Since $T$ is a compact set, there exists  $y_{0}=(y_{0}',h(y_{0}')) \in T$ such that $d(y,T)=|y-y_{0}|$. \\Therefore 
			 \begin{equation*}
			x,y \in B(y_{0},6|x-y|),
			\end{equation*}
			 since 
			\begin{align*}
				&|x-y_{0}|\leq |x-y|+|y-y_{0}|\leq  |x-y|+4|x-y|=5|x-y|,\\
				\intertext{and}
				&|y-y_{0}|=d(y,T)\leq 4|x-y|.
			\end{align*}
		Now, we apply Proposition \ref{prop0.1} with $r=6|x-y|$ and $x_{0}=y_{0}$, thus there exist constants $C'=C'(N,V,\widetilde{R})>0$ and $\alpha=\alpha(N,V) \in (0,1)$ such that  
		\begin{equation}
			\begin{split}
			|u(x)-u(y)|&\leq 6^{\alpha}C' (\|u\|_{L^{\infty}(\Omega)}+\|f(u)\|_{L^{N}(\Omega)}) |x-y|^{\alpha}\\
			&\leq M_{3}(L_{f}+|f(0)|+1)(\|u\|_{L^{\infty}(\Omega)}+1)|x-y|^{\alpha}.
		\end{split}\label{estimé3}
		\end{equation}
where $M_{3}=6^{\alpha}C' \left[ 1 + \left( 2\widetilde{\tau}\mathcal{L}^{N-1}(B'(0',1)) \widetilde{R}^{N-1} \right)^{\frac{1}{N}} \right] $.\\

\textit{Case 2.3 (see Figure \ref{fig:cas3})}: $d(x,T)\leq d(y,T) \leq \dfrac{R_{1}}{2}$ and $ |x-y| < \frac{1}{4} d(y,T)$. \\ 

\begin{figure}[!h]
		\centering
		\includegraphics[width=0.6\textwidth]{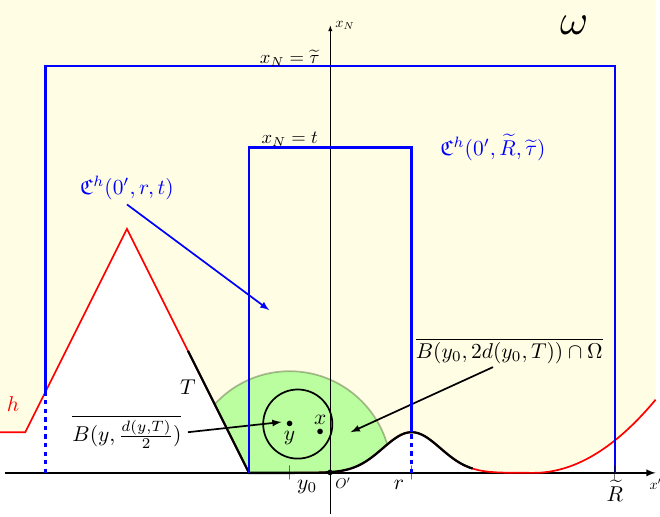}
		\caption{Case 2.3}
		\label{fig:cas3}
\end{figure}

By applying Lemma  \ref{lem2.4} with $\delta=\frac{d(y,T)}{2}$ and $\gamma= \alpha$, 
			\begin{equation}
				\begin{split}
					|u(x)-u(y)| &\leq \frac{\sqrt{N}}{2^{1-\alpha}}\Big( \|f(u)\|_{L^{\infty}(\Omega)} \Big(\frac{d(y,T)}{2}\Big)^{2-\alpha}\\
					&+2N \|u\|_{L^{\infty}(\overline{ B(y,\frac{d(y,T)}{2})})} \Big(\frac{d(y,T)}{2}\Big)^{-\alpha}\Big)  |x-y|^{\alpha} .
				\end{split}\label{brandt-split}
			\end{equation}
			Now, we want to estimate  $\|u\|_{L^{\infty}(\overline{ B(y,\frac{d(y,T)}{2})})} $.  Let $ z \in \overline{ B(y,\frac{d(y,T)}{2})}$ and $y_{0}=(y'_{0},h(y'_{0}))\in T$ such that $d(y,T)=|y_{0}-y|.$ Then $z \in \Omega \cap B(y_{0},2d(y,T))$
			since
			\begin{equation*}
				|z-y_{0}|=|z-y+y-y_{0}|\leq |z-y|+|y-y_{0}| \leq \frac{3}{2}d(y,T) < 2d(y,T).
			\end{equation*}
			Thus we can apply Proposition \ref{prop0.1} with $r=2d(y,T)$ and $x_{0}=y_{0}$ to get
			\begin{equation*}
				|u(z)|= |u(z) - u(y_0)| \leq 2^{\alpha}C'(\|u\|_{L^{\infty}(\Omega)}+\|f(u)\|_{L^{N}(\Omega)})d(y,T)^{\alpha}.
			\end{equation*}
			Hence
			\begin{equation}
				\|u\|_{L^{\infty}(\overline{B(y,\frac{d(y,T)}{2})})}\leq 2^{\alpha}C'(\|u\|_{L^{\infty}(\Omega)}+\|f(u)\|_{L^{N}(\Omega)})d(y,T)^{\alpha},
				\label{estimée_3}
			\end{equation}
			and by \eqref{brandt-split} and \eqref{estimée_3}, we have
			\begin{equation}
				|u(x)-u(y)|\leq M_{4}(L_{f}+|f(0)|+1)(\|u\|_{L^{\infty}(\Omega)}+1)|x-y|^{\alpha}.\label{estimée_4}
			\end{equation}
		where $M_{4}=\dfrac{\sqrt{N}}{2^{1-\alpha}} \left\lbrace \Big( \dfrac{R_{1}}{4}\Big)^{2-\alpha} +2N2^{2\alpha}C'
	\left[ 1 + \left( 2\widetilde{\tau}\mathcal{L}^{N-1}(B'(0',1)) \widetilde{R}^{N-1} \right)^{\frac{1}{N}} \right] \right\rbrace  $.\\

\noindent The desired conclusion \eqref{estimeéelliptique1} then follows from \eqref{estimée_1},\eqref{estimée_2},\eqref{estimé3} and \eqref{estimée_4} by taking  $C=\max(M_{1},M_{2},M_{3},M_{4})$. \qed

\section{Proofs} \label{SectC}
Thanks to the translation invariance of problem \eqref{NonLin-PoissonEq}, we may and do suppose that the 
function $g\,:\, \mathbb{R}^{N-1}\rightarrow \mathbb{R}$, defining $ \partial \Omega $, satisfies 
$\inf_{\mathbb{R}^{N-1}} \, g=0$.

\smallskip

The proofs of our main results are based on the moving planes method {\textit {suitably adapted to the geometry of the epigraph $\Omega$}}. To this end, we first set the notations that will be used in our analysis. 

\medskip

For $0<a<b$ and  $\lambda>0$ we set (see Figures \ref{fig:cas21} and \ref{fig:cas22}):

\smallskip
	
$ \Sigma_{\lambda} := \{ \,  x \in \R^N \, : \, 0 < x_N < \lambda  \, \}$,		
		
$\Sigma_{b}^{g}=\lbrace x=(x',x_{N})\in \R^{N} \, : \,  g(x')<x_{N}<b \rbrace$,  

$\Sigma _{a,b}^{g}=\lbrace x=(x',x_{N})\in \R^{N} \, : \, g(x')+a<x_{N}<b \rbrace$,

$$
\forall x \in \Sigma_{\lambda}^{g}, \qquad u_{\lambda}(x)=u(x',2\lambda-x_{N}) .
$$

\begin{multicols}{2}
	
	\begin{figure}[H]
		\centering
		\includegraphics[width=8.1cm]{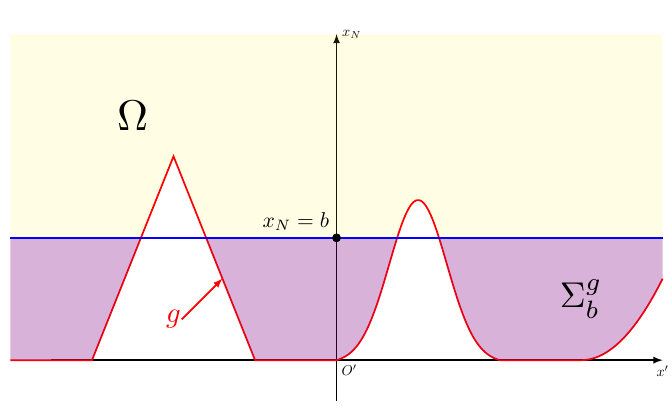}	
		\caption{Case 2.1}
		\label{fig:cas21}
	\end{figure}
	\begin{figure}[H]
		\centering
		\includegraphics[width=8.3cm]{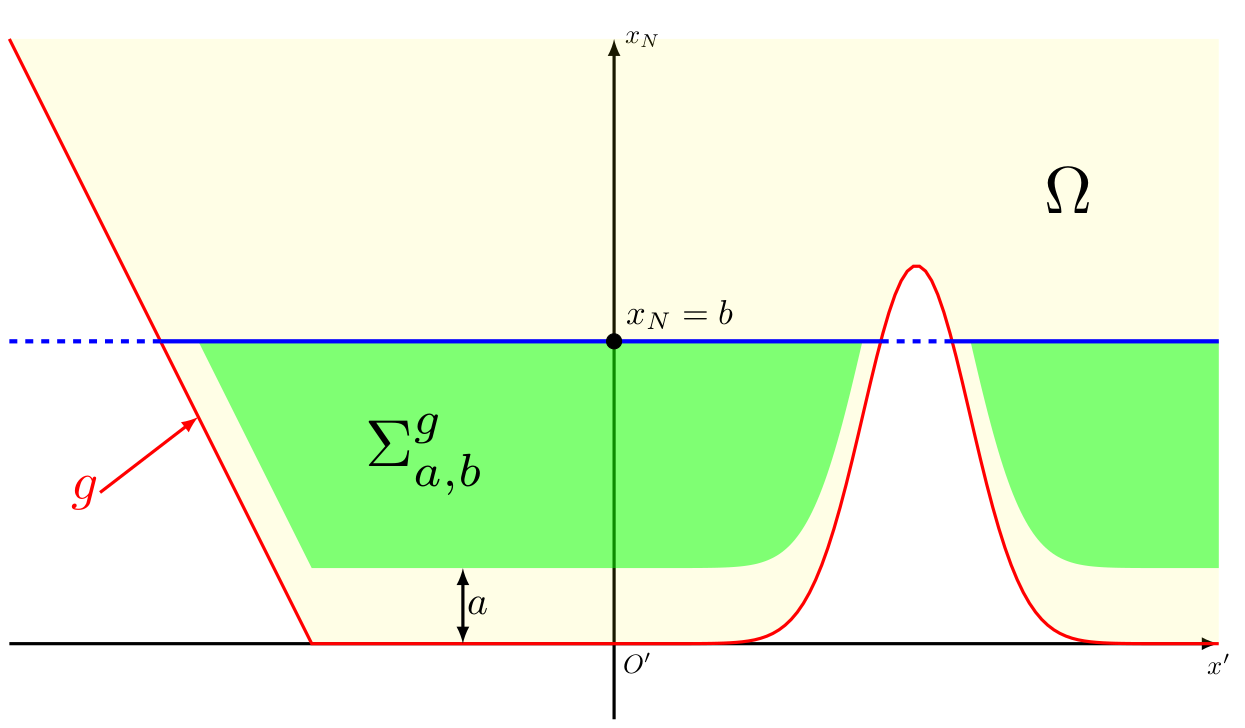}
		\caption{Case 2.2}
		\label{fig:cas22}
	\end{figure}
\end{multicols}

In the following, for any subset $S \subseteq \R^N$, we denote by $UC(S)$ the set of uniformly continuous functions defined on $S$. 

\bigskip

\textit{Proof of Theorem \ref{TH1}.}

We set $\Lambda:= \lbrace t >0  \,\, : \,\, u \leqslant u_{\theta} \,\,\, \text{in} \,\,\, \Sigma_{\theta}^{g} \,, \,\,  \forall \, 
0<\theta< t \rbrace$ and we aim at proving that 
$$
\tilde{t}:=\sup \Lambda = +\infty.
$$
To this end we split the remaining part of the proof into three steps.

\medskip

\textit{Step 1 : $\Lambda$ is not empty.}   

For any $\theta >0$ small enough, we have
		$u,u_{\theta} \in H^{1}_{\text{loc}}(\overline{\Sigma^{g}_{\theta}})\cap UC(\overline{\Sigma_{\theta}^{g}})$ and 
		\begin{equation*}
			\left\{
			\begin{array}{ccl}
				- \Delta u-f(u) =0=-\Delta u_{\theta}-f(u_{\theta}) & \text{in} & \Sigma_{\theta}^{g},\\  
				u\leq  u_{\theta} & \text{on} & \partial \Sigma_{\theta}^{g}.
			\end{array}
			\right.
		\end{equation*} 
For the latter, notice that $\partial \Sigma_{\theta}^{g}=(\lbrace x_{N}=\theta \rbrace \cap \overline \Omega ) \cup (\partial \Omega \cap \lbrace x_{N}<\theta\rbrace)$ and so

		\begin{itemize}
			\item if $x \in \lbrace x_{N}=\theta \rbrace \cap \overline\Omega,$ then $u(x)=u_{\theta}(x),$
			\item if $x \in \partial \Omega \cap \lbrace x_{N}< \theta \rbrace,$ then $u(x)=0$ and 
			$u_{\theta}(x)=u(x',2\theta -x_{N}) > 0,$ since $(x',2 \theta -x_{N} )\in \Omega$.
		\end{itemize}

Since $\Sigma^{g}_{\theta} \subseteq \{ \,  x \in \R^N \, : \, 0 < x_N < \theta  \, \}$, we also have that 
$\mathtt{S}_{e_N} ( \Sigma^{g}_{\theta} ) \leq \theta $. Moreover, $u$ has at most linear growth on $\Sigma^{g}_{\theta}$, since $u$ is uniformly continuous on $\Sigma^{g}_{\theta}$ by assumption. Therefore, we can apply Theorem \ref{thComp2} (see Remark \ref{rem-striscia-exp}), with any $\theta < \varepsilon(L_{f}, \gamma),$\footnote{\,  Here $\gamma$ can be any positive real number.} to get 
		\begin{equation*}
			u \leq  u_{\theta} \quad \mbox{on} \quad \Sigma_{\theta}^{g}.
		\end{equation*}
		Hence, $(0,\varepsilon(L_{f}, \gamma)) \subset \Lambda$.
		
\medskip		
		
\textit{Step 2 : $\tilde{t}=\sup \Lambda = +\infty$. } 
		  
If $\tilde{t}:=\sup \Lambda < +\infty$, then we have 
\begin{prop}\label{PropDS} 
		For every $ \delta \in (0, \frac{\tilde t}{2})$  there is $ \varepsilon(\delta) >0 $ such that 
		\begin{equation}\label{enunciatoDS}
		\forall  \, \varepsilon \in (0, \varepsilon(\delta))  \qquad u \le u_{\tilde t + \varepsilon}  \quad \text{on} \quad \overline{\Sigma_{\delta,\tilde{t}-\delta}^{g}}.
		\end{equation}
	\end{prop}

{\textit {Proof of Proposition \ref{PropDS}.}}  If the claim were not true, there would exist 
$ \delta \in (0, \frac{\tilde t}{2})$ such that 
	\begin{equation}\label{contrad-enunciatoDS}
	\forall  \, k \geq 1 \quad \exists  \, \varepsilon_k \in \left( 0, \frac{1}{k}\right), \, \exists \, x^k \in \overline{\Sigma_{\delta,\tilde{t}-\delta}^{g}} \quad : \quad u(x^k) > u_{\tilde t + \epsilon_k}(x^k),
	\end{equation}
and so 
\begin{equation}\label{controllo-AA}
\delta \leq g((x^{k})')+ \delta < x_{N}^{k}< \tilde{t}-\delta, \qquad \forall k \geq 1.
\end{equation}
Therefore, the sequence $(x^{k}_{N})$ is bounded, and so, up to a subsequence, we may and do suppose that 
$x^k_N \to x^{\infty}_N$, as $k \to \infty$.

Now, set  
\begin{equation}\label{trasl-g_k}
\begin{array}{ccc}
g_{k}(x')=g(x' + (x^{k})') \qquad \forall x' \in \R^{N-1}, \quad \forall k \geq 1,
\end{array} 
\end{equation}
and observe that the sequence $(g_k)$ is uniformly equicontinuous on  $\R^{N-1}$ (since 
$g\in UC(\R^{N-1})$) and that $ 0 \leq g_{k}(0')=g((x^{k})') \leq \tilde{t}$, thanks to \eqref{controllo-AA}. Therefore,  by Ascoli-Arzelà theorem (and a standard diagonal procedure) there exists a function $g_{\infty} \in UC(\R^{N-1})$ such that, up to a subsequence, $g_{k} \to g_{\infty}$ in $C^{0}_{\text{loc}}(\R^{N-1})$. \\
We also observe that, passing to the limit in \eqref{controllo-AA}, we obtain
\begin{equation}\label{controllo-AAinfinito}
\delta \leq g_{\infty}(0')+ \delta \leq x^{\infty}_N \leq \tilde{t}-\delta.
\end{equation}

For any $k \geq 1,$ let us consider $\Omega^{k}=\{(x',x_N) \in \R^N, \, x_N>g_{k}(x') \}$,  the epigraph of $g_k$, $\Omega^{\infty}=\lbrace (x',x_N)\in \R^N, \, x_N >g_{\infty}(x') \rbrace$, the epigraph of $g_{\infty}$ and define 
	\begin{equation}\label{ext-u}
				\tilde{u}(x)=\left\{
				\begin{array}{crl}
				u(x) & \text{if}& x \in \overline{\Omega},\\
				0 & \text{if} & x \in \R^{N} \setminus \overline{\Omega}.
				\end{array}
				\right.
	\end{equation}
Clearly, $\tilde{u} \in UC(\{x_{N}<R\})$ for every $R >0$, and so the sequence $(\tilde{u}_{k})$ defined by
\begin{equation}\label{trasl-tilde-u_k}
\begin{array}{ccc}
\tilde{u}_{k}(x)=\tilde{u}(x'+(x^{k})',x_N) \qquad \forall \, x=(x',x_N)  \in \R^{N}, \quad \forall k \geq 1,
\end{array} 
\end{equation}
 is uniformly equicontinuous on $\{x_{N}<R\},$ for any $ R>0$. We also notice that, for $k$ large enough,  
the point $ (0',-1)$ belongs to  $\R^{N} \setminus \overline{\Omega_k}$, since $g_{k} \to g_{\infty}$ in $C^{0}_{\text{loc}}(\R^{N-1})$ and $ g_k \geq 0$ on $ \R^{N-1}$. Hence, $\tilde{u}_{k}(0',-1) =0$, for $k$ large enough, and the sequence $(\tilde{u}_{k}(0',-1) )$ is bounded in $\R$. Therefore, using once again the Ascoli-Arzelà theorem, there exists 
$\tilde{u}_{\infty} \in C^0(\R^N)$
such that, up to a subsequence, $\tilde{u}_{k} \to \tilde{u}_{\infty}$ in $C^{0}_{\text{loc}}(\R^N)$ and  
			\begin{equation}\label{eq18}
				\left\{
				\begin{array}{ccl}
					- \Delta \tilde{u}_{\infty}=f(\tilde{u}_{\infty}) & \text{in} & \mathcal{D}'(\Omega^{\infty}),\\ 
					\tilde{u}_{\infty}\geq 0 & \text{in} &  \Omega^{\infty},\\ 
					\tilde{u}_{\infty}=0 & \text{on} & \partial \Omega^{\infty},
				\end{array}
				\right.
			\end{equation} 
where the boundary condition follows by observing that, 
\begin{equation*}
\forall \, k \geq 1, \quad \forall \, x' \in \R^{N-1} \qquad 0 = \tilde{u}_{k}(x',g_k(x')) , 
\end{equation*}
and thus
\begin{equation*}
0 = \tilde{u}_{k}(x',g_k(x')) \longrightarrow \tilde{u}_{\infty}(x',g_{\infty}(x')) \qquad  \text{as }k
\rightarrow \infty,
\end{equation*}
thanks to the uniform convergence of $(\tilde{u}_{k})$ and  $(g_k)$ on compact sets. 

By construction, $u_{\infty}:=\tilde{u}_{\infty_{\vert \Omega^{\infty}}} \in  C^{0}(\overline{\Omega^{\infty}})$. Furthermore, $u_{\infty}  \in  C^2(\Omega^{\infty})$ by \eqref{eq18} and standard interior regularity theory for elliptic equations, and it satisfies
\begin{equation}\label{eq-u-infinito}
				\left\{
				\begin{array}{cll}
					-\Delta u_{\infty}=f(u_{\infty}) & \text{in} & \Omega^{\infty},\\ 
					u_{\infty} \geq 0 & \text{in} & \Omega^{\infty},\\ 
					\tilde{u}_{\infty}=0 & \text{on} & \partial \Omega^{\infty}.
				\end{array}
				\right.
\end{equation}

			Since $u \leq u_{\tilde{t}}$ in $\Sigma_{\tilde{t}}^{g}$, then $\tilde{u}_{k} \leq \tilde{u}_{k, \tilde{t}}$ in 
			$\Sigma^{g_{k}}_{\tilde{t}}$, for any $k \geq 1$. Passing to the limit, we get 
			\begin{equation}\label{disug-infinito-1}
				\begin{array}{c@{\;}c@{\;}c@{\;}}
					u_{\infty} \leq u_{\infty, \tilde{t}} \quad  \text{in} \quad \overline{\Sigma^{g_{\infty}}_{\tilde{t}}}.
				\end{array}
			\end{equation}
			Also, 
			\begin{equation*}
				u_{k}(0',x^{k}_{N})=u(x^{k})>u_{\tilde{t}+\varepsilon_{k}}(x^{k})=u_{k, \tilde{t}+\varepsilon_{k}}(0',x^{k}_{N}),
			\end{equation*}
			so, taking the limit as $k \rightarrow +\infty,$ we have 
\begin{equation}\label{disug-infinito-2}
u_{\infty}(0',x^{\infty}_N) \geq  u_{\infty, \tilde{t}}(0',x^{\infty}_N).
\end{equation}
In view of \eqref{controllo-AAinfinito} we see that $(0',x^{\infty}_N)\in \overline{\Sigma^{g_{\infty}}_{\delta,\tilde{t}-\delta}} \subset \Sigma^{g_{\infty}}_{\tilde{t}}$ and so, by combining \eqref{disug-infinito-1} and \eqref{disug-infinito-2},  we obtain 
\begin{equation}\label{disug-infinito-3}
u_{\infty}(0',x^{\infty}_N)= u_{\infty, \tilde{t}}(0',x^{\infty}_N).				
\end{equation}

Next we use the assumption \eqref{cond-in-zero} to prove that $ u_{\infty} >0$ on $\Omega^{\infty}$.	We first observe that, by  \eqref{cond-in-zero}, we can find $\eta >0$ and $\varepsilon_{0}>0$ such that 
\begin{equation}\label{hyp_f}
	 f(t)\geq \eta t \quad \text{for any } t \in [0,\varepsilon_{0}],
\end{equation} 		
then, we choose $R$ large enough so that the first eigenvalue $ \lambda_1(R)$ of $-\Delta$ on the open ball $B(0,R) \subset \R^{N}$ (with homogeneous Dirichlet boundary condition) satisfies $ \lambda_1(R) < \eta$. 
		
Since $g_{k} \to g_{\infty}$ in $C^{0}_{\text{loc}}(\R^{N-1})$, there is $C(R) >0$ such that 
\begin{equation*}\label{bound-pour-sliding}
\forall k \geq 1  \qquad 0 \leq g, g_k \leq C(R)\qquad \text{on} \quad B'(0', 2R) \subset \R^{N-1}.
\end{equation*}
In particular, for $ T > 2(C(R) + R)$, the open ball $ \mathcal{B}:= B((0',T),R) \subset \R^{N}$ satisfies 
\begin{equation}\label{inclusion_B_RT}
\overline{\mathcal{B}} \subset \Omega, \qquad \overline{\mathcal{B}} \subset \Omega_{k} \quad \forall k \geq 1, \quad \text{and} \quad  \overline{\mathcal{B}} \subset \Omega_{\infty}
\end{equation}
and so, also
\begin{equation}\label{inclusion_B_RT-bis}
\overline{\mathcal{B}} + ((x_{k})',0) \subset \Omega \quad \forall k \geq 1. 
\end{equation}
 
Now, set $m_0 =\min_{\overline{\mathcal{B}}}u >0$ and denote by $\phi_1$ the positive first eigenfunction of $-\Delta$ on  $\mathcal{B}$ such that $\max_{\mathcal{B}} \phi_1 = 1$. Therefore, the first eigenfunction $ \phi := \min \left\lbrace \frac{m_0}{2}, \varepsilon_0 \right\rbrace \phi_1$ satisfies 
 \begin{equation*}
	\left\{
	\begin{array}{cll}
		0<\phi < u & \text{in} & \mathcal{B},\\
		\Delta \phi+f(\phi) \geq 0 & \text{in} & \mathcal{B},\\
		\phi=0 & \text{on} & \partial \mathcal{B},
	\end{array}
	\right.
\end{equation*}
where in  the latter we have used that \eqref{hyp_f} is in force. 

Now, since $\Omega$ is an epigraph and \eqref{inclusion_B_RT}-\eqref{inclusion_B_RT-bis} hold, we can use the sliding method (see for instance \cite{BCN3}) to get 
\begin{equation*}
	\phi(x'-(x_{k})',x_{N})< u(x) \qquad  \forall \,\,  x \in \mathcal{B} +((x_{k})',0) \subset \Omega, 
\end{equation*}
that is, 
\begin{equation*}
	\phi(x) < u_{k}(x) \quad  \forall \,\,  x \in \mathcal{B}. 
\end{equation*} 
By passing to the limit, we deduce that 
$$ 0 < \phi(x) \leq u_{\infty}(x) \quad  \forall \,\,  x \in \mathcal{B},$$
therefore, by \eqref{eq-u-infinito} and the strong maximum principle, we deduce that $u_{\infty} > 0$ on $\Omega^{\infty}$.

Now we are ready to complete the proof of  Proposition \ref{PropDS}. 

To this aim, we set 
$w_{\tilde{t}}:=u_{\infty, \tilde{t}}-u_{\infty}$ on $\Sigma^{g_{\infty}}_{\tilde{t}}$ and we claim that 
			\begin{equation}\label{compCNX}
				\begin{array}{lll}
					w_{\tilde{t}} \equiv 0 & \text{in the connected component of $\Sigma^{g_{\infty}}_{\tilde{t}}$ containing the point  $(0',x^{\infty}_N)$.}
				\end{array}
			\end{equation}			
Indeed, denote by $\mathcal{O}$ the connected component of $\Sigma^{g_{\infty}}_{\tilde{t}}$ containing the point  $(0',x^{\infty}_N)$ and let $B=B((0',x^{\infty}_N), R)$ be any open ball such that $\overline{B} \subset \mathcal{O}.$ Then
\begin{equation*}
				\left\{
				\begin{array}{cll}
					-\Delta w_{\tilde{t}} = f(u_{\infty, \tilde{t}})  - f(u_{\infty}) \geq- L_f w_{\tilde{t}} & \text{in} & B,\\
					w_{\tilde{t}}\geq 0 & \text{in} & B,\\
					w_{\tilde{t}}(0',x_{\infty})=0,
				\end{array}
				\right.
			\end{equation*}
where $L_f$ is the Lipschitz constant of $f$ on the compact set $[0, \max_{\overline{B}} \, u_{\infty, \tilde{t}}]$ (here we have also used \eqref{disug-infinito-1}). \\
Since $(0',x^{\infty}_N)\in B$ and \eqref{disug-infinito-3} holds, the strong maximum principle ensures that $ w_{\tilde{t}} \equiv 0$ in $B$ and thus, a standard connectedness argument implies  \eqref{compCNX}. 

Since $\Omega^{\infty}$ is an epigraph, the continuity of  $u_{\infty}$ on $ \overline{\Omega^{\infty}}$ and \eqref{compCNX} imply that $u_{\infty}$ must vanish at some point $ \overline{x} \in \Omega^{\infty}$. This contradicts $u_{\infty}>0$ on $\Omega^{\infty}$ and so, the proof of  Proposition \ref{PropDS} is complete. 
		
\medskip

To conclude the proof of step 2, let us pick $\delta>0$ such that $ 3\delta < \min(\frac{\tilde{t}}{2},\varepsilon(L_f,\gamma))$. By Proposition \ref{PropDS}, there exists $\varepsilon(\delta) \in (0,\delta)$ such that, for any $\varepsilon \in (0,\varepsilon(\delta))$ we have 
$u\leq u_{\tilde{t}+\varepsilon}$ in $\overline{\Sigma_{\delta,\tilde{t}-\delta}^{g}}$.\\
		On $\Sigma^{g}_{\tilde{t}+\varepsilon} \backslash \overline{\Sigma_{\delta,\tilde{t}-\delta}^{g}}$, we have $u,u_{\tilde{t}+\varepsilon} \in H^{1}_{\text{loc}}(\overline{\Sigma^{g}_{\tilde{t}+\varepsilon} \backslash \overline{\Sigma_{\delta,\tilde{t}-\delta}^{g}}})\cap UC(\overline{\Sigma^{g}_{\tilde{t}+\varepsilon} \backslash \overline{\Sigma_{\delta,\tilde{t}-\delta}^{g}}})$ and
				
		\begin{equation*}
			\left\{
			\begin{array}{cll}
				-\Delta u-f(u)=0=-\Delta u_{\tilde{t}+\varepsilon}-f(u_{\tilde{t}+\varepsilon}) & \text{on} & 
				\Sigma^{g}_{\tilde{t}+\varepsilon} \backslash \overline{\Sigma_{\delta,\tilde{t}-\delta}^{g}} \\
				u \leq u_{\tilde{t}+\varepsilon} & \text{on} & \partial \Big( \Sigma^{g}_{\tilde{t}+\varepsilon} 
				\backslash \overline{\Sigma_{\delta,\tilde{t}-\delta}^{g}} \Big)
			\end{array}
			\right.
		\end{equation*}
For the latter, notice that 
$$\partial \Big(\Sigma^{g}_{\tilde{t}+\varepsilon} \backslash \overline{\Sigma_{\delta,\tilde{t}-\delta}^{g}}\Big)=(\{x_N=\tilde{t}+\varepsilon\}\cap \overline\Omega) \bigcup (\partial \Omega \cap \{x_N <\tilde{t}+\varepsilon\}) \bigcup \partial(\overline{\Sigma_{\delta,\tilde{t}-\delta}^{g}}),
$$ 
and so
\begin{itemize}
		\item if $x \in \{x_N=\tilde{t}+\varepsilon\} \cap \overline\Omega$ then $u(x)=u_{\tilde{t}+\varepsilon}(x)$,
		\item if $x \in \partial \Omega \cap \{x_N <\tilde{t}+\varepsilon\}$ then $0=u(x)<u_{\tilde{t}+\varepsilon}(x)$,
		\item if $\partial (\overline{\Sigma_{\delta,\tilde{t}-\delta}^{g}}),$ then Proposition \ref{PropDS} implies that $u(x) \leq u_{\tilde{t}+\varepsilon}(x)$.
		\end{itemize}

Notice that $\mathtt{S}_{e_{N}}(\Sigma^{g}_{\tilde{t}+\varepsilon} \backslash 
\overline{\Sigma_{\delta,\tilde{t}-\delta}^{g}}) \leq 2 \delta+\varepsilon < 3 \delta <\varepsilon(L_f,\gamma)$ and that 
$\Sigma^{g}_{\tilde{t}+\varepsilon} \backslash \overline{\Sigma_{\delta,\tilde{t}-\delta}^{g}} \subseteq 
\{0<x_N< \tilde{t}+\varepsilon\}$ (see Figure \ref{fig_etape_2}). Hence, since $u$ is uniformly continuous, we can apply Theorem \ref{thComp2}, as in Step 1, to get
$$
u \leq u_{\tilde{t}+\varepsilon} \quad \text{in} \quad \Sigma^{g}_{\tilde{t}+\varepsilon} \backslash \overline{\Sigma_{\delta,\tilde{t}-\delta}^{g}}. 
$$

\begin{figure}[h]
	\centering
	\includegraphics[width=0.6\textwidth]{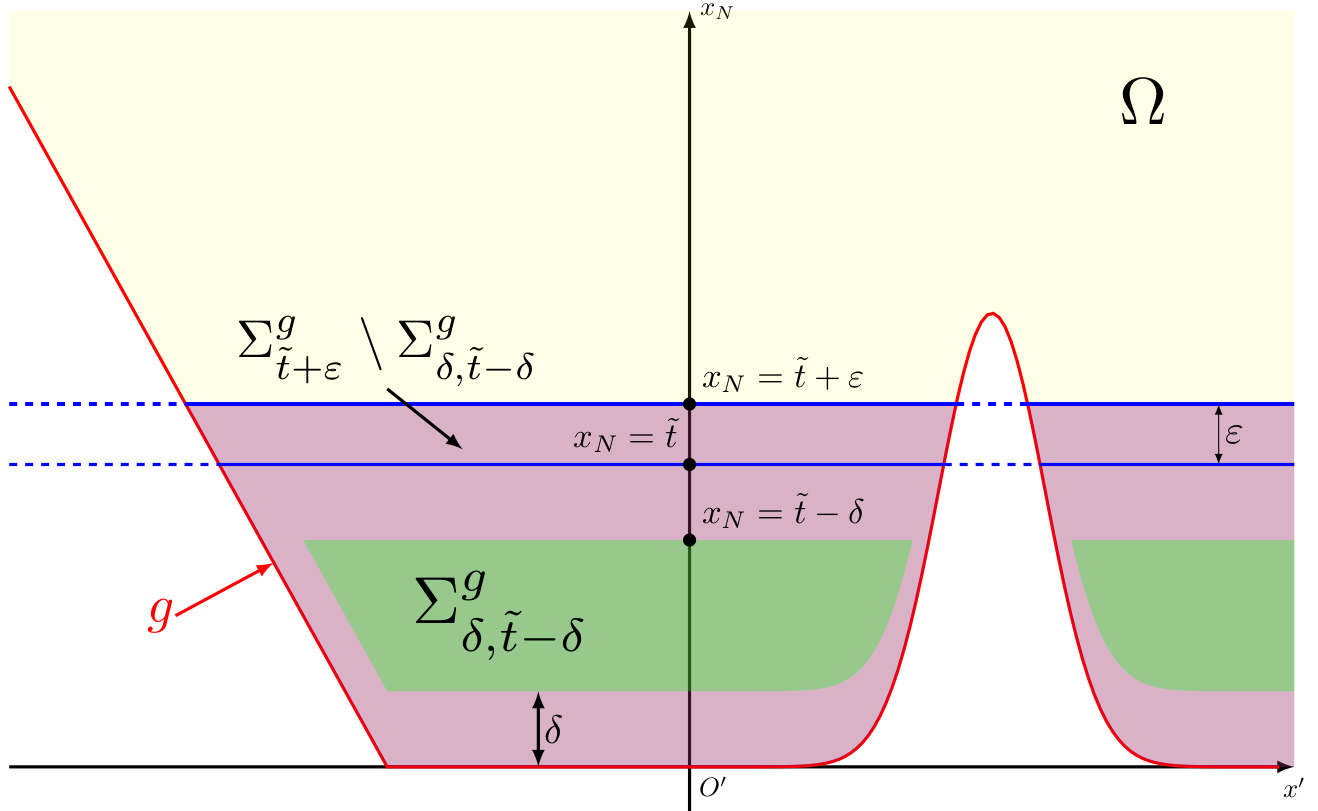}
	\caption{Proof of Step 2}
	\label{fig_etape_2}
\end{figure}	

The latter and Proposition \ref{PropDS} imply that $u \leq u_{\tilde{t}+\varepsilon}$ in 
$\Sigma^{g}_{\tilde{t}+\varepsilon}$. This contradicts the definition of $\tilde{t}$, thus $\tilde{t}=+\infty$.

\medskip

\textit{Step 3 : End of the proof.} 

For each $t>0$, let $(x',t) \in \Omega \cap \{x_{N}=t\}$ and pick  $r>0$ such that $\overline{B((x',t),r)} \subset \Omega$. 
From the previous step we infer that
\begin{equation*}
	\left\{
	\begin{array}{cll}
		-\Delta(u_{t}-u)+L_{f,t}(u_{t}-u) \geq 0 & \text{in} & \Sigma^{g}_{t}\cap B((x',t),r),\\
		u_{t}-u \geq 0 & \text{in} & \Sigma^{g}_{t}\cap B((x',t),r),\\
		u_{t}-u=0 & \text{on} & \{x_N=t\}\cap B((x',t),r), 
	\end{array}
	\right.
\end{equation*} 
where $L_{f,t}$ is the Lipschitz constant of $f$ in the compact set $ \left[ 0,\displaystyle\max_{\overline{\Sigma^{g}_{t}\cap B((x',t),r)}}u_{t} \right]$. 

So, by the Hopf lemma, we have 
\begin{equation}
	\begin{array}{ccc}
	-2 \frac{\partial u}{\partial x_N}(x',t)= \frac{\partial (u_{t}-u)}{\partial x_N}(x',t)<0.
	\end{array}\label{monotonie}
\end{equation}
The latter proves the desired conclusion. \qed

\bigskip

\textit{Proof of Theorem \ref{TH2}.}   
We follow the proof of Theorem \ref{TH1}.  However, this strategy requires significant changes in Step 2 
(recall that Step 2 in the proof of Theorem \ref{TH1} was crucially based on the assumption \eqref{cond-in-zero}, as well as on the uniform continuity of $u$ on $ \Omega\cap \{x_{N}<R\}$ for any $R>0$).

\medskip

\textit{Step 1 : $\Lambda$ is not empty.} 

It is enough to observe that Step 1 in the proof of Theorem \ref{TH1} holds true irrespectively of the value 
of $f(0)$ and that we can apply Theorem \ref{thComp2} (see also Remark \ref{rem-striscia-exp}), since $u$ has at most exponential growth on finite strips by assumption.

\medskip		
		
\textit{Step 2 : $\tilde{t}=\sup \Lambda = +\infty$. } 

To prove the claim, it is enough  to show that Proposition \ref{PropDS} is still true under the (more general) assumptions of Theorem \ref{TH2}. Unfortunately, the method used in the proof of Theorem \ref{TH1} no longer works if the assumption \eqref{cond-in-zero} is not in force. Indeed, in this case we cannot exclude that the limit profile $u_{\infty}$ coincides with the function identically equal to $0$, and this prevents us from proceeding as before. Also, the fact that $u$ is no longer uniformly continuous on finite strips adds new difficulties. 
To circumvent those problems we use a different strategy based on translation and scaling arguments. 

\smallskip

{\textit {Proof of Proposition \ref{PropDS} under the assumptions of Theorem \ref{TH2}.}} 
We proceed as in the proof of Proposition \ref{PropDS} (in Theorem \ref{TH1}) until formula \eqref{controllo-AAinfinito} and, for any 
$k \geq 1,$ we consider again $\Omega^{k}=\{(x',x_N) \in \R^N, \, x_N>g_{k}(x') \}$ and 
$\Omega^{\infty}=\lbrace (x',x_N)\in \R^N, \, x_N >g_{\infty}(x') \rbrace$.

Since $g_{\infty}(0')<\tilde{t} - 2 \delta,$ by the continuity of $g_{\infty}$ and $\eqref{controllo-AAinfinito},$ we can choose $R>0$ small enough such that
	 \begin{equation}\label{eq22sup}
	 	\disp\sup_{x' \in \overline{B'(0',4R)}}g_{\infty}(x')< \tilde{t} - \delta,
	 \end{equation}
where $B'(0',4R) \subset \R^{N-1}.$ 
Since $g_{k} \to g_{\infty}$ in $C^{0}_{\text{loc}}(\R^{N-1}) $, there exists $k_{1}=k_{1}(R) \geq 1$ such that 
\begin{equation}\label{eq22bis}
\forall k \geq k_{1}, \quad \forall x' \in \overline{B'(0',4R)}, \qquad g_{\infty}(x')-\dfrac{\delta}{4} \leq g_{k}(x') \leq  g_{\infty}(x')+\dfrac{\delta}{4}.
\end{equation}
Moreover, in view of  \eqref{controllo-AA},\eqref{controllo-AAinfinito} and \eqref{eq22bis} there exists $k_{2}=k_{2}(R) \geq k_1$ such that for any $k \geq k_{2}$
 	\begin{equation}\label{eq222}
 		(0',x^{k}_{N}) \in \left\lbrace x=(x',x_{N})\in \R^{N} \,  : \,  x' \in \overline{B'(0',2R)}, \quad g_{\infty}(x') + \frac{\delta}{2} \leq x_{N} \right\rbrace .
 	\end{equation}
 	 
Pick $T>2\tilde{t}+2R$ and define the compact set 
 	\begin{equation*}
 		\mathcal{C}= \left( \overline{B'(0',2R)}\times \R \right) \cap \overline{\Sigma^{g_{\infty}+\frac{\delta}{2}}_{2T}},
 	\end{equation*}
and the bounded open set 
 	\begin{equation*}
 		\mathcal{O}= \left( B'(0',4R) \times \R \right) \cap  \left\lbrace x=(x',x_{N})\in \R^N \, : \, g_{\infty}(x')  + \frac{\delta}{4} < x_N < 2T +1 \right\rbrace,
 	\end{equation*}
then, by  \eqref{eq22bis} and \eqref{eq222}, we have (see Figure \ref{fig3}) 
 	\begin{equation}\label{eq22tris}
\forall k \geq k_2, \qquad (0',x^{k}_{N})\in \mathcal{C} \subset \subset \mathcal{O} \subset \Omega^{k}.
 	\end{equation}
 
\begin{figure}[h]
	\centering
	\includegraphics[width=0.6\textwidth]{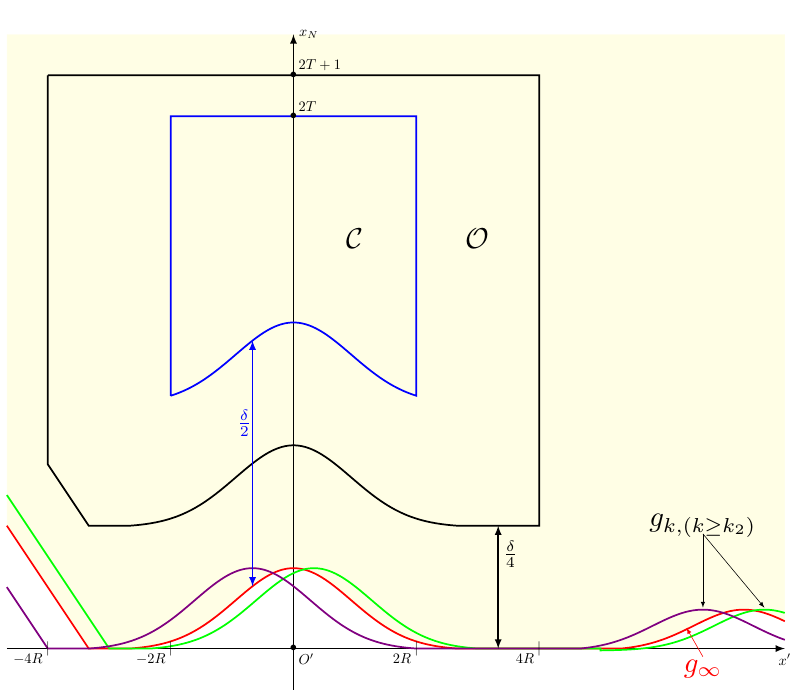}
	\caption{$\mathcal{C} \subset \subset \mathcal{O} \subset \Omega^{k}$}
	\label{fig3}
\end{figure}

Now, for any $ k \geq k_2$ and any $ x \in \Omega^{k}$ we define
	 \begin{equation}\label{eq19}
	 v_{k}(x)=\dfrac{u((x^{k})'+x',x_N)}{u(x^{k})}=\dfrac{u((x^{k})'+x',x_N)}{\alpha_{k}},
	 \end{equation}
	where $\alpha_{k}=u(x^{k})>0$, and 
\begin{equation}\label{def-wk-eq19bis}
	 w_{k}(x)= v_k(x) + {\frac{f(0)}{2 \alpha_k}} \left(x_N-x_N^k\right)^2.
	 \end{equation}

	Then, for any $ x \in \mathcal{O} $ we have 
\begin{align}\label{eq21}
		-\Delta w_{k}(x) & = -\Delta v_{k}(x) - {\frac{f(0)}{\alpha_k}}  = \dfrac{f(u((x^{k})'+x',x_N)) - f(0)}{\alpha_{k}} \\
		& =\dfrac{f(u((x^{k})'+x',x_N)) - f(0)}{\alpha_{k} w_k(x)} w_k(x) :=c_{k}(x) w_{k}(x),
	\end{align}
	where $c_{k}(x) = \dfrac{f(u((x^{k})'+x',x_N)) -f(0)}{\alpha_{k} w_k(x)}$ is a bounded continuous function satisfying 
	\begin{equation}\label{eq21bis}
	\Vert c_k \Vert_{L^\infty( \mathcal{O})} \leq L_f,\qquad \forall \, k \geq k_2.
	\end{equation}	
Indeed, for any $ k \geq k_2$ and any $ x \in \mathcal{O}$, 
     \begin{equation}\label{eq21tris}
	\vert c_k(x) \vert  \leq \dfrac{\vert f(u((x^{k})'+x',x_N)) - f(0) \vert }{\alpha_{k} w_k(x)} \leq L_f  \dfrac{u((x^{k})'+x',x_N)}{\alpha_{k} w_k(x)} = L_f  \dfrac{v_k(x)}{w_k(x)} \leq L_f, 
	\end{equation}		
since $f$ is globally Lipschitz-continuous and $v_k \leq w_k$ in $ \Omega^k$, in view of $f(0) \geq 0$. 

We also note that 
	\begin{equation}\label{eq22}
	w_{k}(0',x^{k}_{N}) = 	v_{k}(0',x^{k}_{N})=\dfrac{u(x^{k})}{u(x^{k})}=1, \qquad \forall \, k \geq k_2.
	\end{equation}

Therefore, for any $ k \geq k_2$, the function $w_k$ satisfies 
\begin{equation}\label{eq21Harnack}
		  \left\{
			\begin{array}{cll}
				-\Delta w_{k} = c_{k} w_{k} & \text{in} & \mathcal{O},\\
				w_k > 0 & \text{in} & \mathcal{O},\\
				\vert c_k \vert \leq L_f & \text{in} &\mathcal{O},
			\end{array}
			\right.
\end{equation}  
and so we can apply Harnack inequality on the compact set $\mathcal{C}$ to obtain 
 	\begin{equation}\label{eq21Harnackbis}
\forall k \geq k_2, \qquad  \sup_{\mathcal{C}}w_{k} \leq C_{7}  \inf_{\mathcal{C}} w_{k},
 	\end{equation}
for some  constant $C_{7}=C_{7}(N, L_f, \mathcal{C},\mathcal{O})>0$.

By combining \eqref{eq21Harnackbis}, \eqref{eq22} and \eqref{eq22tris} we get that 
 	\begin{equation}\label{eq23}
\forall k \geq k_2, \qquad \sup_{\mathcal{C}}v_{k} \leq \sup_{\mathcal{C}}w_{k}\leq C_{7} \inf_{\mathcal{C}}w_{k}\leq C_{7} w_{k}(0',x^{k}_{N}) \leq C_7.
 	\end{equation}
From the latter we also deduce that 
     \begin{equation}\label{eq23bis}
\forall k \geq k_2, \qquad  0 \leq {\frac{f(0)}{\alpha_k}} \leq \frac{2C_7}{\delta^2}, 
 	\end{equation} 
indeed, since the point $(0',x^{k}_{N} + \delta) \in \mathcal{C}$, from \eqref{eq23} we have 
    \begin{equation}\label{eq23tris}
 0 \leq {\frac{f(0) \delta^2}{2 \alpha_k}}  \leq w_k((0',x^{k}_{N} + \delta)) \leq C_7.
 	\end{equation}

Furthermore, for $k \geq 1$,  if $x \in \Sigma^{g_{k}}_{\tilde{t}} \cap (\overline{B'(0',2R)} \times \R)$ then 
$(x'+(x^{k})',x_N) \in \Sigma^{g}_{\tilde{t}}$, therefore by definition of $\Lambda$ we get 
    \begin{equation*}
    	\dfrac{\dr u}{\dr x_N}(x'  +  (x^{k})',x_N) > 0, 
    \end{equation*}	
and so, from  \eqref{eq19} we deduce that 
    \begin{equation}\label{eq24}
 \forall k \geq 1, \quad \forall x \in \Sigma^{g_{k}}_{\tilde{t}} \cap (\overline{B'(0',2R)} \times \R), \qquad  \dfrac{\dr v_{k}}
 {\dr x_N}(x)=\dfrac{1}{\alpha_{k}}\dfrac{\dr u}{\dr x_N}(x'  +  (x^{k})',x_N) > 0. 
    \end{equation}	
The latter combined with \eqref{eq23} immediately leads to the following uniform bound 
	\begin{equation}\label{eq24unifbound}
		\begin{array}{ccc}
\forall k \geq k_2, \quad \forall x \in  \overline{\Sigma_{2T}^{g_{k}} \cap (B'(0',2R) \times \R)}, \qquad 
v_{k}(x) \leq C_{7}. \qquad \qquad \qquad \qquad 
		\end{array}
	\end{equation}
	
To proceed further we observe that $ \Sigma_{2T}^{g_{k}} \cap (B'(0',2R) \times \R) = \mathfrak{C}^{g_{k}}(0',2R,2T)$ and that, for any $ k \geq k_2$,
\begin{equation}\label{eq25}
		  \left\{
			\begin{array}{cll}
			-\Delta v_{k} = \dfrac{f(\alpha_{k} v_k)}{\alpha_{k}}=f_{k}(v_{k}) & \text{in} &\mathfrak{C}^{g_{k}}(0',2R,2T),\\
			v_k=0 & \text{on} & \partial \Omega^k \cap (B'(0',R/2)\times\R), 
			\end{array}
			\right.
		\end{equation}  	
where,  for any $ t \geq 0$, we have set $f_{k}(t)=\dfrac{f(\alpha_{k} t)}{\alpha_{k}}$. Notice that $f_{k} \in {Lip}([0,+\infty))$ with 
\begin{equation}\label{eq25bis}
L_{f_{k}} \leq L_{f}, \qquad 0 \leq f_k(0) = {\frac{f(0)}{\alpha_k}} \leq \frac{2C_7}{\delta^2}, \qquad \forall k \geq k_2.
\end{equation}

By definition of $g_k$, all the epigraphs $\Omega^k$  satisfy a uniform exterior cone condition on $\partial \Omega^k$ with the same reference cone $V$ (the reference cone for the epigraph $ \Omega$), therefore we can apply Proposition \ref{prop2.2} with $ h=g_k, \omega=\Omega^{k}, u = v_{k}, \widetilde{R}=2R, \widetilde{\tau}=2T$, $r=\frac{\widetilde{R}}{4}=\frac{R}{2}, t=\frac{\widetilde{\tau}}{2}=T$ and $k \geq k_2,$ to get 
\begin{equation}\label{estimée_holder_vk-1}
\exists \, \alpha=\alpha(N,V) \in (0,1) \quad : \quad \forall k \geq k_2 \qquad v_{k} \in C^{0,\alpha} (\overline{\mathfrak{C}^{g_{k}} (0',R/2,T)})
\end{equation}
and
\begin{align}\label{estimée_holder_vk-2}
	\|v_{k}\|_{C^{0,\alpha}(\overline{\mathfrak{C}^{g_{k}}(0',R/2,T)})} & \leq C_5(R,T,V,N)  
	\left( L_{f_k} + \vert f_k(0) \vert +1  \right) \left( \|v_k\|_{L^\infty (\mathfrak{C}^{g_k}(0',2R, 2T))} +1 \right) \\
	& \leq C_5(R,T,V,N) \left( L_f + \frac{2C_7}{\delta^2} +1  \right)  \left( C_7+1 \right) = C_8,
\end{align}
for some  constant $C_8=C_8(R,T,V,N, L_f,\delta) >0.$

Now we consider the compact set $ {\mathcal{K}} = \overline{B'(0',R/2) \times (-T,T) }$ 
(see Figure \ref{fig4}) and, for any $ k \geq k_2$, we set 
\begin{equation*}
	\widetilde{v_{k}}(x)=\left\{
	\begin{array}{crl}
		v_{k}(x) & \text{if}& x \in \overline{\mathfrak{C}^{g_{k}}(0',R/2,T)},\\
		0 & \text{if} & x \in {\mathcal{K}}  \setminus \overline{\mathfrak{C}^{g_{k}}(0',R/2,T)},
	\end{array}
	\right.
\end{equation*}
then, by \eqref{estimée_holder_vk-1} and \eqref{estimée_holder_vk-2}, $\widetilde{v_{k}} \in 
C^{0,\alpha}({\mathcal{K}})$ and
\begin{equation}\label{borne_de_vk_sur_tout_le_cylindre} 
\|\widetilde{v_{k}} \|_{C^{0,\alpha}({\mathcal{K}})} \leq C_8,
\end{equation}
and so, by Ascoli-Arzelà theorem, 
\begin{equation}\label{limite_de_vktilde}
\exists \, v_{\infty} \in C^{0,\alpha}({\mathcal{K}}) \quad : \quad 
\widetilde{v_{k}} \to v_{\infty} \quad \text{in} \quad  C^{0}({\mathcal{K}}),
\end{equation}
along a subsequence.

\begin{figure}[!h]
	\centering
	\includegraphics[width=0.6\textwidth]{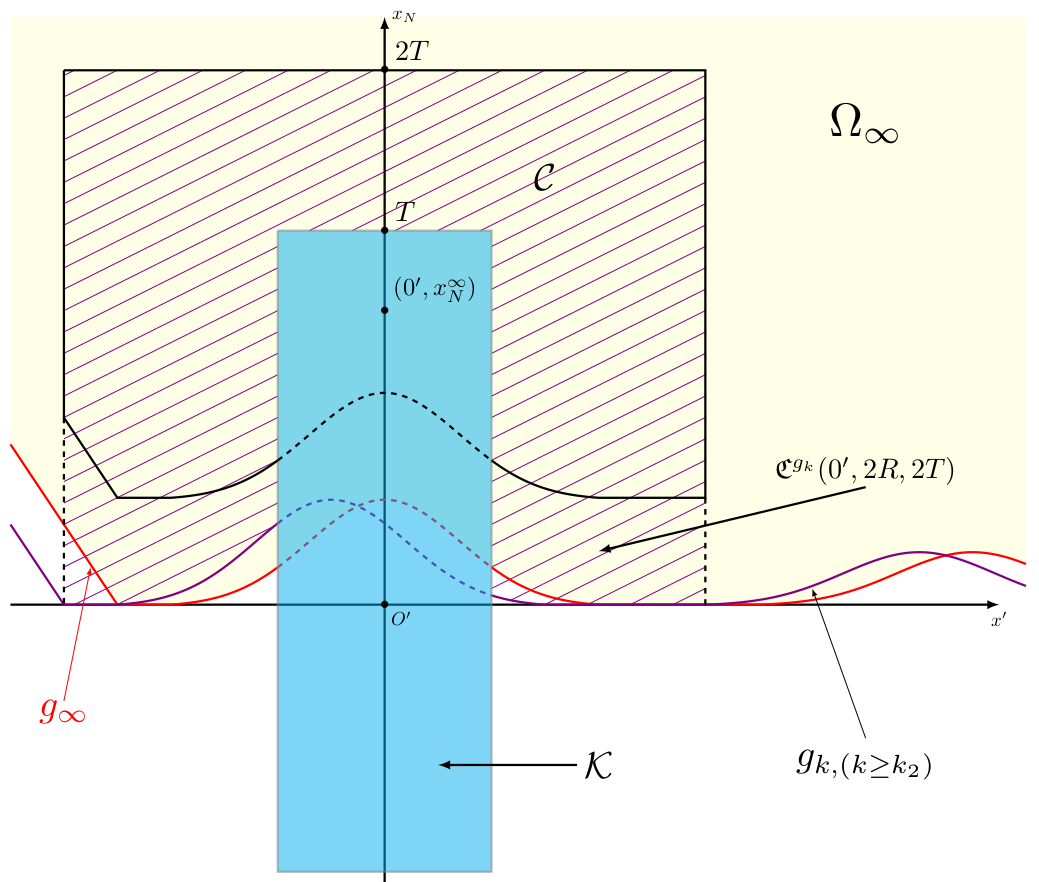}
	\caption{$ {\mathcal{K}} = \overline{B'(0',R/2) \times (-T,T) } $}
	\label{fig4}
\end{figure}

In view of \eqref{eq25}, for any $k \geq k_2$, we have that 
\begin{equation}\label{eqvktilde}
		  \left\{
			\begin{array}{cll}
			-\Delta \widetilde{v_{k}} = f_{k}(\widetilde{v_{k}} ) & \text{in} &\mathfrak{C}^{g_{k}}(0',R/2,T) \subset {\mathcal{K}},\\
			\widetilde{v_{k}} > 0 & \text{in} &\mathfrak{C}^{g_{k}}(0',R/2,T),\\
			\widetilde{v_{k}} =0 & \text{on} & \partial \Omega^k \cap (B'(0',R/2)\times\R), \\
			\widetilde{v_{k}}(0',x^{k}_{N})= v_{k}(0',x^{k}_{N})=1.
			\end{array}
			\right.
\end{equation} 
Since $f_{k} \in {Lip}([0,+\infty))$ and  \eqref{eq25bis} is in force, another application of Ascoli-Arzelà theorem tell us that, up to a subsequence, 
\begin{equation}\label{limite_de_fk}
	f_{k} \to f_{\infty} \quad \text{in} \quad  C^{0}_{\text{loc}}([0,+\infty)), 
\end{equation}
for some $f_{\infty} \in {Lip}([0,+\infty))$ satisfying $f_{\infty}(0) \geq 0$.

Now we recall that, $x^{k}_{N} \longrightarrow x^{\infty}_N  \in \left[g_{\infty}(0')+ \delta, \tilde{t}-\delta \right],$
$g_k \longrightarrow g_{\infty}$ in $ C^{0}_{\text{loc}}(\R^{N-1})$ and that, \eqref{limite_de_vktilde} and \eqref{limite_de_fk} are in force, therefore we can pass  
to the limit, as $k \longrightarrow + \infty$, into \eqref{eqvktilde} to get 
\begin{equation}\label{equation_v_infini}
	\left\{
	\begin{array}{cll}
		- \Delta v_{\infty}=f_{\infty}(v_{\infty}) & \text{in}& {\mathcal{D}}'(\mathfrak{C}^{g_{\infty}}(0',R/2,T)),\\
		v_{\infty} \geq 0 & \text{in} &\mathfrak{C}^{g_{\infty}}(0',R/2,T),\\
		v_{\infty} = 0  & \text{on} & \partial \Omega^{\infty} \cap (B'(0',R/2)\times\R), \\
		(0',x^{\infty}_N) \in \mathfrak{C}^{g_{\infty}}(0',R/2,T), & & v_{\infty}(0',x^{\infty}_N)=1.
	\end{array}
	\right.
\end{equation}
By standard interior regularity theory for elliptic equations we see that $v_{\infty} $ is a classical solution to \eqref{equation_v_infini} and the strong maximum principle then implies 
\begin{equation}\label{equation_v_infini-positive}
v_{\infty} >0 \qquad \text{in}  \qquad \mathfrak{C}^{g_{\infty}}(0',R/2,T).
\end{equation}
By letting $k \longrightarrow + \infty$ into \eqref{eq24}, and recalling the definition of $\widetilde{v_{k}}$, we deduce that 
\begin{equation}\label{inégalitéelimite}
	v_{\infty}(x) \leq v_{\infty,\tilde{t}}(x) \qquad \forall x \in \overline{\Sigma^{g_{\infty}}_{\tilde{t}} \cap (B'(0',R/2) \times \R)}.
\end{equation}

We notice that \eqref{contrad-enunciatoDS} implies $u_{k}(0',x^{k}_{N}) > u_{k,\tilde{t}+\varepsilon_{k}}(0',x^{k}_{N}),$ for any $ k \geq 1$. Then, by definition of $v_k$ and $\widetilde{v_{k}}$, we deduce that $\widetilde{v_{k}} (0',x^{k}_{N}) 
> {\widetilde{v_k,}_{\tilde{t}+\varepsilon_{k}}}(0',x^{k}_N)$, for any $ k \geq 1$.  Passing to the limit in the latter we deduce 
\begin{equation*}
v_{\infty}(0',x^{\infty}_N)\geq v_{\infty,\tilde{t}}(0',x^{\infty}_N), 
 \end{equation*}
and so 
\begin{equation*}
 	v_{\infty}(0',x^{\infty}_N)=v_{\infty,\tilde{t}}(0',x^{\infty}_N)
 \end{equation*}
since $(0',x^{\infty}_N) \in \Sigma^{g_{\infty}}_{\tilde{t}} \cap (B'(0',R/2) \times \R)$ and 
\eqref{inégalitéelimite} holds true.

Now, for any $x \in \Sigma^{g_{\infty}}_{\tilde{t}} \cap (B'(0',R/2) \times \R)$ we set
\begin{equation*}
	w_{\infty}(x)=v_{\infty,\tilde{t}}(x)-v_{\infty}(x),
\end{equation*}
then $w_{\infty}$ satisfy 
\begin{equation}\label{inégalitéelimite-fine}
	\left\{
	\begin{array}{cll}
		- \Delta w_{\infty}=f_{\infty}(v_{\infty,\tilde{t}})-f_{\infty}(v_{\infty}) \geq -L_{f}w_{\infty} & \text{in}& 
		\Sigma^{g_{\infty}}_{\tilde{t}} \cap (B'(0',R/2) \times \R), \\
		w_{\infty} \geq 0  & \text{in} & \Sigma^{g_{\infty}}_{\tilde{t}} \cap (B'(0',R/2) \times \R),\\
		(0',x^{\infty}_N) \in \Sigma^{g_{\infty}}_{\tilde{t}} \cap (B'(0',R/2) \times \R) , & & w_{\infty}(0',x^{\infty}_N)=0.
	\end{array}
	\right.
\end{equation}
By applying the strong maximum principle to \eqref{inégalitéelimite-fine}, we infer that $w_{\infty} \equiv 0$ on the domain 
$\Sigma^{g_{\infty}}_{\tilde{t}} \cap (B'(0',R/2) \times \R)$. Therefore, we have
\begin{equation*}
	0=v_{\infty}(x',g_{\infty}(x'))=v_{\infty}(x',2\tilde{t}-g_{\infty}(x')),\quad \forall \, x' \in B'(0',R/2),
\end{equation*}
which contradicts  \eqref{equation_v_infini-positive}, since $0 < 2\tilde{t}-g_{\infty}(x') < T$ whenever $x' \in B'(0',R/2)$.
This completes the proof of  Proposition \ref{PropDS} when $f(0)=0$. 

The remaining part of the proof of Step 2  is the same of the one of Step 2 in Theorem \ref{TH1}, just observe that one can use Theorem \ref{thComp2} (see also Remark \ref{rem-striscia-exp}), since $u$ has at most exponential growth on finite strips by assumption.

Step 3 is unchanged. \qed

\bigskip

\textit{Proof of Theorem \ref{TH2bis}.}   
The proof of Theorem \ref{TH2bis} is a straightforward  adaptation of the one of Theorem \ref{TH2}. Indeed, any step of the proof of Theorem \ref{TH2} remains unchanged if one observes that 

i) one can apply Theorem \ref{thComp1}, since by assumption $u$ is bounded on any finite strip and $f$ is a locally-Lipschitz continuous function, 

ii) Step 2 only requires  that $u$ is bounded on any finite strip and that $f$ belongs to $ {Lip}_{loc}([0,+\infty))$ and satisfies $f(0)\geq0$.  Indeed, those properties imply the validity of \eqref{eq21bis} and \eqref{eq21tris}, where $L_f$ is replaced by the Lipschitz constant of $f$ on the interval $\left[0, M \right]$. Here, we have set  
$M := \sup_{\Omega \cap \left \lbrace x_N \, < \, 2T+1 \right\rbrace } u $, which is finite by assumption. 
		
Step 3 is unchanged, since it only requires that $f$ belongs to $ {Lip}_{loc}([0,+\infty)).$  \qed			
	
\medskip

\textit{Proof of Corollary \ref{Cor1}.}   

Since $ \nabla u$ is bounded on $ \Omega$, $u \in C^0(\overline{\Omega})$ and $ u=0$ on $ \partial \Omega$, it is straightforward to check that $u$ satisfies $ \vert u(x) - u(y) \vert \leq  \Vert \nabla u  \Vert_{\infty} \vert x - y \vert $ for any $ x,y \in \Omega$. Then,  $u$ is uniformly continuous on $\Omega$ and $u \in W^{1,\infty}_{loc}(\overline{\Omega}) \subseteq H^{1}_{loc}(\overline{\Omega}).$ The  conclusion then follows by applying Theorem  \ref{TH1}. \qed	

\medskip

\textit{Proof of Corollary \ref{Cor2}.}   

$u \in W^{1,\infty}_{loc}(\overline{\Omega}) \subseteq H^{1}_{loc}(\overline{\Omega}),$ since $u \in C^0(\overline{\Omega}),$ $u=0$ on $ \partial \Omega$ and $ \nabla u$ is bounded in $ \Omega$ by assumption. The latter assumption also implies (via the mean value theorem) that $u$ is bounded on  finite strips.
The conclusion then follows by applying Theorem  \ref{TH2bis}. \qed

\medskip

\textit{Proof of Corollary \ref{Cor3}.}  

Since $\partial \Omega$ is locally of class $C^{1, \alpha}$ and $u \in C^2(\Omega) \cap C^0(\overline{\Omega})$, standard regularity results for elliptic equations implies that $u \in C^{1,\alpha}_{loc}(\overline{\Omega}).$ Therefore $ u  \in H^{1}_{loc}(\overline{\Omega})$ and the  conclusion then follows from Theorem  \ref{TH2bis}. \qed

\medskip

\textit{Proof of Theorem \ref{TH2coerc}.}  

The proof is similar to the one of Theorem \ref{TH1}, but it is easier, since the coercivity assumption on $g$ implies that any cap $\Sigma_{\theta}^{g} $ is a bounded open set of $\R^N$. 
Thus, we provide only the modifications necessary to deal with the new points of the proof related to the fact that $g$ is supposed to be merely continuous and that no restriction on the sign of $f(0)$ is imposed. \\ 
Since every cap $\Sigma_{\theta}^{g} $ is bounded, the comparison principle immediately gives the conclusion of Step 1. 

To achieve the conclusion of Step 2 we need to show that Proposition \ref{PropDS} still holds under the  assumptions of Theorem \ref{TH2coerc}. \\ 
To this end we first observe that the sequence of points $ (x^k)$ appearing in \eqref{contrad-enunciatoDS} is bounded (since the cap $\Sigma_{\tilde{t}- \delta}^{g} $ is bounded) and so, up to a subsequence, we may and do suppose that $x^k \to x^{\infty} := ((x^{\infty})^{'}, x^{\infty}_N)$, as $k \to \infty$.  Therefore, the sequence of continuous functions $(g_k)$ defined by \eqref{trasl-g_k} is again relatively compact in $ C^0_{loc}(\R^{N-1})$. Hence, up to a subsequence, $g_{k} \to g_{\infty}$  in $C^{0}_{\text{loc}}(\R^{N-1})$. Note that $g_{\infty}(x') = g(x' + (x^{\infty})^{'})$ for any $x' \in \R^{N-1}$. 

Moreover, the sequence $(\tilde{u}_{k})$ defined by \eqref{trasl-tilde-u_k} is again relatively compact in $C^0_{loc}(\R^N)$.
This is because the sequence $(\tilde{u}_{k}(0',-1) )$ is bounded in $\R$ and the function $\tilde{u}$ is bounded and uniformly continuous on any finite strip. Hence, up to a subsequence, $\tilde{u}_{k} \to \tilde{u}_{\infty}$ in $C^{0}_{\text{loc}}(\R^N)$, where $\tilde{u}_{\infty}(x) = \tilde{u}(x' +(x^{\infty})^{'}, x_N)$ for any $x \in \R^N$. 

Thanks to those information, we can follow the proof until \eqref{disug-infinito-3}. 
Now, if we set $w_{\tilde{t}}:=u_{\infty, \tilde{t}}-u_{\infty}$ on $\Sigma^{g_{\infty}}_{\tilde{t}},$ we see that 
			\begin{equation}\label{prepar-strong-max-princ}
				\left\{
				\begin{array}{cll}
					-\Delta w_{\tilde{t}} = f(u_{\infty, \tilde{t}})  - f(u_{\infty}) \geq- L_f w_{\tilde{t}} & \text{in} & \Sigma^{g_{\infty}}_{\tilde{t}},\\
					w_{\tilde{t}}\geq 0 & \text{in} & \Sigma^{g_{\infty}}_{\tilde{t}},\\
					w_{\tilde{t}}(0',x_{\infty})=0,
				\end{array}
				\right.
			\end{equation}
where now $L_f$ denotes the Lipschitz-constant of $f$ on the closed interval $\left[ 0, \max_{\overline{\Sigma^{g_{\infty}}_{\tilde{t}}}} \, u_{\infty, \tilde{t}} \right]$ (note that $\overline{\Sigma^{g_{\infty}}_{\tilde{t}}}$ is a compact set, since $g_{\infty}$ is coercive). 			 Since $(0',x^{\infty}_N)\in  \Sigma^{g_{\infty}}_{\tilde{t}}$,	the strong maximum principle and the continuity of $w_{\tilde{t}}$ ensure that 
			\begin{equation*}
				\begin{array}{lll}
					w_{\tilde{t}} \equiv 0 & \text{in the connected component of $\Sigma^{g_{\infty}}_{\tilde{t}}$ containing the point  $(0',x^{\infty}_N)$.}
				\end{array}
			\end{equation*}			
Since $\Omega^{\infty}$ is an epigraph, the latter implies that $u_{\infty}$ must vanish at some point of $\Omega^{\infty}$.
But (in view of the form of  $g_{\infty}$ and $\tilde{u}_{\infty}$) the latter implies that the solution $u$ must vanish at some point of $ \Omega$. A contradiction.  This proves the claim of Proposition \ref{PropDS}. 

The remaining part of the proof of Step 2  is the same of the one of Step 2 in Theorem \ref{TH1}, since $u$ is bounded on finite strips. 

Step 3 is unchanged, since it only requires that $f$ belongs to $ {Lip}_{loc}([0,+\infty)).$  \qed	

\section{Extensions to merely continuous epigraphs bounded from below \\ and further observations} \label{SectEXT}

In this section we prove some monotonicity results for solutions of \eqref{NonLin-PoissonEq} on certain merely continuous epigraphs $\Omega$ bounded from below, i.e.,where the function $g$ defining $\partial \Omega$ is continuous but not necessarily uniformly continuous. 
To this end we observe that the uniform continuity of the function $g$ enters into the proofs of Theorems \ref{TH1}-\ref{TH2bis}  only to use the following classical compactness result (via the Ascoli-Arzelà theorem):
\textit{let $(g_k)$ be a sequence of translations of a uniformly continuous $g : \R^{N-1} \to \R$, i.e., $g_k (\cdot) = g(\cdot + x_k),$ for some sequence $(x_k)$ of points of  $\R^{N-1}$.  If the sequence $(g_k)$ is bounded at a fixed point of $\R^{N-1}$, then it admits a subsequence converging uniformly on every compact sets of $\R^{N-1} $ }. \\
Therefore, all we need for the above mentioned proofs to work, is that the continuous function $g :\R^{N-1} \to \R$ is bounded from below and it satisfies the following compactness property :

$(\mathcal{P})$ \quad \textit{Any sequence $(g_k)$ of translations of $g$, which is bounded at some fixed point of $\R^{N-1}$, admits a subsequence converging uniformly on every compact sets of $\R^{N-1}$. } 

The above discussion leads to the following 

\begin{Defi} \label{Def-classG} Assume $N \geq 2.$ We say that a continuous function $g : \R^{N-1} \mapsto \R$ belongs to the class $\mathcal{G}$, if it satisfies the compactness property $(\mathcal{P})$.
\end{Defi} 

Before stating the new monotonicity results, let us show how large the class $\mathcal{G}$ is. Hereafter, we provide a wide (but non-exhaustive) list of members of $\mathcal{G}$.

\begin{enumerate}
\item Uniformly continuous functions on $\R^{N-1}$ belong to $\mathcal{G}$. 

\smallskip

\item Coercive continuous functions on $\R^{N-1}$ belong to $\mathcal{G}.$\footnote{Since the sequence $(g_k):= 
(g(\cdot + x_k))$ is bounded at some point of $\R^{N-1}$ and $g$ is coercive,  it follows that the sequence $(x_k)$ is bounded in $\R^{N-1}$. Therefore $(g_k)$ is uniformly equicontinuous on  compact sets of $\R^{N-1}$ , and the compactness follows from the Ascoli-Arzelà theorem.} 

\smallskip

\noindent In particular, any continuous function on $\R^{N-1}$ such that $ \lim_{\mid x \mid \mapsto\infty} g(x) \in (-\infty , + \infty]$ belongs to $\mathcal{G}$. 

\smallskip

\item Let us denote by $\mathsf{G}(\R^{N-1})$ the set of continuous functions $g :\R^{N-1} \to \R$ enjoying the following 
property : \textit{there exists a \textsf{continuous bijection} $\phi: \R  \to \R$  such that $\phi \circ g \in \mathcal{G}$.
}  \\ 
It is immediate to check that $\mathsf{G}(\R^{N-1}) \subset \mathcal{G}$.\footnote{Assume that $(g_k)$ is bounded at 
some point of $\R^{N-1}$, say $\bar x$, then the sequence of functions $((\phi \circ g)_k) :=((\phi \circ g)(\cdot + x_k))$ 
is bounded at $ \bar x$. Therefore, up to a subsequence, $(\phi \circ g)_k \rightarrow \varphi$ uniformly on compact sets of $\R^{N-1}$, since $\phi \circ g \in \mathcal{G}$ by assumption. But then we also have $g_k = \phi^{-1}((\phi \circ g)_k) \rightarrow \phi^{-1} \circ \varphi$ uniformly on compact sets of $\R^{N-1}$. } Moreover, the family $\mathsf{G}(\R^{N-1})$  strictly contains the one of  \textit{uniformly continuous functions} on $\R^{N-1}$ and the one of  \textit{coercive continuous functions}  on $\R^{N-1}$, as shown by the next examples.\footnote{\, Taking the identity function as $\phi$, we immediately see that uniformly continuous functions, as well as coercive continuous functions $g :\R^{N-1} \to \R$ do belong to $\mathsf{G}(\R^{N-1})$.} 

\smallskip

\begin{itemize}

\item [(3a)] For instance, the functions $g_1(x_1) = e^{x_1}$ if $ N=2,$ and $g(x)=e^{x_1 + \sum_{j=2}^{N-1} \cos^j(x_j)}$ if $ N \geq 3$, are neither uniformly continuous nor coercive on $\R^{N-1}$. Nevertheless, they  belong to the class 
$\mathsf{G}(\R^{N-1})$ (consider, for instance, the function
\begin{equation*}
			\phi(t)=\left\{
			\begin{array}{crl}
      		t & \text{if}&  t \leq 0,\\
        	    \log(t+1) & \text{if} & t > 0,
			\end{array}
				\right.
\end{equation*}
in the previous definition and observe that $\phi \circ g \in \mathcal{G}$ is uniformly continuous, hence  $\phi \circ g \in \mathcal{G}$). The same argument also proves that $g(x_1) = e^{e^{x_1}}$ and $g(x)=e^{e^{x_1 + \sum_{j=2}^{N-1} \cos^j(x_j)}}$ do belong to $\mathsf{G}(\R^{N-1})$. Since this argument can be iterated, we see that the class $\mathsf{G}(\R^{N-1})$ contains smooth functions, bounded from below, that are neither uniformly continuous nor coercive, and with arbitrary large growth at infinity.
Also note that, $g_1$ being convex, its epigraph satisfies a uniform exterior cone condition.

\smallskip

\item [(3b)] Assume $ N \geq 2$ and let $g \in C^2(\R^{N-1})$ be any positive function such that $ \nabla^{2} g\in L^{\infty}(\R^{N-1})$,  
then $\sqrt{g}$ is globally Lipschitz-continuous on $\R^{N-1}$ (see Lemma $I$ in \cite{gg} for $N=2$). Therefore, we have $g \in \mathsf{G}(\R^{N-1})$.\footnote{Choose the function 
\begin{equation*}
			\phi(t)=\left\{
			\begin{array}{crl}
      		- \sqrt{-t} & \text{if}&  t < 0,\\
        	     \sqrt{t} & \text{if} & t \geq 0,
			\end{array}
				\right.
\end{equation*}
}
More generally, any $g \in C^2(\R^{N-1})$, bounded from below and such that $ \nabla^{2} g\in L^{\infty}(\R^{N-1})$ is a member of the class  $\mathsf{G}(\R^{N-1})$. 
\\ 
For instance, the function $ g= g(x_1,\ldots, x_{N-1}) = {(x_1)}^2 + \prod_{j=2}^{N-1} \sin(j x_j)$ belongs to $\mathsf{G}(\R^{N-1})$ for any $N \geq3$. 

\end{itemize}

\smallskip

\item For $N=2$, any continuous function $g :\R \to \R$ such that ${\ell}_{-}:= \lim_{x \mapsto -\infty}g(x)\in (-\infty , + \infty]$ and $ {\ell}_{+}:= \lim_{x \mapsto +\infty} g(x) \in (-\infty , + \infty]$ belongs to $\mathcal{G}$. Therefore, any quasiconvex (resp. quasiconcave) continuous function bounded from below belongs to $\mathcal{G}$. In particular, any monotone  continuous function bounded from below and  any convex functions bounded from below belongs to $\mathcal{G}$.

\smallskip

\item Assume $ 2 \leq n < N$ and let $ g : \R^{n-1} \to \R$  be a member of $\mathcal{G}$. Then, the function $ \tilde{g} : \R^{N-1} \to \R$ defined by $ \tilde{g} (x_1, \ldots, x_{N-1}) = g(x_1,\ldots, x_{n-1})$ satisfies the compactness property $(\mathcal{P})$, as a function on $\R^{N-1}$. Therefore, $ \tilde{g} \in  \mathcal{G}$, as a function on $\R^{N-1}$. \\ 
In particular, for $ N \geq 2$, the functions $ g(x_1,\ldots, x_{N}) = e^{e^{x_1}}$ and $ g(x_1,\ldots, x_{N}) =  e^{x_1} - 4 \arctan(x_1) - 2 $ belong to $\mathcal{G}$, are bounded from below  and their epigraphs satisfy a uniform exterior cone condition.

\smallskip

\item  Assume $N \geq 2$. Let $g \in \mathcal{G}$ and let $T : \R^{N-1} \to \R^{N-1}$ be a transformation of the form $ T(x) = Ax +b$, where $A$ an invertible real matrix and $b \in \R^{N-1}$. Then, $g \circ T \in \mathcal{G}$. In particular, this results applies when $T$ is an isometry of $\R^{N-1}$.

\smallskip

\item  Assume $N \geq 2$ and $ \lambda \geq 0.$ Let $g, \tilde{g} \in \mathcal{G}$ be bounded from below. Then, $\lambda g \in \mathcal{G}$ and $ g+ \tilde{g} \in \mathcal{G}$.\\ 
In particular, for $ N \geq 2$, the function $g(x_1,\ldots, x_{N}) = x_1^4 + e^{x_2}$ belongs to $\mathcal{G}$, is bounded from below  and its epigraph satisfies a uniform exterior cone condition. 

\smallskip

\item By combining items $(1)$-$(7)$, one can easily build further examples of functions $g$  belonging to $\mathcal{G}$ in any dimension $N \geq 2$. In particular, one can construct members of $\mathcal{G}$ bounded from below, that are neither uniformly continuous nor coercive, with arbitrary large growth at infinity, and such that their epigraphs satisfy a uniform exterior cone condition.
 
\end{enumerate} 

In view of the above discussions,  we can now state the monotonicity results for continuous epigraphs defined by functions $g$ belonging to the class $\mathcal{G}$.  Let us start with the extension of Theorem \ref{TH1} and Corollary \ref{Cor1}. 

\begin{thm}\label{TH_fonc_G}
Let  $N \geq 2$ and let $\Omega $ be an epigraph bounded from below and defined by a function $g \in \mathcal{G}$. Assume $f \in {Lip}_{loc}([0,+\infty))$ with 
\begin{equation}\label{cond-in-zero-bis}
          \liminf_{t \to 0^+} \frac{f(t)}{t} > 0.
\end{equation}

(i) If $u \in C^0(\overline{\Omega}) \cap H^1_{loc}(\overline{\Omega})$ is a   distributional solution to \eqref{NonLin-PoissonEq} which is uniformly continuous on finite strips. \\ 
Then $u$ is strictly increasing in the $x_N$-direction, i.e., $\frac{\partial u}{\partial x_N}>0$ in $\Omega$.

(ii) If $u \in C^2(\Omega) \cap C^0(\overline{\Omega})$ is a classical solution of \eqref{NonLin-PoissonEq} such that $\nabla u \in L^{\infty}(\Omega)$. 
Then $u$ is strictly increasing in the $x_N$-direction, i.e., $\frac{\partial u}{\partial x_N}>0$ in $\Omega$.
\end{thm}

If we also assume that  the epigraph satisfies a uniform exterior cone condition, then we can prove the following extensions of Theorem \ref{TH2} and Theorem \ref{TH2bis}.

\smallskip

\begin{thm}\label{TH_fonc_G-bis}
Let  $N \geq 2$ and let $\Omega $ be an epigraph defined by a function $g \in \mathcal{G}$.
Also suppose that $\Omega$ is bounded from below and satisfies a uniform exterior cone condition.

(i) Assume $ f \in {Lip}([0,+\infty))$ with $f(0) \geq 0$ and let $u \in C^0(\overline{\Omega}) \cap H^1_{loc}(\overline{\Omega})$ be a distributional solution to \eqref{NonLin-PoissonEq} with at most exponential growth on finite strips.
Then $u$ is strictly increasing in the $x_N$-direction, i.e., $\frac{\partial u}{\partial x_N}>0$ in $\Omega$.

(ii) Assume $ f \in {Lip}_{loc}([0,+\infty))$ with $f(0)\geq0$ and let $u \in C^0(\overline{\Omega}) \cap H^1_{loc}(\overline{\Omega})$ be a distributional solution to \eqref{NonLin-PoissonEq} which is bounded on finite strips.
Then $u$ is strictly increasing in the $x_N$-direction, i.e., $\frac{\partial u}{\partial x_N}>0$ in $\Omega$.
\end{thm}

\smallskip

Note that Theorem \ref{TH_fonc_G} and Theorem \ref{TH_fonc_G-bis} apply to the explicite examples of functions $g$ provided in the above items $(1)$-$(7)$.  They also recover the results in the very recent preprint \cite{GMS}, where the monotonicity is proved for some classical solutions to \eqref{NonLin-PoissonEq} with $f(t) = t^q$ and $q \geq 1$.\footnote{It is immediate to see that any semicoercive and continuous function is bounded from below and belongs to our class $ \mathcal{G}$ (see items $(5)$ and $(6)$ above). This observation and Lemma A.2 in \cite{GMS} also show that the convex epigraph case is covered by the techniques we have developed in the present article. \smallskip}  

\medskip

The next result applies to solutions of \eqref{NonLin-PoissonEq} where $\Omega$ is merely a continuous epigraph and 
$f$ is a non-increasing function, possibly discontinuous, and with no restriction on the sign of $f(0)$. 

\begin{thm}\label{TH3-disc}
Let $\Omega$ be any continuous epigraph bounded from below and let $f: [0,\infty)  \mapsto \R$ be any non-increasing function.
Let $u \in C^0(\overline{\Omega}) \cap H^1_{loc}(\overline{\Omega})$ be a distributional solution to \eqref{NonLin-PoissonEq} 
with subexponential growth on finite strips 
\footnote{i.e., for any $R>0$, 
$$ \limsup_{\overset {\vert x \vert \to \infty,} {x \in \Omega \cap \left \lbrace x_N <R \right\rbrace}} \frac{\ln u(x)}{\vert x \vert} \leq 0. $$}.  \\   
Then $u$ is non-decreasing, i.e., $\frac{\partial u}{\partial x_N} \geq 0$ in $\Omega$.\footnote{\, Note that $ u \in C^1(\Omega),$ since $f(u) \in L^{\infty}_{loc}(\Omega)$.}  \\ 
Moreover, if $f \in  {Lip}_{loc},$ then $u$ is strictly increasing, i.e., $\frac{\partial u}{\partial x_N} > 0$ in $\Omega$.
\end{thm}

\begin{proof} 
The proof is a straightforward application of Theorem \ref{thComp3}. Indeed, for any $\theta >0$ we have
$u,u_{\theta} \in H^{1}_{\text{loc}}(\overline{\Sigma^{g}_{\theta}})\cap C^{0}(\overline{\Sigma_{\theta}^{g}})$ and 
\begin{equation*}
	\left\{
	\begin{array}{ccl}
		- \Delta u-f(u) =0=-\Delta u_{\theta}-f(u_{\theta}) & \text{in} & \Sigma_{\theta}^{g},\\  
		u\leq  u_{\theta} & \text{on} & \partial \Sigma_{\theta}^{g}.
	\end{array}
	\right.
\end{equation*} 
Since $u$ and $u_{\theta}$ have subexponential growth on $\Sigma_{\theta}^{g}$ (which is bounded in the direction $e_{N}$) we can apply Theorem \ref{thComp3} to get that 
\begin{equation*}
	u\leq u_{\theta} \quad \text{in} \quad \Sigma_{\theta}^{g}.
\end{equation*} 
Therefore, $u$ is non-decreasing in the $x_N$-direction. \\  
If $f$ is locally Lipschitz-continuous, we conclude as in Step $3$ of Theorem \ref{TH2bis}. 
\end{proof} 

\medskip

Some remarks are in order.

\begin{rem}
(i) We note  that the proof of Theorem \ref{TH3-disc} also applies to some unbounded domains that are not necessarily epigraphs. 
Indeed, given any domain $\Omega \subset \R^N$ bounded from below, it is immediate to see that the proof of Theorem \ref{TH3-disc} applies if $\Omega$ contains the reflection $($with respect to the hyperplane $\{x_{N}=\theta\})$ of any cap 
$\Omega \cap \{x_{N}< \theta\}$. Since the latter property is clearly satisfied if $\Omega \subsetneq \R^N$ supports a monotone solution to \eqref{NonLin-PoissonEq}, we see that the above result actually caracterizes the euclidean proper domains for which the homogeneous Dirichlet BVP \eqref{NonLin-PoissonEq} admits a monotone solution.\footnote{ \, this can be formulated in the following equivalent way : Problem \eqref{NonLin-PoissonEq} with $f$ non-increasing, admits a monotone solution on a domain $\Omega\subsetneq \R^N$, if and only if, $\Omega$ is contained in an affine half-space (whose inner normal is denoted by $\nu$) and, for any $x \in \Omega$ the open half-line $\left\lbrace x + t \nu, \, t >0 \right\rbrace$ is contained in $\Omega.$ That is, if and only if, $\Omega$ is contained in an affine half-space (whose inner normal is denoted by $\nu$) and $\Omega$ is $\nu$-invariant.} \\  
For example, the domain $ \Omega_4 = \bigcup_{k \in \Z} \{x \in \R^2\, : \,  |x_1-(3+4k)|<1 \, , \,\, x_2 >-1\} \bigcup 
\{x \in \R^2 \, : \, x_2> 0\}$ 
$($see Figure \ref{fig:omega_1}$)$ and the open orthant $ \Omega_5 = \{x \in \R^{N} \, : \, x_1>0, \ldots, \, x_N>0\}$ satisfy this property, but they are not epigraphs with respect to $e_N$ (in fact, $\Omega_4$ is never an epigraph, that is, there is no unit vector $\nu$ of $\R^2$ that allows it to be represented as an epigraph with respect to $\nu$). \\ 
Another example is depicted in Figure \ref{fig:omega_1bis}.  

\medskip
\begin{multicols}{2}
	
	\begin{figure}[H]
		\centering
		\includegraphics[width=8.3cm]{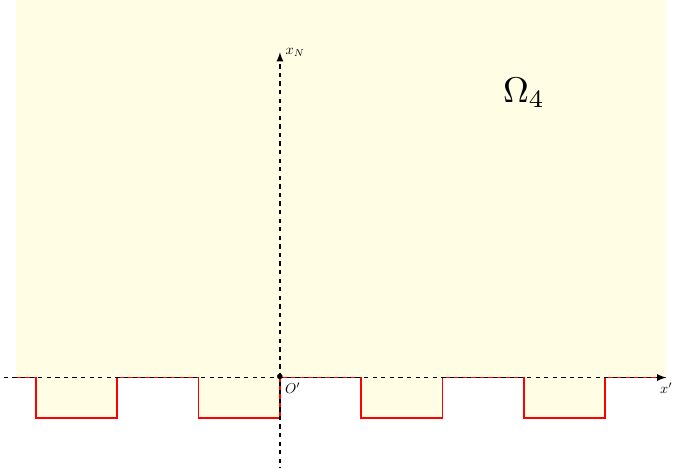}	
		\caption{$\Omega_4$}
		\label{fig:omega_1}
	\end{figure}
	\begin{figure}[H]
		\centering
		\includegraphics[width=8.1cm]{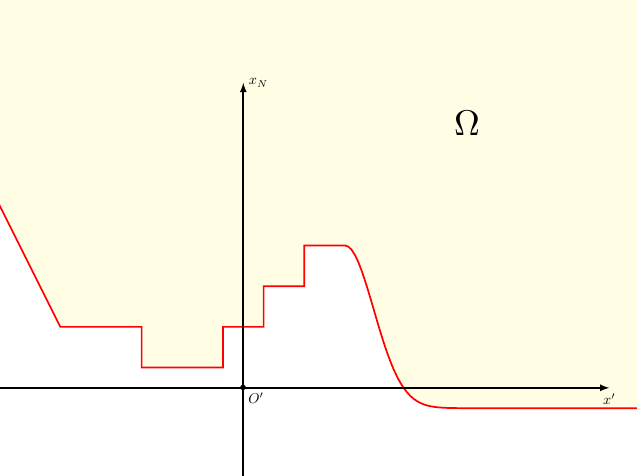}
		\caption{}
		\label{fig:omega_1bis}
	\end{figure}
\end{multicols}

(ii) We also observe that the first conclusion of Theorem \ref{TH3-disc}, namely that $u$ is non-decreasing, is sharp. 
See the explicit example \ref{Ex3} in Section \ref{SectEXAMPLES}. 
\end{rem}

We conclude this section with the following general result, which we believe to be of independent interest. This is a straightforward application of our comparison principles (see Step 1 in the proof of our monotonicity theorems). 
It shows that the moving plane method can be started irrespectively of the value of $f(0)$ and for a large class of unbounded domains. Specifically we have the following

\begin{cor}[\textbf {Starting the moving planes method}]\label{IniziareMP}
Assume $N \ge 2$. Let $\Omega$ be a domain contained in the upper half-space $\R^N_+$ and also assume that $\Omega$ contains the reflection $($with respect to the hyperplane $\{x_{N}=R\})$ of a finite strip 
$\Omega \cap \{x_{N}< R\} \not \equiv \emptyset$. \\
Let $u$ be a distributional $($resp. classical$)$ solution to \eqref{NonLin-PoissonEq}. Suppose that one of the following assumptions is in force : 

(i) $ f \in {Lip}_{loc}([0,+\infty))$ and $u$ is bounded on $ \, \Omega \cap \{x_{N}< R\}$;

(ii) $ f \in {Lip}([0,+\infty))$ $($resp. a non-increasing function$)$ and $u$  has at most exponential growth on 
$\, \Omega \cap \{x_{N}< R\}$.

Then there exists $R_0 \in (0,R)$ such that $ \, {\Omega \cap \left \lbrace x_N <R_0 \right\rbrace  \not \equiv \emptyset}$ and 
$$
\frac{\partial u}{\partial x_N} > 0 \quad  \textit{in} \quad  \Omega \cap \left \lbrace x_N <R_0 \right\rbrace .
$$
\end{cor}

The above result extends the one proved in \cite{Far-Sandro} for the special case of a half-space. 

For sake of clarity (and simplicity) we have stated the above result in the case $\Omega \subseteq \R^N_+$. Since the considered problem is invariant by isometry, it is clear that the same result holds if $\Omega$ is contained in an affine open half-space $H$. In this case, the monotonicity will be obtained with respect to $\nu$, the inner normal to $H$.

\section{Some applications to classification and non-existence results} \label{Sect-Applic}

In this section, we apply our monotonicity results to prove some classification and nonexistence results for the problem \eqref{NonLin-PoissonEq} on general continuous epigraphs defined by a function $g \in \mathcal{G}$ (see Section \ref{SectEXT}).

\begin{thm}\label{th_n_entre_2_et_11}
Let $\Omega $ be an epigraph defined by a function $g \in \mathcal{G}$.
Also suppose that $\Omega$ is bounded from below and satisfies a uniform exterior cone condition. \\
Let $u \in C^0(\overline{\Omega}) \cap H^1_{loc}(\overline{\Omega})$ be a bounded distributional solution to
\begin{equation}\label{equation_sans_neumann}
		\left\{
		\begin{array}{cll}
			- \Delta u=f(u) & \text{in}& \Omega,\\
			u \geq 0 & \text{in} & \Omega,\\
			u=0 & \text{on} & \dr\Omega.
		\end{array}
		\right.
	\end{equation} 
Assume that 	$f \in C^{1}([0,+\infty))$, $f(t)>0$ for $t>0$ and $2\leq N \leq 11$, then $u \equiv 0$ and $f(0)=0$. 
\end{thm}  

\begin{proof} By standard interior regularity theory for elliptic equations $u$ belongs to $C^2(\Omega)$ hence, either $u \equiv 0$ in $\Omega$ or $u>0$ in $\Omega$ by the strong maximum principle.  Let us rule out the second case. Suppose $u>0$ in $\Omega$, then we can apply Theorem \ref{TH_fonc_G-bis} to prove that $\dfrac{\partial u}{\partial x_{N}}>0$ in $\Omega$. Hence, $u $ is a bounded stable solution to \eqref{equation_sans_neumann} of class $C^2(\Omega)$ and so, the function
	\begin{equation*}
		v(x')=\lim\limits_{x_{N} \to + \infty}u(x',x_{N}), \qquad \qquad x' \in \R^{N-1},
	\end{equation*}
is a positive stable classical solution to $-\Delta v=f(v)$ in $\R^{N-1}$, with $ N-1 \leq 10$. We can therefore apply  Theorem $1$ in \cite{DuFa1} to get that $v\equiv  const. =a>0$ et $f(a)=0$. Since the latter contradicts the positivity assumption on $f$, we infer that $u \equiv 0$ and so $ f(0)=0$.
 \end{proof}

\medskip 
 
The previous theorem remains true even for $N \geq 12$, if we add an assumption about the behaviour of $f$ at the origin. 
In this case, the desired conclusion is obtained by making use of some Liouville-type theorems for stable solutions established in \cite{Farina}, \cite{DuFa} and \cite{Fa3} (instead of Theorem $1$ in \cite{DuFa1}). \\
The desired results are the contents of Theorem \ref{th_n_plus que11} and Theorem \ref{th_n_plus que11-bis} below. 

\smallskip 

\begin{thm}\label{th_n_plus que11}
Assume $N \geq 12$ and let $\Omega $ be an epigraph defined by a function $g \in \mathcal{G}$.
Also suppose that $\Omega$ is bounded from below and satisfies a uniform exterior cone condition. \\
Let $u \in C^0(\overline{\Omega}) \cap H^1_{loc}(\overline{\Omega})$ be a bounded distributional solution to \eqref{equation_sans_neumann}
where $f \in C^{1}([0,+\infty))$ satisfies $f(t)>0$ for $t>0$ and $\liminf_{t \to 0^+} \frac{f(t)}{t^s} > 0,$
for some $ s \in \left[0, \frac{N-3}{N-5} \right) $. \\
Then $u \equiv 0$ and $f(0)=0$. 
\end{thm}  

\begin{proof} By proceeding as in the proof of Theorem \ref{th_n_entre_2_et_11}, if $u>0$ in $ \Omega$ we can construct $v \in C^2(\R^{N-1})$ which is a positive, bounded,  stable solution to $-\Delta v=f(v)$ in $\R^{N-1}$. Now, pick  $ z > \sup_{\R^{N-1}} v >0$ and consider the new nonlinear function $\tilde f \in C^{1}([0,+\infty))$ defined as follows :  $ \tilde f = f$ on $\left[0, \sup_{\R^{N-1}} v \right]$, $ \tilde f > 0$ on $ \left[ \sup_{\R^{N-1}} v, z \right)$, $\tilde f \leq 0$ on $\left[z, +\infty \right)$.\footnote{Note that, thanks to the assumptions on $f$ and $z$, such a function $\tilde f$ does exist.}  
Hence, $v$ is also a bounded stable solution to 
\begin{equation}
		\left\{
		\begin{array}{cll}
			- \Delta v=\tilde f (v) & \text{in}& \R^{N-1},\\
			v >0 & \text{on} & \R^{N-1}.\\
		\end{array}
		\right.
	\end{equation} 
Since $\liminf_{t \to 0^+} \frac{\tilde f(t)}{t^s} = \liminf_{t \to 0^+} \frac{f(t)}{t^s} > 0$ and $\tilde f > 0$ on 
$ \left ( 0, \sup_{\R^{N-1}} v \right)$, we see that 
$$
\exists \, \delta_1>0 \quad : \quad \forall \, x \in \R^{N-1} \qquad \tilde f(v(x)) \geq  \delta_1 v^s(x).
$$
Therefore we can apply Theorem $1.2$ of \cite{Fa3} to $v$ (with $M=N-1$) and find that $v \equiv 0$.
A contradiction. Hence $ u \equiv 0$ and $f(0)=0$. 
\end{proof}

\smallskip
 
Notice that  Theorem \ref{th_n_entre_2_et_11} and Theorem \ref{th_n_plus que11} immediately  imply the following 
non-existence result when $f(0)>0$.

\smallskip 

\begin{cor}\label{non-ex-f(0)>0}
Assume $N \geq 2$ and let $\Omega \subset \R^N$ be a uniformly continuous epigraph bounded from below and satisfying 
a uniform exterior cone condition. \\
If $f \in C^{1}([0,+\infty))$ satisfies $f(t)>0$ for $t\geq0$, then problem \eqref{equation_sans_neumann} does not admit any bounded distributional solution of class $C^0(\overline{\Omega}) \cap H^1_{loc}(\overline{\Omega})$. 
\end{cor}  

The next result concerns the following natural class of nonlinearities introduced in \cite{DuFa}: 
\begin{equation}\label{cvx-liouv}
		\left\{
		\begin{array}{cll}
f \in C^{1}([0,+\infty)) \cap C^{2}((0,+\infty)), \quad f(0) =0, \\ 
& \\
f>0,  \textit{  nondecreasing and convex in }  (0,+\infty)  \\
    & \\

\textit{ s.t.  }  
\quad \lim_{u \to 0^+} \frac{f'(u)^2}{f(u)f''(u)} := q_0 \in [0,+\infty]\\
		\end{array}
		\right.
	\end{equation}
where, in the latter, we agree to set $\frac{f'(u)^2}{f(u)f''(u)} = + \infty$ if $f''(u) =0$. \\
Notice that, we necessarily have $q_0 \in [1,+\infty]$ (see Lemma 1.4 in \cite{DuFa}). 

\smallskip

The typical representative of this class is given by the function $f(u) = u^p$, $p>1$. In this case $\frac{f'(u)^2}{f(u)f''(u)} = \frac{p}{p-1}$
and so $ q_0$ coincides with the conjugate exponent of $p$. Consequently, if we define $p_0 \in [1,+\infty]$ as the conjugate exponent of $q_0$ by $\frac{1}{p_0} + \frac{1}{q_0} =1$, we have that the exponent $p_0$ can be considered as a "measure" of the flatness of $f$ at the origin.  \\ 
Other members of the preceding class are provided by the functions $f_n(u) = e^u - \sum_{k=0}^n \frac{u^k}{k!}$,  where $n \geq 1$ is an integer. It is easily seen that $ q_0 (f_n) = \frac{n+1}{n}$ and so $ p_0(f_n) = n+1$. 

\medskip

\begin{thm}\label{th_n_plus que11-bis}
Assume $N \geq 12$ and let $\Omega $ be an epigraph defined by a function $g \in \mathcal{G}$.
Also suppose that $\Omega$ is bounded from below and satisfies a uniform exterior cone condition. \\
Let $u \in C^0(\overline{\Omega}) \cap H^1_{loc}(\overline{\Omega})$ be a bounded distributional solution to \eqref{equation_sans_neumann} where $f$  satisfies \eqref{cvx-liouv}. \\
Suppose that $p_0$, the conjugate exponent of $q_0$, satisfies 
\begin{equation}\label{JL-cvx-liouv}
1 \leq p_0 < p_c(N-1),
\end{equation}
where $p_c$ is the Jospeh-Lundgren stability exponent given by 
\begin{equation*}
p_c(N) = \frac{(N-2)^2 -4N + 8 \sqrt{N-1}}{(N-2)(N-10)} . 
\end{equation*}
Then $u \equiv 0$. 
\end{thm}  

Note that the preceding theorem applies to $f(u) =u^p$ if $ 1 \leq p < p_c(N-1)$ and to $ f_n$ if $n+1 < p_c(N-1)$. 

\smallskip

\textit{Proof of Theorem \ref{th_n_plus que11-bis}.} 
By proceeding as in the proof of Theorem \ref{th_n_entre_2_et_11}, if $u>0$ in $ \Omega$ we can construct $v \in C^2(\R^{N-1})$ which is a positive, bounded,  stable solution to $-\Delta v=f(v)$ in $\R^{N-1}$. Since the assumption \eqref{JL-cvx-liouv} is in force, we can  apply  Theorem $1.5$ in \cite{DuFa} to get that $v\equiv  const. =a>0$ et $f(a)=0$. A contradiction, hence $u \equiv 0$. \qed

\medskip

Next we state the following immediate consequence of Theorem \ref{TH2coerc}.

\begin{cor}\label{cor-TH2coerc}
When the epigraph $\Omega$ is coercive, the conclusion of Theorems \ref{th_n_entre_2_et_11}-\ref{th_n_plus que11-bis} 
and  Corollary \ref{non-ex-f(0)>0} holds true under the sole assumption of continuity of $g$, i.e., we do not need to require the uniform exterior cone condition for the epigraph. 
\end{cor}

Theorems \ref{th_n_entre_2_et_11}-\ref{th_n_plus que11-bis} and  Corollary \ref{cor-TH2coerc} recover and extend/complement some results in \cite{Farina}, \cite{DuFa1}, \cite{CLZ}, \cite{DFP} and \cite{GMS}.

\medskip 

We conclude this section with the following classification result for solutions of \eqref{equation_sans_neumann} that tend to zero at infinity.

\begin{thm}\label{th-tende-a-zero}
Assume $N \geq 2$ and let $\Omega $ be an epigraph defined by a function $g \in \mathcal{G}$.
Also suppose that $\Omega$ is bounded from below and satisfies a uniform exterior cone condition. \\
Assume $f \in {Lip}_{loc}([0,+\infty))$ and let $u \in C^0(\overline{\Omega}) \cap H^1_{loc}(\overline{\Omega})$ be a distributional solution to \eqref{equation_sans_neumann} such that
    \begin{equation}\label{tende-a-zero}
\lim_{\substack {x \in \Omega, \\ \vert x \vert \longrightarrow \infty}} u(x) =0 .
	\end{equation}
Then $u \equiv 0$ and $f(0)=0$. \\ 
When the epigraph $\Omega$ is coercive, the above conclusion holds true under the sole assumption of continuity of $g$.
\end{thm}  

\medskip

Note that, in the previous result, we make no assumptions on $f$ (beside $f \in {Lip}_{loc}([0,+\infty))$) or  the boundedness of $u$.

\smallskip

The above result recovers and improves upon a result of \cite{{Esteban-Lions}}, where the conclusion has been obtained  
for a smooth coercive epigraph. 

\smallskip

\textit{Proof of Theorem \ref{th-tende-a-zero}.} First we prove that $u$ is bounded and then we use this information to deduce that $f(0)=0$.  The boundedness of $u$ easily follows from the continuity of $u$  on $ \overline{\Omega}$ and \eqref{tende-a-zero}. Then, pick an open ball $B \subset\subset \Omega$ and, for any integer $k \geq 1$ and any $ x \in B$, consider the function $u_k(x) = u(x', x_N +k)$. By the boundedness of $u$ and standard elliptic estimates, we have that 
a subsequence of $(u_k)$ (still denoted by $(u_k)$) converges in the $C^2_{loc}(B)$-topology to a solution $v \in C^2(B)$  of $-\Delta v=f(v)$ in $B$. Since \eqref{tende-a-zero} is in force, we necessarily have $v\equiv0$ which, in turn, yields $ f(0)=0$. By the latter and the strong maximum principle,  either $u \equiv 0$ in $\Omega$ or $u>0$ in $\Omega$. Let us prove  that the second case cannot occur. Since $u$ is bounded and $f(0)=0$, we can apply Theorem \ref{TH_fonc_G-bis} (resp. Theorem \ref{TH2coerc}, when $\Omega$ is coercive)  to prove that $\dfrac{\partial u}{\partial x_{N}}>0$ in $\Omega$. But the latter contradicts the assumption \eqref{tende-a-zero}. Hence $ u \equiv 0$ and $f(0)=0$. \qed

\section{Some examples} \label{SectEXAMPLES}
  
\begin{Exam}\label{Ex1}   Let us consider the following two functions $g_1, \, g_2 : \R \mapsto \R$ defined by 

	\begin{multicols}{2}
		\begin{equation*}
		g_{1}(x)=\left\{
		\begin{array}{crl}
			0 & \text{if} & x \in  (-\infty,-4], \\
			\sqrt{4-(x+2)^{2}} & \text{if}&  x \in [-4,0],\\
			\sqrt{4-(x-2)^{2}} & \text{if}&  x \in [0,2],\\
			2 & \text{if} & x \in [2,+\infty),
		\end{array}
		\right.
	\end{equation*}
	    \newline
		\begin{figure}[H]
			\centering
		\includegraphics[width=5.5cm]{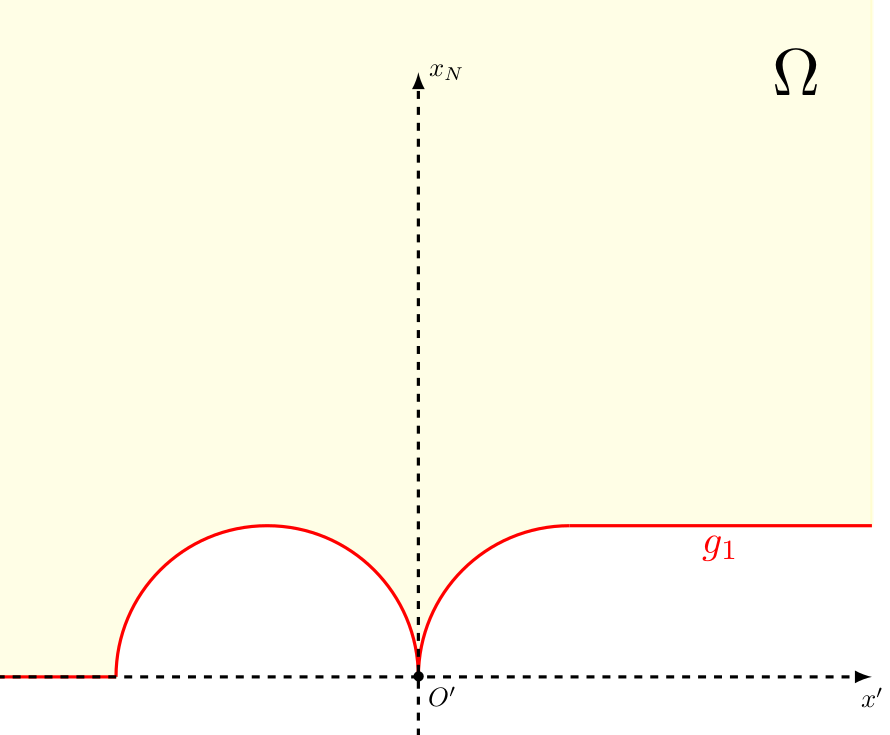}
		\caption{$g_{1}$}
		\label{fig:exemple1}
		\end{figure}
	\end{multicols}
\begin{multicols}{2}
	\begin{equation*}
		g_{2}(x)=\left\{
		\begin{array}{crl}
			0 & \text{if} & x \in (-\infty,-4], \\
			\sqrt{4-(x+2)^{2}} & \text{if}&  x \in [-4,0],\\
			\sqrt{4-(x-2)^{2}} & \text{if}&  x \in [0,2],\\
			2 & \text{if} & x \in [2,6], \\
			x-4 & \text{if} & x \in [6,+\infty). 
		\end{array}
		\right.
	\end{equation*}
	\newline
	\begin{figure}[H]
		\centering
	\includegraphics[width=6cm]{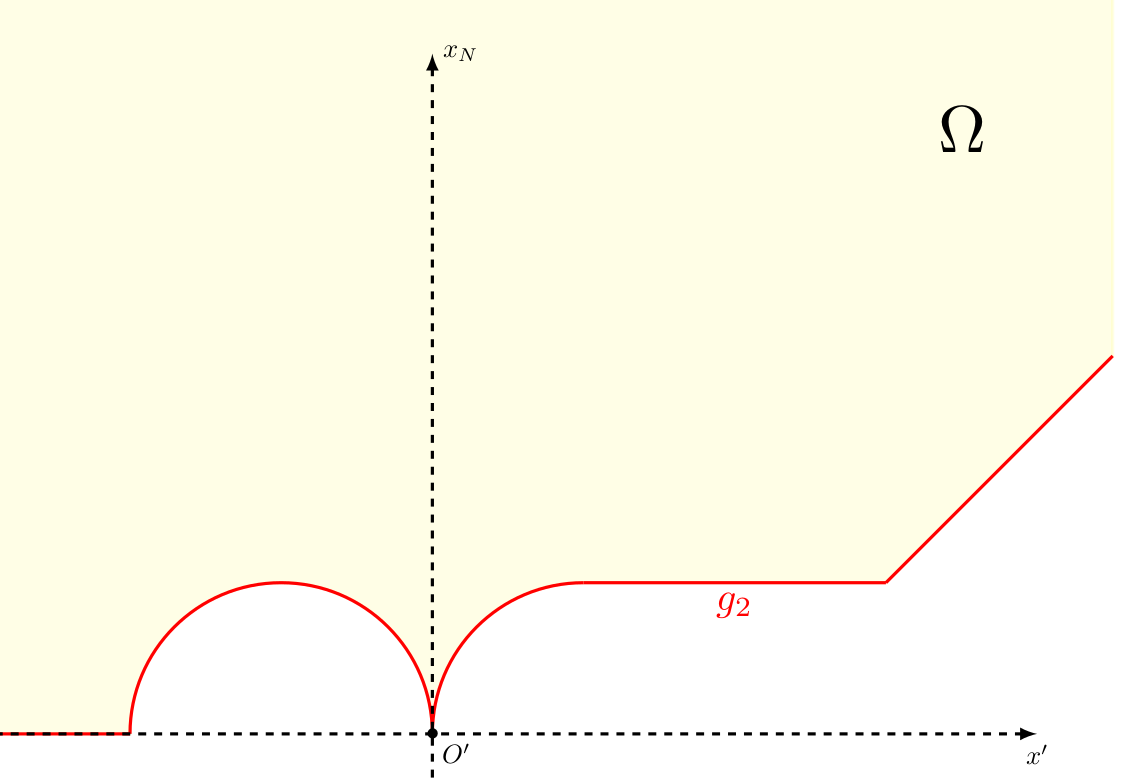}
		\caption{$g_{2}$}
		\label{fig:exemple2}
	\end{figure}
\end{multicols}
Note that $g_1$ is globally $\frac{1}{2}$-Hölder-continuous
and $g_2 $ is locally $\frac{1}{2}$-Hölder-continuous, but neither $g_1$ nor $g_2$ are locally $\alpha$-Hölder-continuous, for any $ \alpha \in ( \frac{1}{2} , 1 ] $.  Also observe that $ g_2(x) = g_1(x) + (x-6)^+$ for any 
$ x \in \R$, therefore the two-dimensional epigraphs defined by $g_1$ and $g_2$ are uniformly continuous,  bounded from below, but not locally Lipschitz-continuous. Moreover, it is easily seen that both of them satisfy a uniform exterior sphere condition (of radius $\frac{1}{2}$). Hence, they also satisfy a uniform exterior cone condition.  

\end{Exam}

\smallskip

\begin{Exam}\label{Ex3} Assume $N \geq 2$ and let $\Omega$ be the half-space $ \{x \in \R^N \, : \, x_N> 0 \}$.  \\
The bounded function 
\begin{equation}\label{ex-disc}
			u(x)=\left\{
			\begin{array}{crl}
      		1 -(x_N-1)^4 & \text{if}&  0 \leq x_N \leq 1,\\
        	    1 & \text{if} & x_N >1,
			\end{array}
				\right.
\end{equation}
is a classical $C^2$ solution to \eqref{NonLin-PoissonEq}, where $f$ is given by the following non-increasing function 
\begin{equation}\label{f-disc}
			f(t)=\left\{
			\begin{array}{crl}
			12 & \text{if} & t <0, \\
      		12 \sqrt{1-t} & \text{if}&  0 \leq t \leq 1,\\
        	    0 & \text{if} & t >1.
			\end{array}
				\right.
\end{equation}
Note that $f$ is globally Hölder-continuous, but not locally Lipschitz-continuous on $\R$.

\end{Exam}

The following example (inspired from \cite{Farinaexemple}) shows that we cannot remove the assumption on the Lipschitz character of $f$ in our main results. 

\begin{Exam}\label{Ex4} Assume $N \geq 2$ and let $\Omega$ be the half-space $ \{x \in \R^N \, : \, x_N> 0 \}$.  \\
	The bounded function (see Figure \ref{fig_exemple_3})	
	\begin{equation}\label{contre_exemple_monotonie}
		u(x)=\left\{
		\begin{array}{crl}
			0 & \text{if}&  0 \leq x_N \leq 1,\\
			(1-(x_N-2)^{4})^{4} & \text{if}& 1<x_N \leq 3,\\
			(1-(x_N-4)^{4})^{4} & \text{if}& 3< x_N \leq 4,\\
			1 & \text{if} & x_N >4,
		\end{array}
		\right.
	\end{equation}
	is a classical $C^2$ solution to \eqref{NonLin-PoissonEq}, where $f$ is given by the following globally Hölder-continuous function 
	\begin{equation}\label{f-contre_exemple_monotonie}
		f(t)=\left\{
		\begin{array}{crl}
			0 & \text{if}&   t< 0,\\
			-192(t(1-t^{\frac{1}{4}}))^{\frac{1}{2}}(1-\frac{5}{4}t^{\frac{1}{4}}) & \text{if}& 0 \leq t \leq 1,\\
			0 & \text{if} & t >1.
		\end{array}
		\right.¨
	\end{equation}
	Note that $f$ is not locally Lipschitz-continuous on $\R$ and that $\frac{\partial u}{\partial x_{N}}$ changes sign on $\R^{N}_{+}$.

\end{Exam}

\begin{figure}[!h]
	\centering
	\includegraphics[width=0.5\textwidth]{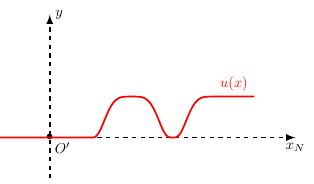}
	\caption{Graph of $u$}
	\label{fig_exemple_3}
\end{figure}

\bigskip

\section{Notations}\label{Notations} \quad 

\noindent  $\mathbb{R}^N_+ =\{x=(x', x_N) \in \R^{N-1} \times \R  \ | \ x_N>0\}$, the open upper half-space of $\R^N$.  

\bigskip

\noindent $ \vert \cdot \vert$ : the Euclidean norm. 

\bigskip

\noindent $B(x,R)$  : the Euclidean $N$-dimensional open ball of center $x$ and radius $R>0$. 

\bigskip

\noindent $B'(x',R)$  : the Euclidean $N-1$-dimensional open ball of center $x'$ and radius $R>0$. 

\bigskip

\noindent $B_R : = B(0 , R)$  and $B'_R := B'(0', R)$, where $0 =(0',0) \in \R^{N-1} \times \R$ is the origin of $ \R^N$.  

\bigskip

\noindent $UC(X)$ : the set of uniformly continuous functions on $X$. 

\bigskip

\noindent $Lip(X)$ : the set of globally Lipschitz-continuous functions on $X$. 

\bigskip

\noindent $Lip_{loc}(X)$ : the set of locally Lipschitz-continuous functions on $X$. 

\bigskip

\noindent $C^k(\overline{U})$ : the set of functions in $C^k(U)$ all of whose derivatives of order $ \leq k$ have continuous (not necessarily bounded) extensions to the closure of the open set $U$. 

\bigskip

\noindent $C^{0,\alpha}(\overline{U})$ : the vector space of bounded and globally $\alpha$-Hölder-continuous functions $h$ on the open set $U$ endowed with the norm :

\medskip

\qquad $\Vert h \Vert_{C^{0,\alpha}(\overline{U})} := \|h\|_{L^{\infty}(U)} + 
\left[ h \right]_{C^{0,\alpha}  (U)} :=  \sup_{x \in U} |h(x)| + \sup_{x,y \in U, x\neq y}
\frac{|h(x)-h(y)|}{|x-y|^{\alpha}} $ . 

\bigskip

\noindent $C^{k,\alpha}(\overline{U})$ : the vector space of functions in $C^k(U)$ all of whose derivatives of order $ \leq k$ belong to $C^{0,\alpha}(\overline{U})$, endowed with the norm :

\medskip

\hskip4truecm  $\Vert h \Vert_{C^{k,\alpha}(\overline{U})} := \sum_{0 \leq \vert \beta \vert \leq k} \Vert \partial^{\beta} h 
\Vert_{C^{0,\alpha}(\overline{U})}$ . 

\bigskip

\noindent $\mathcal{D}'(U)$ : the space of distributions on the open set $U$. 

\bigskip

\noindent $H^1_{loc}(\overline{U}) = \{ u : U \mapsto \R,  \,\,  \text{$u$ Lebesgue-mesurable} \, : \, u \in H^1(U \cap B(0,R)) \quad \forall \, R>0 \}$, \\
 i.e., $u$ is Lebesgue-measurable on the open set $U$ and $u \in H^1(V)$ for any open bounded set $ V\subset U$.
 
 \bigskip

\noindent $W^{1,\infty}_{loc}(\overline{U}) = \{ u : U \mapsto \R,  \,\,  \text{$u$ Lebesgue-mesurable} \, : \, u \in W^{1,\infty}(U \cap B(0,R)) \quad \forall \, R>0 \}$, \\
 i.e., $u$ is Lebesgue-measurable on the open set $U$ and $u \in W^{1,\infty}(V)$ for any open bounded set $ V\subset U$.
 
 \bigskip

\bibliography{plain.bst}

\end{document}